\documentclass[10pt, a4paper]{amsart}

\usepackage{amssymb,amsfonts}
\usepackage[all,arc]{xy}
\usepackage{enumerate}
\usepackage{mathrsfs}
\usepackage{tikz-cd}
\usetikzlibrary{decorations.pathmorphing}
\usepackage{enumitem}
\usepackage{empheq}
\usepackage{kotex}
\usepackage{ifpdf}
\usepackage{mathtools}
\usepackage{pst-node, pst-plot, pstricks}
\usepackage{mathrsfs}
\usepackage{amsmath}
\usepackage{epsfig}
\usepackage{amscd}
\usepackage{graphicx, color}
\graphicspath{ {./images/} }

\def\E{\ifmmode{\mathbb E}\else{$\mathbb E$}\fi} 
\def\N{\ifmmode{\mathbb N}\else{$\mathbb N$}\fi} 
\def\R{\ifmmode{\mathbb R}\else{$\mathbb R$}\fi} 
\def\Q{\ifmmode{\mathbb Q}\else{$\mathbb Q$}\fi} 
\def\C{\ifmmode{\mathbb C}\else{$\mathbb C$}\fi} 
\def\H{\ifmmode{\mathbb H}\else{$\mathbb H$}\fi} 
\def\Z{\ifmmode{\mathbb Z}\else{$\mathbb Z$}\fi} 
\def\P{\ifmmode{\mathbb P}\else{$\mathbb P$}\fi} 
\def\T{\ifmmode{\mathbb T}\else{$\mathbb T$}\fi} 
\def\SS{\ifmmode{\mathbb S}\else{$\mathbb S$}\fi} 
\def\DD{\ifmmode{\mathbb D}\else{$\mathbb D$}\fi} 

\DeclareSymbolFont{yhlargesymbols}{OMX}{yhex}{m}{n}

\DeclareMathAccent{\widetriangle}{\mathord}{yhlargesymbols}{"E6}

\def\E{\ifmmode{\mathbb E}\else{$\mathbb E$}\fi} 
\def\N{\ifmmode{\mathbb N}\else{$\mathbb N$}\fi} 
\def\R{\ifmmode{\mathbb R}\else{$\mathbb R$}\fi} 
\def\Q{\ifmmode{\mathbb Q}\else{$\mathbb Q$}\fi} 
\def\C{\ifmmode{\mathbb C}\else{$\mathbb C$}\fi} 
\def\H{\ifmmode{\mathbb H}\else{$\mathbb H$}\fi} 
\def\Z{\ifmmode{\mathbb Z}\else{$\mathbb Z$}\fi} 
\def\P{\ifmmode{\mathbb P}\else{$\mathbb P$}\fi} 
\def\T{\ifmmode{\mathbb T}\else{$\mathbb T$}\fi} 
\def\SS{\ifmmode{\mathbb S}\else{$\mathbb S$}\fi} 
\def\DD{\ifmmode{\mathbb D}\else{$\mathbb D$}\fi} 

\newcommand{\ben}{\begin{enumerate}}
\newcommand{\een}{\end{enumerate}}
\newcommand{\be}{\begin{equation}}
\newcommand{\ee}{\end{equation}}
\newcommand{\bea}{\begin{eqnarray}}
\newcommand{\eea}{\end{eqnarray}}
\newcommand{\beastar}{\begin{eqnarray*}}
\newcommand{\eeastar}{\end{eqnarray*}}
\newcommand{\bc}{\begin{center}}
\newcommand{\ec}{\end{center}}

\theoremstyle{theorem}
\newtheorem{thm}{Theorem}[section]
\newtheorem{cor}[thm]{Corollary}
\newtheorem{lem}[thm]{Lemma}
\newtheorem{prop}[thm]{Proposition}
\newtheorem{'thm'}[thm]{'Theorem'}

\theoremstyle{definition}
\newtheorem{defn}[thm]{Definition}
\newtheorem{rem}[thm]{Remark}

\newtheorem{exam}[thm]{Example}

\newtheorem{cond}[thm]{Condition}
\newtheorem{proof-sketch}[thm]{Proof-Sketch}

\newtheorem{claim}[thm]{Claim}

\newtheorem{lem-defn}[thm]{Lemma-Definition}
\newtheorem{prop-defn}[thm]{Proposition-Definition}
\newtheorem{thm-defn}[thm]{Theorem-Definition}

\newtheorem{notation}[thm]{\rm\bfseries{Notation}}

\newtheorem*{thm*}{Theorem}

\numberwithin{equation}{section}

\def\R{{\mathbb R}}

\def\E{{\mathbb E}}
\def\Z{{\mathbb Z}}
\def\C{{\mathbb C}}
\def\R{{\mathbb R}}
\def\P{{\mathbb P}}

\def\N{{\mathbb N}}

\def\11{{\mathbb I}}

\def\sgn{{\text{\rm sgn}}}

\def\C{\mathbb{C}}
\def\Z{\mathbb{Z}}

\def\T{\mathbb{T}}

\def\Q{\mathbb{Q}}

\def\E{\ifmmode{\mathbb E}\else{$\mathbb E$}\fi} 
\def\N{\ifmmode{\mathbb N}\else{$\mathbb N$}\fi} 
\def\R{\ifmmode{\mathbb R}\else{$\mathbb R$}\fi} 
\def\Q{\ifmmode{\mathbb Q}\else{$\mathbb Q$}\fi} 
\def\C{\ifmmode{\mathbb C}\else{$\mathbb C$}\fi} 
\def\H{\ifmmode{\mathbb H}\else{$\mathbb H$}\fi} 
\def\Z{\ifmmode{\mathbb Z}\else{$\mathbb Z$}\fi} 
\def\P{\ifmmode{\mathbb P}\else{$\mathbb P$}\fi} 
\def\SS{\ifmmode{\mathbb S}\else{$\mathbb S$}\fi} 
\def\DD{\ifmmode{\mathbb D}\else{$\mathbb D$}\fi} 

\def\R{{\mathbb R}}

\def\E{{\mathbb E}}
\def\Z{{\mathbb Z}}
\def\C{{\mathbb C}}
\def\R{{\mathbb R}}

\def\N{{\mathbb N}}

\def\CV{{\mathcal V}}

\def\darr#1{\raise1.5ex\hbox{$\leftrightarrow$}
\mkern-16.5mu #1}

\def\roughly#1{\raise.3ex\hbox{$#1$\kern-.75em
\lower1ex\hbox{$\sim$}}}

\def\opname#1{\mathop{\kern0pt{\rm #1}}\nolimits}

\def\dim{\opname{dim}}

\def\group#1{\opname{#1}}

\def\U{\group{U}}

\def\Incl{\operatorname{Incl}}
\def\Eval{\operatorname{Eval}}

\begin{document}

\quad \vskip1.375truein

\bibliographystyle{plain}

\title[The moduli of pseudoholomorphic disks]
{$L_\infty$-Kuranishi spaces and the moduli space of pseudoholomorphic disks}

\author{Taesu Kim}
\thanks{}


\begin{abstract}
We show that the moduli space of pseudoholomorphic disks is an example of the $L_{\infty}$-Kuranishi spaces introduced in \cite{Kim1}, provided that a condition for the existence of a stratification with a system of tubular neighborhoods holds on each chart. With respect to this structure, the forgetful and evaluation maps for the moduli space lift to morphisms between $L_{\infty}$-Kuranishi spaces.
\end{abstract}

\keywords{Kuranishi spaces, $L_{\infty}[1]$-algebras, Pseudoholomorphic disks, Lagrnagian submanifolds}
\subjclass[2020]{Primary 53D35; Secondary 18N70, 55U35}
\thanks{This work was carried out while the author was affiliated with the Department of Mathematics at Pohang University of Science and Technology (POSTECH) and was supported by the BK21 FOUR program funded by the Ministry of Education (MOE), Korea (No. 4120240414885)}

\maketitle

\date{}

\tableofcontents

\section{Introduction}

Let $(M, \omega)$ be a symplectic manifold and $L$ its compact Lagrangian submanifold. We take an almost complex structure $J$ on $M$ which is tamed by $\omega.$ We denote by $\mathcal{M}_{k+1}(L, \beta)$ the compactified moduli space of pseudoholomorphic disks in a given homology class $\beta \in H_2(X;L)$ with Lagnrangian boundary condition with $k+1$ boundary marked points. (For its precise definition, see Definitions \ref{tmsdef} and \ref{relms}.)

Fukaya-Oh-Ohta-Ono \cite{FOOO1}, \cite{FOOO2} showed that (the compactification of) the moduli spaces $\mathcal{M}_{k+1}(\beta, L)$ carries a Kuranishi structure (with corners), which was introduced in \cite{FO} and \cite{FOOO1}. It now serves as the foundation for Gromov-Witten invariant and Fukaya category. The theory of Kuranishi spaces, however, has some limitations. Most notably, there is no natural notion of \textit{morphisms} between Kuranishi spaces. This is partly due to the rigid definition of coordinate changes, which makes it difficult to ensure compatibilities among chart level maps.

In this paper, we apply the notion $L_{\infty}$-Kuranishi spaces of \cite{Kim1} to the study of the moduli space $\mathcal{M}_{k+1}(\beta, L)$ and bring it into a categorical framework. For this purpose, we adopt the \cite{FOOO3}'s settings throughout, while modifying them properly to derive an $L_{\infty}$-version of their theory. The main result is as follows.

\begin{thm}\label{mk1lk1}
Under Condition \ref{condusi1} (explained below) on the bases of Kuranishi charts, the moduli space $\mathcal{M}_{k+1}(\beta, L)$ determines an $L_{\infty}$-Kuranishi space.
\end{thm}

This categorical perspective proves sufficiently natural to have both forgetful and evaluation maps for the moduli space as morphisms. In other words, the forgetful and evaluation maps in \cite{FOOO1} have their $L_{\infty}$-analogues:
\begin{thm}
With respect to the $L_{\infty}$-Kuranishi space structure in Theorem \ref{mk1lk1}, there exist morphisms of $L_{\infty}$-Kuranishi spaces
\[
\begin{cases}
\mathrm{Ft}_i : \mathcal{M}_{k+1}(\beta, L) \rightarrow \mathcal{M}_{k}(\beta, L),\\
\mathrm{Ev}_i : \mathcal{M}_{k+1}(\beta, L) \rightarrow L, \ i= 0, \cdots, k,
\end{cases}
\]
whose underlying topological maps are the ordinary forgetful and evaluation maps.
\end{thm}
The main ingredient of its proof is to show that $\mathrm{Ft}_i$'s and $\mathrm{Ev}_i$'s are compatible with the $L_\infty$-version coordinate changes. Flexible structures that originates from the $L_{\infty}[1]$-algebras on charts allow these conditions to be satisfied easily compared to the approach of \cite{FOOO1}.

\ \
For this purpose, we consider an $L_{\infty}$-Kuranishi chart for each point $\mathbf{p} \in \mathcal{M}_{k+1}(L, \beta)$: 
\begin{equation}\label{kchrt}
\mathcal{U}_{\mathbf{p}}= \left(U_{\mathbf{p}}, {E}_{\mathbf{p}}, s_{\mathbf{p}}, \Gamma_{\mathbf{p}}, \psi_{\mathbf{p}}\right),
\end{equation}
where $U_{\mathbf{p}} = (U_{\mathbf{p}}, \omega_{\mathbf{p}})$ is a smooth manifold (possibly with boundary and) with a two-form $\omega_{\mathbf{p}} \in \Omega^2(U_{\mathbf{p}})$ given by (\ref{opy}). $E_{\mathbf{p}} \rightarrow U_{\mathbf{p}}$ is a vector bundle with a distinguished smooth section $s_{\mathbf{p}}.$ Let $\Gamma_{\mathbf{p}}$ be a finite group acting on $U_{\mathbf{p}}$ that restricts to $ s_{\mathbf{p}}^{-1}(0),$ and $\psi_{\mathbf{p}} : s_{\mathbf{p}}^{-1}(0)/\Gamma_{\mathbf{p}} \hookrightarrow \mathcal{M}_{k+1}(\beta, L)$ a homeomorphism onto a neighborhood of $\mathbf{p}$.

Here, the base manifold $U_{\mathbf{p}}$ is equipped with a closed two-form $\omega_{\mathbf{p}}$ whose value at $\mathbf{y} = \big((\Sigma_{\mathbf{y}}, \vec{z}_{\mathbf{y}}), u_{\mathbf{y}}\big) \in U_{\mathbf{p}}$ is given by
\begin{equation}\label{opy}
\omega_{\mathbf{p}, \mathbf{y}}(X_{\mathbf{y}}, Y_{\mathbf{y}}) := \int\limits_{\Sigma} \omega_{\mathbf{p}}(X_{\mathbf{y}}, Y_{\mathbf{y}}) \mathrm{dvol}_{\Sigma}
\end{equation}
for a volume form $\mathrm{dvol}_{\Sigma}$ on $\Sigma.$ We have:

\begin{thm}
$\omega_{\mathbf{p}}= \{\omega_{\mathbf{p,y}}\}_{\mathbf{y}}$ is closed.
\end{thm}

For the aforementioned statements to be meaningful, we need to assume an essential condition on the two-form $\omega_{\mathbf{p}}$ for each $\mathbf{p}$.

\begin{cond}\label{condusi1}
We assume that the closed two-form $\omega_{\mathbf{p}}$ on $U_{\mathbf{p}}$ allows the stratification 
\begin{equation}\label{dfaddsfa}
\U_{\mathbf{p}} = \bigcup\limits_i \mathcal{S}_{\mathbf{p},i},
\end{equation}
into \textit{submanifolds} $ \mathcal{S}_{\mathbf{p},i} := \{y \in U_{\mathbf{p}} \mid \mathrm{rk}\ker\left(\omega_{\mathbf{p}, \mathbf{y}}\right) = i\}$ together with their tubular neighborhoods:
\begin{equation}
\begin{cases}
\iota_{i} : N_i \rightarrow  U_{\mathbf{p}}, \text{ an open neighborhood of each (possibly non-connected) }  \mathcal{S}_{\mathbf{p},i},\\
\pi_{i} : N_i \rightarrow \mathcal{S}_{\mathbf{p},i}, \text{ the associated projection}.
\end{cases}
\end{equation}
\end{cond}

By virtue of this condition, it is possible to obtain a presymplectic manifold which we call a \textit{local presymplectic neighborhood} of ${\mathbf{x}}$
\[
W_{\mathbf{x}} = (W_{\mathbf{x}}, \omega_{\mathbf{p},W_{\mathbf{x}}}) := \left(\pi_{i}^{-1}(\overset{\circ}{W_{\mathbf{x}}}), \pi^*_{i}(\iota_{i}^*\omega_{\mathbf{p}})\right)
\] 
for each interior zero point ${\mathbf{x}} \in s_{\mathbf{p}}^{-1}(0)  \cap \text{int}\left(U_{\mathbf{p}}\right) \cap \mathcal{S}_{\mathbf{p},i}$ and a choice of its open neighborhood $\overset{\circ}{W_{\mathbf{x}}} \subset \mathcal{S}_{\mathbf{p},i}.$ (For boundary zero points, see (\ref{adslfjkdh}) for the definition of $W_{\mathbf{x}}.$)

The Condition \ref{condusi1} also leads to \textit{the local $L_{\infty}[1]$-algebra} on each $W_{\mathbf{x}},$ 
\[
\mathcal{C}_{\mathbf{p},\mathbf{x}} := \bigwedge{}^{-\bullet}\Gamma(E^*_{\mathbf{p}}|_{W_{\mathbf{x}}}) \oplus \Omega_{\mathrm{aug}}^{\bullet +1}(\mathcal{F}_{\mathbf{x}}).
\]
Here $\bigwedge{}^{-\bullet}\Gamma(E^*_{\mathbf{p}}|_{W_{\mathbf{x}}})$ is given by the Koszul complex of the $C^{\infty}(W_{\mathbf{x}})$-module $\Gamma(E|_{W_{\mathbf{x}}})$ (considered as an $\mathbb{R}$-module) with vanishing higher-degree operations, while $\Omega^{\bullet+1}_{\mathrm{aug}}(\mathcal{F}_{\mathbf{x}})$ is the foliation de Rham complex with augmentation, obtained from the regular foliation (i.e., one in which each leaf has the same dimension) $T\mathcal{F}_{\mathbf{x}} \subset TW_{\mathbf{x}}$ given by the kernel of the presymplectic form $\pi^*_{i}(\iota_{i}^*\omega_{\mathbf{p}})$. We can equip $\Omega^{\bullet+1}_{\mathrm{aug}}(\mathcal{F}_{\mathbf{x}})$ with the $L_{\infty}[1]$-algebra structure following \cite{OP} and consider its extension to the augmentation.

We close the introduction by making a remark on Condition \ref{condusi1}. In fact, \cite[Corollary 6.2]{KO} proves that there exists a residual subset of such closed two-forms. We conjecture that the same can be achieved by a generic choice of almost complex structure $J$. For a more general statement, we may be able to eliminate this condition either by proving this conjecture or by developing a method that accomplishes the same task using a general foliation that is not necessarily regular (possibly via some derived geometry), which we intend to explore in future work.

We outline the contents of this paper. In Section 2, we briefly review $L_{\infty}$-Kuranishi spaces from \cite{Kim1}. Section 3 recalls the notion of the moduli space of pseudoholomorphic disks, following \cite{FOOO1} and \cite{FOOO2}, and equips its chart with the structure of an $L_{\infty}$-Kuranishi chart. In Section 4, we construct coordinate changes between two charts to obtain Kuranishi atlases and spaces. Section 5 concerns examples of morphisms of the moduli space: the evaluation and forgetful morphisms. We show that these naturally extend their classical analogues. For the sake of completeness, the appendices revisit the notion of $L_{\infty}[1]$-algebras and their higher homotopy theory, together with a more detailed description of the structures placed on neighborhoods of the zero locus. Moreover, we mention statements on $L_{\infty}[1]$-algebras arising from $V$-algebras. Finally, we explain the rather technical definitions of equivalence relations of charts and pre-morphisms, which are needed in the definition of $L_{\infty}$-Kuranishi spaces.

\subsection*{Acknowledgement}
We thank Yong-Geun Oh and Kaoru Ono for helpful comments. This work was supported by the BK21 FOUR program funded by the Ministry of Education (MOE), Korea (No. 4120240414885).

\section{$L_{\infty}$-Kuranishi spaces and their category}

In this section, we recall the definition of $L_{\infty}$-Kuranishi spaces in \cite{Kim1} and show that their collection form a category where the category of smooth manifolds embeds.

\subsection{$L_{\infty}$-Kuranishi charts}

We define Kuranishi charts by associating $L_{\infty}[1]$-algebras to presymplectic neighborhoods of the zero locus of the Kuranishi section.

\begin{defn}[$L_\infty$-Kuranishi charts]\label{kurdef}
Let $X$ be a compact metrizable space. An $L_\infty$-\textit{Kuranishi chart} of $X$ is given by a tuple
\begin{equation}\nonumber
\mathcal{U} = (U, E, s, \Gamma, \psi),
\end{equation}
where
\begin{enumerate}
\item[--] $U = (U, \alpha)$ is a pair of a smooth manifold (possibly with boundary and) with a closed two-form $\alpha \in Z^2(U).$
\item[--] $\pi : E \rightarrow U$ is a (finite rank) vector bundle.
\item[--] $s : U \rightarrow E$ is a smooth section.
\item[--] $\Gamma$ is a finite group acting on $U$ that restricts to the zero set of $s$, that is, $\Gamma \cdot s^{-1}(0) \subset s^{-1}(0).$
\item[--] $\psi : {s^{-1}(0)}/{\Gamma} \overset{\simeq}{\hookrightarrow} X$ is a homeomorphism onto the image.
\end{enumerate}
The \textit{dimension} of $\mathcal{U}$ is defined by $\dim \mathcal{U} := \dim U - \text{rk}E.$

We require that the chart  $\mathcal{U}$ be endowed with the following structures:
\begin{itemize}
\item[--] $U$ has a decomposition 
\begin{equation}\label{wstri}
U = \bigcup\limits_i \mathcal{S}_i,
\end{equation}
into (possibly non-connected) submanifolds,
\[
\mathcal{S}_i := \{x \in U \mid \text{rk} (\ker\alpha_x) =i \}, \ 0 \leq i \leq \dim U,
\]
together with their tubular neighborhoods for each $i$:
\[
\begin{cases}
\iota_{i} : N_i \rightarrow  U, \text{ an open neighborhood of each }  \mathcal{S}_{i},\\
\pi_{i} : N_i \rightarrow \mathcal{S}_{i}, \text{ the associated projection}.
\end{cases}
\]

In case we have $\partial U \neq \emptyset$, we further require the decomposition (\ref{wstri}) restricts to $\partial U$, that is, we have a decomposition
\begin{equation}\label{adffff}
\partial U = \bigcup_i \mathcal{S}_i \cap \partial U,
\end{equation}
where $\mathcal{S}_i \cap \partial U$ is a submanifold of $\partial U$ given by
\[
\mathcal{S}_i \cap \partial U = \{x \in \partial U \mid \text{rk} (\ker\alpha_x) = i - 1 \}, \ 1 \leq i \leq \dim U,
\]
together with the corresponding tubular neighborhood in $\partial U$ for each $i$:
\[
\begin{cases}
\iota^{\partial}_{i} : N^{\partial}_i \rightarrow  \partial U, \text{ an open neighborhood of each }  \mathcal{S}_{i},\\
\pi^{\partial}_{i} : N^{\partial}_i \rightarrow \mathcal{S}_{i} \cap \partial U, \text{ the associated projection}.
\end{cases}
\]

\item[--] To each zero point $x \in s^{-1}(0),$ we assign:
\begin{enumerate}[label = (\roman*)]
\item A presymplectic open neighborhood $W_x$ of $x$ in $U$ with $W_x \simeq B^n,$
\item A local $ L_{\infty}[1]$-algebra $\mathcal{C}_x$
\[
\mathcal{C}_x := \overbrace{\bigwedge\nolimits^{-\bullet}\Gamma(E^*|_{W_x})}^{\text{Koszul}} \oplus \overbrace{\Omega^{\bullet + 1}_{\mathrm{aug}}(\mathcal{F}_x)}^{\text{de Rham}},
\]
\end{enumerate}
whose detailed descriptions are provided in Appendix A.
\end{itemize}
\end{defn}

\begin{rem}
\cite[Corollary 6.2]{KO} proves that there exists a residual subset of such closed two-forms. Furthermore, the stratification structure turns out to be Whitney (cf. \cite[Definition 6.5 \& Theorem 6.6]{KO}), which leads to the construction of \textit{stratified $L_{\infty}$-spaces.}
\end{rem}

\subsection{Morphisms of charts and Kuranishi atlases}

We define morphisms of Kuranishi charts, with embeddings as a special case, which yields the construction of a global object called a Kuranishi atlas.

Let $f : X \rightarrow X'$ be a continuous map between compact topological spaces. 

\begin{defn}[Chart morphisms]
A \textit{chart morphism} between $L_{\infty}$-Kuranishi charts ${\mathcal{U}} = (U,E,\Gamma,s,\psi)$ and ${\mathcal{U}}' = (U',E',\Gamma',s',\psi')$ of $X$ and $X',$ respectively, is defined by a pair
\[
\Phi = (\phi, \widehat{\phi}) : {\mathcal{U}} \rightarrow {\mathcal{U}}' ,
\]
given by:
\begin{enumerate}
\item[--] $\phi : U \rightarrow U', \text{ a } (\Gamma, \Gamma')\text{-equivariant map of manifolds,}$
\item[--] $\widehat{\phi} = \left\{\widehat{\phi}_x :  \mathcal{C}'_{\phi(x)} \rightarrow \mathcal{C}_x \right\}_{x \in s^{-1}(0)}$ is a family of $L_{\infty}[1]$-morphisms,
\end{enumerate}
satisfying 
\begin{enumerate}[label = (\roman*)]
\item $\psi' \circ \phi = f \circ \psi$ on $s^{-1}(0),$
\item $\phi(W_x) \subset W'_{\phi(W_x)},$
\item $\widehat{\phi}_x$ factors through $\mathcal{C}'_{\phi(x), \phi},$ that is, we have $\widehat{\phi}_x = \widehat{\phi}_x^{\mathrm{c}} \circ \varepsilon_{\phi(x), \phi}$ for some $L_{\infty}[1]$-morphism, $\widehat{\phi}^{\mathrm{c}}_x :  \mathcal{C}'_{\phi(x), \phi} \rightarrow \mathcal{C}_x.$
\end{enumerate}
Here, $\mathcal{C}_{\phi(x),\phi}'$ stands for the \textit{completion} of $\mathcal{C}'_{\phi(x)}$ at the image of $\phi,$ induced from the \text{completed} V-algebra at the image. (See Definition \ref{ladef} for the detailed definition.) On the other hand, $\widehat{\varepsilon}_{\phi(x), \phi} : \mathcal{C}'_{\phi(x)} \rightarrow \mathcal{C}'_{\phi(x), \phi}$ is the $L_{\infty}[1]$-morphism given by considering elements in $\mathcal{C}'_{\phi(x)}$ naturally as trivial, which can be easily shown to define an $L_{\infty}[1]$-morphism. (See Definition \ref{varep} and Lemma \ref{vecpf}.)
\end{defn}

\begin{defn}[Embedding of charts]\label{ourcemb}
A chart morphism $(\phi, \widehat{\phi})$ is called an \textit{embedding} if $\phi$ is an (equivariant) embedding of manifolds and $\widehat{\phi} = \left\{\widehat{\phi}_x\right\}_{x \in s^{-1}(0)}$ consists of quasi-isomorphic $\widehat{\phi}_x$'s.
\end{defn}

A \textit{coordinate change} of $L_{\infty}$-Kuranishi charts provides a main example of embedding:

\begin{defn}[Coordinate changes]\label{kstr}
For two points $p, q \in X$ with $\mathrm{Im}\psi_p \cap \mathrm{Im}\psi_q \neq \emptyset,$ we define a \textit{coordinate change} $\Phi_{pq} : \mathcal{U}_p \rightarrow \mathcal{U}_q$ by a tuple
\[
\Phi_{pq} := \left(U_{pq}, \phi_{pq}, \left\{\widehat{\phi}_{pq,x}\right\}\right),
\]
where $U_{pq} \subset U_p$ is an open submanifold, and
\[
\left(\phi_{pq}, \widehat{\phi}_{pq}\right) : \mathcal{U}_p|_{U_{pq}} \rightarrow \mathcal{U}_q
\]
is an embedding of $L_{\infty}$-Kuranishi charts from $\mathcal{U}_p|_{U_{pq}},$ that is, the chart restricted to $U_{pq}.$ They are required to satisfy:

\begin{enumerate}[label = (\roman*)]
\item $\Phi_{pp} = \mathrm{id}_{\mathcal{U}_p},$
\item $\psi_q \circ \phi_{pq} = \psi_p$ on $s_{p}^{-1}(0) \cap U_{pq},$
\item $\phi_{qr} \circ \phi_{pq} = \phi_{pr}$ on $s_p^{-1}(0) \cap \phi_{qr}^{-1}(U_{pq}) \cap U_{pr},$
\item $\psi_p\big(s_p^{-1}(0) \cap U_{pq}\big) =\mathrm{Im} \psi_p \cap\mathrm{Im} \psi_q,$
\end{enumerate}
\end{defn}

\begin{defn}
A pair of the compact topological space $X$ and a collection of Kuranishi charts with coordinate changes
\[
(X, \widehat{\mathcal{U}}),
\]
where $ \widehat{\mathcal{U}}=\left(\{\widehat{\mathcal{U}}_p\}, \left\{\Phi_{pq}\right\}\right),$ is called an $L_{\infty}$\textit{-Kuranishi atlas}. For technical reasons, we assume that $\max\limits_{p \in X} \dim U_p < \infty$ with the compactness of $X$.
\end{defn}

\subsection{Definition of $L_{\infty}$-Kuranishi spaces and their categorical structures}

By considering equivalence relations on atlases, we define $L_{\infty}$-Kuranishi spaces. We also define morphisms between them, obtaining the category of $L_{\infty}$-Kuranishi spaces.

Two atlases are said to be equivalent and are denoted by $(X, \widehat{\mathcal{U}}) \sim (X, \widehat{\mathcal{U}}'),$ or simply $\widehat{\mathcal{U}} \sim \widehat{\mathcal{U}}'$ if 
\begin{equation}\label{equu}
\widehat{\mathcal{U}}^0 \times V = \widehat{\mathcal{U}}^{'0} \times V',
\end{equation}
whose precise meaning is provided in Definition \ref{eqats}, for some finite dimensional vector spaces $V, V'$ and for the restrictions to some open subsets $\widehat{\mathcal{U}}^0 = \widehat{\mathcal{U}}|_{U^0 \subset U}$ and $\widehat{\mathcal{U}}^{'0} = \widehat{\mathcal{U}}'|_{U^{'0} \subset U'}.$ 

\begin{defn}[$L_{\infty}$-Kuranishi spaces]
We define an $L_{\infty}$-\textit{Kuranishi space} to be an equivalence class with respect to the relation $\sim$
\[
\mathfrak{X} := (X, [\widehat{\mathcal{U}}]).
\]
\end{defn}

Let $\mathfrak{X} = \left(X, [\widehat{\underline{\mathcal{U}}}]\right)$ and $\mathfrak{Y} = \left(Y,[\underline{\widehat{\mathcal{U}}}']\right)$ be $L_{\infty}$-Kuranishi spaces.
We consider two atlases $\left(X, \widehat{\mathcal{U}}\right)$ and $\left(X', \widehat{\mathcal{U}}'\right)$ with $\widehat{\mathcal{U}} = \left(\{\widehat{\mathcal{U}}_p\}, \left\{\Phi_{pq}\right\}\right) $ and $\widehat{\mathcal{U}}' = \left(\{\widehat{\mathcal{U}}'_{p'}\}, \left\{\Phi'_{p'q'}\right\}\right),$ satisfying $[\widehat{\mathcal{U}}] = [\underline{\widehat{\mathcal{U}}}]$ and $[\widehat{\mathcal{U}}'] = [\underline{\widehat{\mathcal{U}}}'].$

\begin{defn}[Pre-morphisms]\label{morphismkur}
A \textit{pre-morphism} is defined by the following tuple
\[
\overline{F} := \left(\widehat{\mathcal{U}}, \widehat{\mathcal{U}}', f, \{f_p\}, \left\{\widehat{f}_{p,x}\right\}\right).
\]
$f : X \rightarrow Y$ is a continuous map between the zero loci, while $\left\{(f_p, \{\widehat{f}_{p,x}\})\right\}$ is a collection of chart morphisms. Then $\overline{F}$ is required to satisfy the following compatibilities with respect to the coordinate change $\Phi_{pq}=\left(\phi_{pq}, \left\{\widehat{\phi}_{pq,x}\right\}\right)$:
For $p,q \in X$ with Im$\psi_p \cap\mathrm{Im} \psi_q \neq \emptyset,$ we require
\begin{enumerate}[label = (\roman*)]
\item $\phi'_{f(p)f(q)} \circ f_p = f_q \circ \phi_{pq}$ on $s_{p}^{-1}(0) \cap U_{pq},$
\item $\widehat{\phi}_{pq,x} \circ \widehat{f}_{q, \phi_{pq}(x)} = \widehat{f}_{p,x} \circ \widehat{\phi}'_{f(p)f(q),f_p(x)}$ for each $x \in s_p^{-1}(0) \cap U_{pq}$ up to $L_{\infty}[1]$-homotopy.
\end{enumerate}
\end{defn}

We can define an equivalence relation $\sim$ on the set of pre-morphisms as in Definition \ref{defmoreq}.

\begin{defn}[Morphism of Kuranishi spaces]
We define a \textit{morphism} from $\mathfrak{X} = \left(X, [\widehat{\mathcal{U}}]\right)$ to $\mathfrak{X}' = \left(X', [\widehat{\mathcal{U}}']\right)$
by an equivalence class (with respect to $\sim$) of a pre-morphism $\overline{F}$ from $\mathfrak{X}$ to $\mathfrak{X}':$
\[F := [\overline{F}] : \mathfrak{X} \rightarrow \mathfrak{X}'.\]
\end{defn}

\begin{defn}[Composition of morphisms]
Let $\mathfrak{X} = (X, [\widehat{\mathcal{U}}]),$ $\mathfrak{X}' = (X', [\widehat{\mathcal{U}'}]),$ and $\mathfrak{X}'' = (X'', [\widehat{\mathcal{U}''}])$ be Kuranishi spaces. Let $F : \mathfrak{X} \rightarrow \mathfrak{X}'$ and $G : \mathfrak{X}' \rightarrow \mathfrak{X}''$ be morphisms between them represented by 
\begin{equation}\nonumber
\begin{cases}
\overline{F} = \left(\widehat{\mathcal{U}}, \widehat{\mathcal{U}'}, f, \left\{f_{p}\right\}, \left\{\widehat{f}_{p,x}\right\}\right),\\
\overline{G} = \left(\underline{\widehat{\mathcal{U}'}}, \widehat{\mathcal{U}''}, g, \left\{g_{f(p)}\right\}, \left\{\widehat{g}_{f(p),y}\right\}\right),
\end{cases}
\end{equation}
respectively with $[ \widehat{\mathcal{U}'}] = [\underline{\widehat{\mathcal{U}'}}]$.

There exist \textit{extended} pre-morphisms
\begin{equation}\label{extfg1}
\begin{cases}
\overline{F}^{n_d, n'_t} = \left(\widehat{\mathcal{U}}^0 \times \mathbb{R}^{n_d}, \widehat{\mathcal{U}}^{'0} \times \mathbb{R}^{n'_t}, f, \left\{\widetilde{f}_{p} \right\}, \left\{\widetilde{\widehat{f}}_{p,x}\right\}\right),\\
\overline{G}^{\underline{n}_d, \underline{n}'_t} = \left(\underline{\widehat{\mathcal{U}}}^{'0} \times \mathbb{R}^{\underline{n}_d}, \widehat{\mathcal{U}''} \times \mathbb{R}^{\underline{n}_t'}, g, \left\{\widetilde{g}_{p'}\right\}, \left\{\widetilde{\widehat{g}}_{p',x'}\right\}\right),
\end{cases}
\end{equation}
of $\overline{F}$ and $\overline{G},$ respectively, so that $\widehat{\mathcal{U}}^{'0} \times \mathbb{R}^{n'_t} = \underline{\widehat{\mathcal{U}}}^{'0} \times \mathbb{R}^{\underline{n}_d}$ (cf. Definition \ref{eqats}) holds for some open subatlases $\widehat{\mathcal{U}}^{'0} < \widehat{\mathcal{U}}^{'},$ and $\underline{\widehat{\mathcal{U}}}^{'0} < \underline{\widehat{\mathcal{U}}}^{'},$ and that all the base maps $\widetilde{f}_p$ and $\widetilde{g}_{p'}$ are surjective.

Note that we can always assume that ${\widehat{\mathcal{U}}}^{'} = \underline{\widehat{\mathcal{U}}}^{'}$; if not, we can take the extension as (\ref{extfg1}). We then define the \textit{composition} $G \circ F$ to be the following equivalence class:
\begin{equation}\label{morcom}
G \circ F := \left[\left(\widehat{\mathcal{U}}^0 \times \mathbb{R}^{n_d}, \widehat{\mathcal{U}}'' \times \mathbb{R}^{\underline{n}'_t}, g\circ f, \left\{{g}_{f(p)} \circ {f}_{p}\right\}, \left\{{\widehat{f}}_{p,x} \circ {\widehat{g}}_{f(p),f_p(x)}\right\} \right)\right].
\end{equation}
\end{defn}

\begin{prop}\cite[Proposition 3.16]{Kim1}\label{ppidx}
The composition is well-defined and associative with the identity given by
\begin{equation}\label{idxm}
\mathrm{id}_{\mathfrak{X}}:= \left[\left(\widehat{\mathcal{U}}, \widehat{\mathcal{U}}, \mathrm{id}_X, \left\{\mathrm{id}_p\right\}, \left\{\widehat{\mathrm{id}}_{p,x}\right\}\right)\right]
\end{equation}
of each $\mathfrak{X} = \left(X, [\widehat{\mathcal{U}}]\right).$
\end{prop}

The above data give rise to a category denoted by $\mathbf{Kur}$ that consists of:
\begin{equation}\nonumber
\begin{cases}
\text{Ob}(\textbf{Kur}) =\{L_{\infty}\text{-Kuranishi spaces}\}\\
\text{Mor}(\textbf{Kur}) = \{\text{Equivalence classes of pre-morphisms} \}.
\end{cases}
\end{equation}

\begin{exam}[{Smooth manifolds}]\label{man}
Smooth manifolds are Kuranishi spaces endowed with a Kuranishi atlas 
\[
\widehat{\mathcal{U}}^{\mathrm{man}} = \left(\left\{\mathcal{U}^{\mathrm{man}}_p\right\}, \left\{\Phi_{pq}\right\}\right) = \left(\left\{\left(U_p, E_p, s_p, \Gamma_p, \psi_p\right)\right\}, \left\{\left(U_{pq}, \phi_{pq}, \left\{\widehat{\phi}_{pq,x}\right\}\right)\right\}\right)
\] 
of the following restrictive type:
\begin{itemize}
\item[--] $U_p = (U_p, \alpha)$ is the pair of a Euclidean space $\mathbb{R}^n$ of fixed dimension $n$ for all $p$ and the zero form $\alpha \equiv 0.$
\item[--] The decompositions (\ref{wstri}) and (\ref{adffff}) exist for trivial reasons.
\item[--] $E_p = U_p \times \{0\} \simeq U_p$ is the zero-rank vector bundle.
\item[--] $s_p : U_p \xrightarrow{\simeq} E_p$ is the zero section.
\item[--] $\Gamma_p$ is the trivial group.
\item[--] $\psi_p : s^{-1}_p(0) \simeq U_p \hookrightarrow \mathbb{R}^n$ is the manifold coordinate chart.
\item[--] $x \in W_x \subset U_p$ is an open ball $\simeq B^n.$
\item[--] $T\mathcal{F}_x = TU_p|_{W_x}$ is the total tangent bundle,
\item[--] $\mathcal{C}_{p,x} := \Omega_{\mathrm{aug}}^{\bullet +1}(W_x)$ is the augmented de Rham complex with the $L_{\infty}[1]$-algebra $\{l^{\mathrm{man}}_k\}_{k \geq 1}$ with $l^{\mathrm{man}}_{k \geq 2} = 0$. In other words, $\mathcal{C}_{p,x}$ is only a chain complex.

Let $\mathcal{U}_p$ and $\mathcal{U}_q$ be Kuranishi charts at $p$ and $q,$ respectively. The coordinate change $\Phi_{pq} := \left(U_{pq}, \phi_{pq}, \left\{\widehat{\phi}_{pq,x}\right\}\right) : \mathcal{U}_p \rightarrow \mathcal{U}_q$ is given by:
\begin{enumerate}
\item[--] $U_{pq} := \psi_p^{-1}\left(\mathrm{Im} \psi_p \cap \mathrm{Im} \psi_q\right).$
\item[--] $\phi_{pq} : U_{pq} \rightarrow U_q$ is the (usual) coordinate change for manifolds
\begin{equation}\nonumber
\phi_{pq} := \psi_q^{-1} \circ \psi_p\big|_{U_{pq}},
\end{equation}
which is an open embedding.
\item[--] $\widehat{\phi}_{pq, x}^{\mathrm{c}}: \left(\mathcal{C}'_{\phi_{pq}(x)}\right)_{\phi_{pq}} \rightarrow \mathcal{C}_x$ at each $x \in s^{-1}_p(0) \cap U_{pq} = U_{pq}$ is an isomorphism constructed as follows with $\widehat{\phi}_{pq, x} = \widehat{\phi}_{pq, x}^{\mathrm{c}} \circ \widehat{\varepsilon}_{\phi_{pq}(x), \phi_{pq}}.$ For its construction see Example 3.3 in \cite{Kim1}.
\end{enumerate}
\end{itemize}
\end{exam}

Indeed, $\mathbf{Kur}$ contains the category of smooth manifolds as a subcategory, allowing us to treat Kuranishi spaces and smooth manifolds on equal footing.
\begin{thm}\cite[Propositions 3.16 \& 3.18]{Kim1}
$L_{\infty}$-Kuranishi spaces form a category that admits a natural embedding from the category of smooth manifolds.
\end{thm}

\section{The moduli space of pseudoholomorphic maps $\mathcal{M}_{k+1}(L, \beta)$}\label{ksmod}

In this section, we prove that the classical moduli space $\mathcal{M}_{k+1}(L, \beta)$ of pseudoholomorphic disks can be endowed with the structure of $L_{\infty}$-Kuranishi space.

\subsection{FOOO's setting}
The construction of $L_{\infty}$-Kuranishi space structure relies heavily on the existing theory developed by Fukaya-Oh-Ohta-Ono, and in particular, we adopt the framework established in \cite{FOOO1}, \cite{FOOO2}, and \cite{FOOO3}.
 
Let $(M, \omega)$ be  a symplectic manifold and $L$ its compact Lagrangian submanifold.  We take an almost complex structure $J$ on $M$ which is tamed by $\omega.$ We fix a homology class $\beta \in H_2(X, L).$

\begin{defn}[The moduli space]\label{tmsdef}
We define $\mathcal{M}_{k+1}(L, \beta),$ the \textit{moduli space of pseudoholomorphic disks} by the set of tuples $\big((\Sigma, \vec{z}), u\big)$ modulo the equivalence relation $\sim$, where each component is given by:

\begin{itemize}
\item[--] $\Sigma$ is a bordered Riemann surface with genus 0 which has at worst nodal singularities.

\item[--] $\vec{z} = (z_0, \ldots, z_k) \subset \partial \Sigma$ are mutually distinct marked points, away from nodal points and enumerated counterclockwise.

\item[--] $u: (\Sigma, \partial \Sigma) \longrightarrow (M, L)$ is a continuous map with the condition $u_*\big([{\Sigma}, \partial \Sigma]\big) = \beta$ that is smooth and satisfies $\overline{\partial}_J u = 0$ on each irreducible component.

\item[--] $\big((\Sigma, \vec{z}), u\big)$ is stable, i.e., the automorphism group $\mathrm{Aut}\big((\Sigma, \vec{z}), u\big)$ is finite, where its definition is given below.
\end{itemize}
\end{defn}

\begin{defn}\label{relms}
For two tuples $\big((\Sigma, \vec{z}), u\big)$ and $\big((\Sigma', \vec{z}\,'), u'\big)$, we call a homeomorphism $g: \Sigma \to \Sigma'$ an \textit{isomorphism} if 
\begin{itemize}
\item[(i)] $g$ is biholomorphic on each irreducible component of $\Sigma$.
\item[(ii)] $u' \circ g = u$.
\item[(iii)] $g(\vec{z}_i) = \vec{z}\,'_i$, $i = 0, \ldots, k$.
\end{itemize}
We write $\big((\Sigma, \vec{z}), u\big) \sim \big((\Sigma', \vec{z}\,'), u'\big)$ if there exist an isomorphism between them. It immediately follows that $\sim$ defines an equivalence relation. We denote by $\mathrm{Aut}\big((\Sigma, \vec{z}), u\big)$ the set of isomorphism from $\big((\Sigma, \vec{z}), u\big)$ to itself, which naturally has a group structure.
\end{defn}

We denote by $\mathcal{X}_{k+1}(L,\beta)$ the set of all maps $(\Sigma, \vec{z}, u \big)$ satisfying all the axioms of $\mathcal{M}_{k+1}(L, \beta)$ except $u$ being pseudoholomorphic. Instead, we require $u$ to be of $C^2$-class on each irreducible component. Regarding $\mathcal{M}_{k+1}(L, \beta)$ as a subset of the space $\mathcal{X}_{k+1}(L,\beta)$ we can endow the pair $\big( \mathcal{X}_{k+1}(L,\beta), \mathcal{M}_{k+1}(L, \beta) \big)$ with a \textit{partial topology} whose definition we recall below.
\begin{defn}[Partial topology]\label{prttop}
Let $\mathcal{M}$ be a metrizable topological space and $\mathcal{X}$ a set that contains $\mathcal{M}.$ A \textit{partial topology} on the pair of sets $(\mathcal{X}, \mathcal{M})$ by definition assigns a neighborhood $B_{\epsilon}(\mathcal{X}, \mathbf{p}) \subset \mathcal{X}$ to each $\mathbf{p} \in \mathcal{M}$ and $\epsilon >0$ with the following properties:
\begin{enumerate}
\item[--] $\{B_{\epsilon}(\mathcal{X}, \mathbf{p}) \mid \mathbf{p} \in \mathcal{M}, \epsilon > 0 \}$ is a basis of the topology of $\mathcal{M}.$
\item[--] For each $\mathbf{p} \in \mathcal{M}$ and $\epsilon > 0$ and $\mathbf{q} \in B_{\epsilon}(\mathcal{X}, \mathbf{p}) \cap \mathcal{M},$ there exists $\delta >0$ such that $B_{\delta}(\mathcal{X}, \mathbf{q}) \subset B_{\epsilon}(\mathcal{X}, \mathbf{p}).$
\item[--] If $\epsilon_1 < \epsilon_2,$ then $B_{\epsilon_1}(\mathcal{X}, \mathbf{p}) \subset B_{\epsilon_2}(\mathcal{X}, \mathbf{p}).$ Moreover, we have $\{\mathbf{p}\} = \bigcap\limits_{\epsilon > 0} B_{\epsilon}(\mathcal{X}, \mathbf{p}).$
\end{enumerate}
\end{defn}

\begin{rem}
A partial topology on $(\mathcal{X}, \mathcal{M})$ allows us to consider a neighborhood of $\mathbf{p} \in \mathcal{M}$ in $\mathcal{X}$ without endowing (possibly pathological) $\mathcal{X}$ with a topology.
\end{rem}

\begin{thm}
There exists a topology on $\mathcal{M}_{k+1}(L,\beta)$ which is a compact and Hausdorff.
\end{thm}

\begin{proof}
We can use the stable map topology of \cite[Definition 10.3]{FO} and \cite[Definitoin 7.1.42]{FOOO2} on $\mathcal{M}_{k+1}(L,\beta).$
\end{proof}

With respect to this topology, we have:

\begin{prop}\cite[Proposition 4.3]{FOOO3}
The pair $\big(\mathcal{X}_{k+1}(L,\beta), \mathcal{M}_{k+1}(L,\beta)\big)$ defines a partial topology.
\end{prop}

For $\mathbf{p} := \left[\left((\Sigma_{\mathbf{p}}, \vec{z}_{\mathbf{p}}), u_{\mathbf{p}} \right)\right] \in \mathcal{M}_{k+1}(L, \beta),$ we denote by
\begin{equation}\nonumber
\mathscr{U}_{\mathbf{p}} \subset \mathcal{X}_{k+1}(L, \beta)
\end{equation}
an open neighborhood of $\mathbf{p}$ in $\mathcal{X}_{k+1}(L, \beta)$ determined by Definition \ref{prttop}. 

We consider a finite dimensional subspace
\begin{equation}\nonumber
{E}_{\mathbf{p}}(\mathsf{x}) \subset C^2(\Sigma_{\mathsf{x}}; u_{\mathsf{x}}^*TX \otimes \Lambda^{0,1}),
\end{equation}
that consists of $C^2$-maps with the supports away from the nodal points.

\begin{defn}[Obstruction bundle data]
For each point $\mathsf{x} \in \mathscr{U}_{\mathbf{p}} \subset \mathscr{X}_{k+1}(L, \beta),$
we define \textit{obstruction bundle data} by a family of $C^2$-tangent spaces $\{E_{\mathbf{p}}(\mathsf{x})\}_{\mathsf{x} \in \mathscr{U}_{\mathbf{p}}}$ with the following properties:
\begin{itemize}
\item[--] (Transversality) The Fredholm operator
\begin{equation}\nonumber
D_{u_{\mathbf{p}}} \overline{\partial} : W_{m+1}^2(\Sigma_{\mathbf{p}}, \partial \Sigma_{\mathbf{p}}; u_{\mathbf{p}}^*TX, u_{\mathbf{p}}^*TL) \rightarrow L_m^2(\Sigma_{\mathbf{p}}; u_{\mathbf{p}}^*TX \otimes \Lambda^{0,1})
\end{equation}
satisfies Im$D_{u_{\mathbf{p}}} \overline{\partial} + E_{\mathbf{p}}(\mathbf{p}) = L_m^2(\Sigma_{\mathbf{p}}; u_{\mathbf{p}}^*TX \otimes \Lambda^{0,1}).$
\item[--] (Semi-continuity) If $\mathbf{p} \in \mathscr{U}_{\mathbf{q}} \cap  \mathcal{M}_{k+1}(L, \beta)$ and $\mathsf{x} \in \mathscr{U}_{\mathbf{p}} \cap \mathscr{U}_{\mathbf{q}},$ then we have ${E}_{\mathbf{p}}(\mathsf{x}) \subset {E}_{\mathbf{q}}(\mathsf{x}).$
\item[--] (Invariance under automorphisms) We require $v_{*}\big({E}_{\mathbf{p}}(\mathsf{x})\big) = {E}_{\mathbf{p}}(\mathsf{x})$ for the induced automorphism $v_* \in \text{Aut}\big(C^2(\Sigma_{\mathsf{x}}); u_{\mathsf{x}}^*TX \otimes \Lambda^{0,1}\big)$ from $v \in \text{Aut}(\mathsf{x}).$
\item[--] (Smoothness) ${E}_{\mathbf{p}}(\mathsf{x})$ depends smoothly on $\mathsf{x}$ in the sense of \cite{FOOO3} Definition 8.7.
\end{itemize}

Given obstruction bundle data $\{E_{\mathbf{p}}(\mathsf{x})\},$ we now construct a Kuranishi atlas on $\mathcal{M}_{k+1}(L, \beta)$. To each point $\mathbf{p} \in \mathcal{M}_{k+1}(L, \beta),$ we assign a Kuranishi chart
\begin{equation}\label{kchrt}
\mathcal{U}_{\mathbf{p}}= (U_{\mathbf{p}}, {E}_{\mathbf{p}}, s_{\mathbf{p}}, \Gamma_{\mathbf{p}}, \psi_{\mathbf{p}}),
\end{equation}
where each component is given by:
\begin{itemize}
\item[--] $U_{\mathbf{p}} := (U_{\mathbf{p}}, \omega_{{\mathbf{p}}}),$ where
\begin{itemize}
\item[--] $U_{\mathbf{p}} := \{ \mathsf{x} \in \mathscr{U}_{\mathbf{p}} \mid \overline{\partial}u_{\mathsf{x}} \in E_{\mathbf{p}} (\mathsf{x}) \}$ is a neighborhood of $\mathbf{p}$ in $\mathscr{U}_{\mathbf{p}}$ (cf. Remark \ref{asfpgo}).
\item[--] $\omega_{{\mathbf{p}}}$ is a closed two-form on $U_{\mathbf{p}}$ defined in Subsection \ref{c2for}, satisfying Condition \ref{condusi3}.
\end{itemize}
\end{itemize}
\begin{itemize}
\item[--] ${E}_{\mathbf{p}} : = \bigcup\limits_{\mathsf{x} \in {U}_{\mathbf{p}}} E_{\mathbf{p}}(\mathsf{x})\times \{\mathsf{x}\}$ is the vector bundle over ${U}_{\mathbf{p}}$ with fiber obtained from the obstruction bundle data $\{E_{\mathbf{p}}(\mathsf{x})\}.$
\item[--] $s_{\mathbf{p}} : U_{\mathbf{p}} \rightarrow {E}_{\mathbf{p}}$ is the smooth section given by $\mathbf{x} \mapsto (\overline{\partial} u_{\mathbf{x}}, \mathbf{x}).$
\item[--] $\Gamma_{\mathbf{p}}:= \mathrm{Aut}(\mathbf{p})$ is the automorphism group (cf. Remark \ref{asfpgo}).
\item[--] $\psi_{\mathbf{p}} : s^{-1}_{\mathbf{p}}(0)/\Gamma_{\mathbf{p}} \rightarrow \mathcal{M}_{k+1}(L, \beta)$ is the obvious homeomorphism on the image given by $\mathbf{x} \mapsto \mathbf{x}.$
\end{itemize}

\begin{rem}\label{asfpgo}
We may assume that the orbifold $U_{\mathbf{p}}$ in \cite{FOOO3} is a global quotient by taking an open subset of $\mathscr{U}_{\mathbf{p}}$ (containing the point $\mathbf{p}$), if necessary. In other words, we can regard our $U_{\mathbf{p}}$ as a manifold equipped with a effective group action by $\mathrm{Aut}(\mathbf{p}).$ (See \cite[Lemma 29.1]{FOOO1}.)
\end{rem}

In the remaining part of this subsection, we briefly review the definitions of the FOOO Kuranishi charts and their embeddings (cf. \cite{FOOO1}) in the general setting, and recall a result from \cite{Kim1} stating that Definition \ref{ourcemb} of an open embedding is its correct generalization.

\begin{defn}[FOOO Kuranishi charts]\label{deflkur}
Let $X$ be a compact metrizable space. We call a tuple $\mathscr{U} := (U, E, s, \Gamma, \psi)$ an \textit{FOOO Kuranishi chart} of $X$ if the following conditions are satisfied:
\begin{enumerate}[label = (\roman*)]
\item[--] $U$ is a {simply connected} orbifold.
\item[--] $E$ is a {trivial} vector bundle of finite rank on $U_p.$
\item[--] $s : U \rightarrow E$ is a smooth section.
\item[--] $\Gamma$ is a finite group acting on $U,$ preserving $s^{-1}(0).$
\item[--] $\psi : {s^{-1}(0)}/{\Gamma} \overset{\simeq}{\hookrightarrow} X$ is a homeomorphism onto the image.
\end{enumerate}
\end{defn}

\begin{defn}[FOOO chart embeddings]\label{foemb}
Given two FOOO Kuranishi charts $\mathcal{U}$ and $\mathcal{U}'$ of $X,$ an \textit{FOOO embedding} $\Phi := (\phi, \widetilde{\phi}) : \mathcal{U} \hookrightarrow \mathcal{U}'$ consists of:
\begin{enumerate}
\item[--] $\phi : U \hookrightarrow U',$ an orbifold embedding,
\item[--] $\widetilde{\phi} : E \hookrightarrow E',$ a vector bundle embedding,
\end{enumerate}
and we require $\Phi$ to satisfy the following conditions:
\begin{enumerate}
\item[(i)] $\widetilde{\phi} \circ s = s' \circ \phi,$
\item[(ii)] $\psi' \circ \phi = \psi$ on $s^{-1}(0),$
\item[(iii)](\textit{Tangent bundle condition}) $ds'$ induces an isomorphism
\begin{equation}\label{tancond}
[ds'_{\phi(x)}] : \frac{T_{\phi(x)}U'}{\phi_*(T_xU)} \xrightarrow{\simeq} \frac{E'_{\phi(x)}}{\widetilde{\phi}(E_{x})},
\end{equation}
at each $x \in s^{-1}(0).$
\end{enumerate}
\end{defn}

Before proceeding to the statement of Proposition \ref{afec}, we need the following additional conditions.

\begin{cond}[Additional conditions]\label{addcond}
Here we add two more conditions to the conditions (i), (ii), and (iii) in Definition \ref{foemb}. Before proceeding, we write $E^c$ for a complement of $\widetilde{\phi}(E)$ in $E'$ and $p^c: E' \twoheadrightarrow E^c$ for the canonical projection. We then additionally require:
\begin{enumerate}[label=(\roman*), start=4]
\item$ p^c(s')|_{\phi(U)} \equiv 0.$
\item (After fixing a local trivialization,) the tangent bundle condition holds
\begin{equation}\label{tancond4}
[d_ys'|_{W_x}] : \frac{T_{\phi(x)}W'_{\phi(x)}}{\phi_*(T_yW_x)} \xrightarrow{\simeq} \frac{E'_{\phi(y)}}{\widetilde{\phi}(E_{y})}
\end{equation}
for all $x \in s^{-1}(0)$ and for every $y \in W_x$ (and not for $x$ alone).
\end{enumerate}
\end{cond}
We provide justification for imposing the conditions (iv) and (v):
\begin{enumerate}[label=(\roman*), start=4]
\item This condition is indeed satisfied by the coordinate changes for the moduli space pseudoholomorphic disks, one of the primary examples of FOOO Kuranishi spaces (cf. \cite{FOOO2}).
\item The linearization (with a choice of local trivialization of $E$ over on $W_x$) being an isomorphism is an open condition with respect to $x \in W_x.$ Hence, by taking $W_x$ smaller if necessary, one can ensure that $\left[d_ys'|_{W_x}\right]$ is an isomorphism for all $y \in W_x.$ 
\end{enumerate}

Then we have:

\begin{prop}\cite[Proposition 2.30]{Kim1}\label{afec}
An FOOO embedding together with the above conditions $(iv)$ and $(v)$ determines an embedding of Kuranishi chart in the sense of Definition \ref{ourcemb}.
\end{prop}

\subsection{Base coordinate changes}\label{bccd}
The coordinate change for the base component is, in essence, largely consistent with the approach presented in the FOOO's works. Consequently, the material in this subsection can be regarded primarily as a review of \cite{FOOO2}.

Let $\left\{{E_{\mathbf{p'}}(\mathbf{x})}\right\}$ be obstruction bundle data. Let $\mathcal{U}_{\mathbf{p'}}$ and $\mathcal{U}_{\mathbf{q}}$ be two Kuranishi charts at $\mathbf{q} \in \mathcal{M}_{k+1}(L, \beta)$ and $\mathbf{p'} \in \mathscr{U}_{\mathbf{q}} \cap \mathcal{M}_{k+1}(L, \beta),$ respectively, with the property: $\mathbf{p'} \in \mathrm{Im} \psi_{\mathbf{q}}$.
We denote
\begin{equation}\nonumber
{U}_{\mathbf{p'}\mathbf{q}} := U_{\mathbf{p'}} \cap \mathscr{U}_{\mathbf{q}}.
\end{equation}
For $\mathbf{x} \in U_{\mathbf{p'}\mathbf{q}},$ by the semi-continuity of the obstruction bundle data, we have $\partial u_{\mathbf{x}} \in E_{\mathbf{p'}}(\mathbf{x}) \subseteq E_{\mathbf{q}}(\mathbf{x}),$ from which we obtain the inclusion map
\begin{equation}\nonumber
\phi_{\mathbf{p'}\mathbf{q}} : U_{\mathbf{p'}\mathbf{q}} \hookrightarrow U_{\mathbf{q}}.
\end{equation}

Moreover, we have the inclusion of the total space of vector bundles
\begin{equation}\nonumber
\widetilde{\phi}_{\mathbf{p'q}} : {E}_{\mathbf{p'}}|_{U_{\mathbf{p'}\mathbf{q}}} \hookrightarrow {E}_{\mathbf{q}},
\end{equation}
which gives rise to a fiber-wise injection of vector bundles on $U_{\mathbf{p'}\mathbf{q}}$
\begin{equation}\nonumber
{E}_{\mathbf{p'}}|_{U_{\mathbf{p'}\mathbf{q}}} \hookrightarrow \phi_{\mathbf{p'}\mathbf{q}}^*{E}_{\mathbf{q}} = {E_{\mathbf{q}}}|_{U_{\mathbf{p'q}}} \hookrightarrow E_{\mathbf{q}}.
\end{equation}

In fact, we have:
\begin{lem}\cite[Lemma 7.7]{FOOO3}
$\left\{\left(U_{\mathbf{p'}\mathbf{q}}, \phi_{\mathbf{p'q}}, \widetilde{\phi}_{\mathbf{p'q}}\right)\right\}$ defines  a coordinate change for an FOOO Kuranishi space.
\end{lem}

The above discussion yields a bundle embedding
\begin{equation}\nonumber
\begin{tikzcd}
E_{\mathbf{p}}|_{U_{\mathbf{p}\mathbf{q}}} \arrow[hook]{r}{\widetilde{{\phi}}_{\mathbf{p}\mathbf{q}}} \arrow{d} & E_{\mathbf{q}} \arrow{d}\\
U_{\mathbf{p}\mathbf{q}} \arrow[hook]{r}{\phi_{\mathbf{p}\mathbf{q}}} & U_{\mathbf{q}},
\end{tikzcd}
\end{equation}
hence an FOOO chart embedding. (Here, the upper horizontal line is understood as an \textit{inclusion} after the identification by parallel transport.) Moreover, the following properties are satisfied:
\begin{enumerate}[label = (\roman*)]
\item Their (virtual) dimensions are the same: dim $\mathcal{U}_{\mathbf{p}} = $ dim $\mathcal{U}_{\mathbf{q}}.$
\item $\phi_{\mathbf{p}\mathbf{q}}$ is $(\Gamma_p, \Gamma_q)$-equivariant as it is an inclusion and the group action coincides at points of both the domain and the image of $\phi_{\mathbf{p}\mathbf{q}}.$
\item Write $s_{\mathbf{p}} = (s_{\mathbf{p}}^{1}, \cdots, s_{\mathbf{p}}^{\mathrm{rk}E_{\mathbf{p}}})$ and $s_{\mathbf{q}} = (s_{\mathbf{q}}^{1}, \cdots, s_{\mathbf{q}}^{\mathrm{rk}E_{\mathbf{q}}})$ in the pre-chosen orthonormal frame, so that $\overline{\phi} \circ s^i_{\mathbf{p}} = s^i_{\mathbf{q}}$, $i = 1, \cdots, \mathrm{rk}\,E_{\mathbf{p}}$. Then we have $\phi_{\mathbf{pq}}^*s_2^{\mathrm{rk}E_{\mathbf{p}}+1} = \cdots = \phi_{\mathbf{pq}}^*s_2^{\mathrm{rk}E_{\mathbf{q}}} = 0$.
\item The FOOO tangent bundle condition holds, that is, there exists an isomorphism of vector spaces:
\begin{equation}
[ds_{\mathbf{q}, \phi_{\mathbf{pq}(\mathbf{x})}}] : \frac{T_{\phi_{\mathbf{pq}}(\mathbf{x})}U_{\mathbf{q}}}{\phi_{{\mathbf{pq}}*}(T_{\mathbf{x}}U_{\mathbf{p}})} \xrightarrow{\simeq} \frac{E_{\mathbf{q},\phi_{\mathbf{pq}}(\mathbf{x})}}{\widetilde{\phi}_{\mathbf{pq}}(E_{\mathbf{p},\mathbf{x}})}
\end{equation}
for each $\mathbf{x} \in s^{-1}_{\mathbf{p}}(0).$ Here, in fact we have $\mathbf{x} = \phi_{\mathbf{pq}}(\mathbf{x});$ however, we keep this expression to make the context clearer. 
\end{enumerate}

Given the above data with an implicit choice of $\mathbf{p}' \in \mathrm{Im}\psi_{\mathbf{p}} \cap \mathrm{Im}\psi_{\mathbf{q}},$ we obtain the tuple $\left(U_{\mathbf{p}\mathbf{q}}, \phi_{\mathbf{p}\mathbf{q}}, \widetilde{\phi}_{\mathbf{p}\mathbf{q}}\right),$ where we denote

\begin{enumerate}
\item[--] $U_{\mathbf{pq}} := U_{\mathbf{p'p}} \cap U_{\mathbf{p'q}},$
\item[--] $\phi_{\mathbf{pq}} := \phi_{\mathbf{p'q}}|_{U_{\mathbf{pq}}} : U_{\mathbf{p}\mathbf{q}} \hookrightarrow U_{\mathbf{q}},$
\item[--] $\widetilde{\phi}_{\mathbf{pq}} := \widetilde{\phi}_{\mathbf{p'q}}|_{U_{\mathbf{pq}}}.$
\end{enumerate}

\noindent
($\textit{The base coordinate change}$) 
We close this subsection by defining base coordinate change by the above-mentioned data: For our $L_{\infty}$-Kuranishi base coordinate change from $\mathcal{U}_{\mathbf{p}}$ to $\mathcal{U}_{\mathbf{q}}$ for the moduli space, we set
\[
\Phi_{pq} := \left(U_{\mathbf{p}\mathbf{q}}, \phi_{\mathbf{p}\mathbf{q}}, \widetilde{\phi}_{\mathbf{p}\mathbf{q}}\right).
\]

\subsection{The closed two-form $\omega_{\mathbf{p}}$}\label{c2for}
Importantly, the ambient symplectic form plays a crucial role in generating the de Rham part algebraic structure through the introduction of a closed two-form on each Kuranishi chart.

Using the symplectic form $\omega$ of $M$, we define a two-form $\omega_{\mathbf{p}} = \{\omega_{\mathbf{p},\mathbf{y}}\}_{\mathbf{y} \in U_{\mathbf{p}}}$ on $U_{\mathbf{p}}$ by
\begin{equation}\label{o2f}
\omega_{\mathbf{p}, \mathbf{y}}(X_{\mathbf{y}}, Y_{\mathbf{y}}) := \int\limits_{\Sigma} \omega(X_{\mathbf{y}}, Y_{\mathbf{y}}) \mathrm{dvol}_{\Sigma}
\end{equation}
for $X_{\mathbf{y}}, Y_{\mathbf{y}} \in T_{\mathbf{y}}U_{\mathbf{p}} \subset \Gamma(\Sigma, u_\mathbf{y}^*TM), \mathbf{y} \in U_{\mathbf{p}},$ and a fixed volume form $ \mathrm{dvol}_{\Sigma}$ on $\Sigma.$

\begin{rem}
We do not require $\mathrm{Aut}(\mathbf{p})$-invariance of $\omega_{\mathbf{p}}$, and hence it does not restrict to the moduli in general. This is because the essential ingredient is the algebraic structure constructed in Section 4, for which $\omega_{\mathbf{p}}$ serves merely as auxiliary data. In fact, the choice $\mathrm{dvol}_{\Sigma} = u_{\mathbf{p}}^* \omega$ (modulo the fact that it can possibly degenerate) would, for example, ensure such invariance; however, this choice depends on individual charts and would render the subsequent discussion significantly more complicated. We will address this issue in future work.
\end{rem}

\begin{lem}
$\omega_{\mathbf{p}}$ is a closed two-form on $U_{\mathbf{p}}.$
\end{lem}

\begin{proof}
We first compute for vector fields $X, Y,$ and $Z \in \Gamma(TU_{\mathbf{p}}),$
\begin{equation}\label{xapyz}
\begin{split}
X &\omega_{\mathbf{p}}(Y,Z) = X\Big\{\int\limits_{\Sigma} \omega(Y_{\mathbf{y}}, Z_{\mathbf{y}}) \mathrm{dvol}_{\Sigma}\Big\}_{\mathbf{y}} = \frac{d}{d\tau}\Big|_{\tau=0} \Big\{\int\limits_{\Sigma} \omega(Y_{\widetilde{\mathbf{y}}(\tau)}, Z_{\widetilde{\mathbf{y}}(\tau)}) \mathrm{dvol}_{\Sigma}\Big\}_{\mathbf{y}}\\
&=\Big\{\frac{d}{d\tau}\Big|_{\tau=0}\int\limits_{\Sigma} \omega(Y_{\widetilde{\mathbf{y}}(\tau)}, Z_{\widetilde{\mathbf{y}}(\tau)}) \mathrm{dvol}_{\Sigma}\Big\}_{\mathbf{y}}\overset{*}{=}\Big\{\int\limits_{\Sigma} \frac{d}{d\tau}\Big|_{\tau=0} \omega(Y_{\widetilde{\mathbf{y}}(\tau)}, Z_{\widetilde{\mathbf{y}}(\tau)}) \mathrm{dvol}_{\Sigma}\Big\}_{\mathbf{y}}\\
&=\Big\{\int\limits_{\Sigma} X_{{\mathbf{y}}} \left(\omega(Y_{{\mathbf{y}}}, Z_{{\mathbf{y}}})\right) \mathrm{dvol}_{\Sigma}\Big\}_{\mathbf{y}},
\end{split}
\end{equation}
where $\widetilde{\mathbf{y}} : (-1,1) \rightarrow U_{\mathbf{p}},$ is a curve that satisfies $ \widetilde{\mathbf{y}}(0) = \mathbf{y}$ and $\frac{d}{d\tau}\big|_{\tau=0}\widetilde{\mathbf{y}}(\tau) = X_{\mathbf{y}},$ and $\{\cdots\}_{\mathbf{y}}$ stands for a smooth family in $\mathbf{y} \in U_{\mathbf{p}}.$ Among the equalities in (\ref{xapyz}), $*$ non-trivially holds by the Leibniz integral rule (for a \textit{fixed} domain, that is, a $\tau$-independent $\Sigma$) and the Lagrangian boundary condition: For each $\mathbf{y},$ we have
\begin{equation}\nonumber
\begin{split}
\frac{d}{d\tau}&\Big|_{\tau=0}\int\limits_{\Sigma} \omega(Y_{\widetilde{\mathbf{y}}(\tau)}, Z_{\widetilde{\mathbf{y}}(\tau)}) \mathrm{dvol}_{\Sigma}\\ 
&= \int\limits_{\Sigma} \frac{\partial}{\partial\tau}\Big|_{\tau=0} \omega(Y_{\widetilde{\mathbf{y}}(\tau)}, Z_{\widetilde{\mathbf{y}}(\tau)}) \mathrm{dvol}_{\Sigma} + \int_{\partial \Sigma} \omega(Y_{\widetilde{\mathbf{y}}(\tau)}, Z_{\widetilde{\mathbf{y}}(\tau)}) \iota_{\vec{n}}(\mathrm{dvol}_{\Sigma})\\
&= \int\limits_{\Sigma} \frac{d}{d\tau}\Big|_{\tau=0} \omega(Y_{\widetilde{\mathbf{y}}(\tau)}, Z_{\widetilde{\mathbf{y}}(\tau)}) \mathrm{dvol}_{\Sigma},
\end{split}
\end{equation}
where $\iota_{\vec{n}}$ denotes the intreior product with unit normal vector at the boundary $\partial \Sigma.$

Using this, we obtain
\begin{equation}\nonumber
\begin{split}
d\omega_{\mathbf{p}}(X, Y,& Z) = X \omega_{\mathbf{p}}(Y,Z) - Y \omega_{\mathbf{p}}(X,Z) +  Z \omega_{\mathbf{p}}(X,Y)\\ 
&+ \omega_{\mathbf{p}}([X,Y], Z) + \omega_{\mathbf{p}}([X,Z],Y) +  \omega_{\mathbf{p}}([Y,Z],X)\\
&= \Big\{\int\limits_{\Sigma} X_{{\mathbf{y}}} \left(\omega(Y_{{\mathbf{y}}}, Z_{{\mathbf{y}}})\right) \mathrm{dvol}_{\Sigma} \Big\}_{\mathbf{y}} -  \Big\{\int\limits_{\Sigma} Y_{{\mathbf{y}}} \left(\omega(Z_{{\mathbf{y}}}, X_{{\mathbf{y}}})\right) \mathrm{dvol}_{\Sigma} \Big\}_{\mathbf{y}}\\
& +  \Big\{\int\limits_{\Sigma} Z_{{\mathbf{y}}} \left(\omega(X_{{\mathbf{y}}}, Y_{{\mathbf{y}}})\right) \mathrm{dvol}_{\Sigma} \Big\}_{\mathbf{y}} + \Big\{\int\limits_{\Sigma} \omega([X_{{\mathbf{y}}}, Y_{{\mathbf{y}}}], Z_{{\mathbf{y}}}) \mathrm{dvol}_{\Sigma} \Big\}_{\mathbf{y}}\\
&+  \Big\{\int\limits_{\Sigma} \omega([X_{{\mathbf{y}}}, Z_{{\mathbf{y}}}], Y_{{\mathbf{y}}}) \mathrm{dvol}_{\Sigma} \Big\}_{\mathbf{y}} +  \Big\{\int\limits_{\Sigma} \omega([Y_{{\mathbf{y}}}, Z_{{\mathbf{y}}}], X_{{\mathbf{y}}} ) \mathrm{dvol}_{\Sigma} \Big\}_{\mathbf{y}}\\
&=  \Big\{\int\limits_{\Sigma}\Big( X_{{\mathbf{y}}} \left(\omega(Y_{{\mathbf{y}}}, Z_{{\mathbf{y}}})\right) - Y_{{\mathbf{y}}} \left(\omega(Z_{{\mathbf{y}}}, X_{{\mathbf{y}}})\right) + Z_{{\mathbf{y}}} \left(\omega(X_{{\mathbf{y}}}, Y_{{\mathbf{y}}})\right)\\
& \quad \quad \quad + \omega([X_{{\mathbf{y}}}, Y_{{\mathbf{y}}}], Z_{{\mathbf{y}}}) + \omega([X_{{\mathbf{y}}}, Z_{{\mathbf{y}}}], Y_{{\mathbf{y}}}) + \omega([Y_{{\mathbf{y}}}, Z_{{\mathbf{y}}}], X_{{\mathbf{y}}} ) \Big)\mathrm{dvol}_{\Sigma} \Big\}_{\mathbf{y}}\\
&= \Big\{\int\limits_{\Sigma} d\omega \left( X_{{\mathbf{y}}}, Y_{{\mathbf{y}}}, Z_{{\mathbf{y}}}\right) \mathrm{dvol}_{\Sigma} \Big\}_{\mathbf{y}} = \Big\{\int\limits_{\Sigma} (d\omega)( X_{{\mathbf{y}}}, Y_{{\mathbf{y}}}, Z_{{\mathbf{y}}}) \mathrm{dvol}_{\Sigma} \Big\}_{\mathbf{y}} = 0.
\end{split}
\end{equation}
\end{proof}

We now make an important assumption on the closed two-form $\omega_{\mathbf{p}}.$ 
\begin{cond}[Condition on the two-form $\omega_{\mathbf{p}}$]\label{condusi3}
For each point $\mathbf{p},$ we assume that the closed two-form $\omega_{\mathbf{p}}$ on $U_{\mathbf{p}}$ allows the stratification 
\begin{equation}\label{fffffsdfd}
\U_{\mathbf{p}} = \bigcup\limits_i \mathcal{S}_{\mathbf{p},i},
\end{equation}
into \textit{submanifolds} $ \mathcal{S}_{\mathbf{p},i} := \{\mathbf{y} \in U_{\mathbf{p}} \mid \mathrm{rk}\left(\ker\omega_{\mathbf{p}, \mathbf{y}}\right) = i\} \ (0 \leq i \leq \dim U_{\mathbf{p}})$ and their tubular neighborhoods:
\[
\begin{cases}
\iota_{i} : N_i \rightarrow  U_{\mathbf{p}}, \text{ an open neighborhood of each (possibly non-connected) }  \mathcal{S}_{\mathbf{p},i},\\
\pi_{i} : N_i \rightarrow \mathcal{S}_{\mathbf{p},i}, \text{ the associated projection}.
\end{cases}
\]
\end{cond}

\begin{lem}
The decomposition (\ref{fffffsdfd}) restricts to the boundary. That is, we have a decomposition
\begin{equation}\label{adffff}
\partial U_{\mathbf{p}} = \bigcup_i \left( \mathcal{S}_i \cap \partial U_{\mathbf{p}} \right),
\end{equation}
where $\mathcal{S}_i \cap \partial U_{\mathbf{p}}$ is a submanifold of $\partial U_{\mathbf{p}}$ given by
\[
\mathcal{S}_i \cap \partial U_{\mathbf{p}} = \{\mathbf{y} \in \partial U_{\mathbf{p}} \mid \mathrm{rk} (\ker \omega_{\mathbf{p}, \mathbf{y}}) = i - 1 \}, \ 1 \leq i \leq \dim U_{\mathbf{p}},
\]
together with the corresponding tubular neighborhood in $\partial U$ for each $i$:
\[
\begin{cases}
\iota^{\partial}_{i} : N^{\partial}_i \rightarrow  \partial U_{\mathbf{p}}, \text{ an open neighborhood of each }  \mathcal{S}_{i},\\
\pi^{\partial}_{i} : N^{\partial}_i \rightarrow \mathcal{S}_{i} \cap \partial U_{\mathbf{p}}, \text{ the associated projection}.
\end{cases}
\]
\end{lem}

\begin{proof}
Recall that the normal direction at a boundary point concerns resolving the boundary singularity of the pseudoholomorphic disk, that is, the inverse of the gluing process (cf. \cite[Subsection 7.1.3]{FOOO2}). This means that the normal vector $v \in \Gamma\left(N_{\mathbf{y}} U\right)$ at $\mathbf{y} \in \partial U$ is identified with a vector field supported at the singular point of the Riemann surface $\Sigma$. Then, from the formula (\ref{o2f}) for $\omega_{\mathbf{p}}$, we conclude that $v \in \ker(\omega_{\mathbf{p}, \mathbf{y}})$. This further implies that $\mathcal{S}_i \cap \partial U_{\mathbf{p}}$ consists of the points $\mathbf{y}$ for which the rank of $\omega_{\mathbf{p},\mathbf{y}}$ drops by $1$, which amounts to the desired result.
\end{proof}

\end{defn}

\begin{rem}
According to [KO], a generic choice of the closed two-form makes it possible to obtain the stratification of Condition \ref{condusi3}. In this perspective, we conjecture that the same can be achieved by a generic choice of almost complex structure $J$ on the symplectic manifold $M.$ We will study this point in future work.
\end{rem}

Suppose that a base coordinate change 
\[
\Phi_{pq} := \left(U_{\mathbf{p}\mathbf{q}}, \phi_{\mathbf{p}\mathbf{q}}, \widetilde{\phi}_{\mathbf{p}\mathbf{q}}\right).
\]
from $\mathcal{U}_{\mathbf{p}}$ to $\mathcal{U}_{\mathbf{q}}$ (from the previous subsection) is given.

\begin{lem}\label{kwpq}
We have $\phi_{{\mathbf{p}\mathbf{q}}}^*\omega_{\mathbf{q}}' = \omega_{\mathbf{p}}.$
\end{lem}
\begin{proof}
It follows from the fact that  $\phi_{\mathbf{p}\mathbf{q}}$ is an inclusion and that $\omega_{\mathbf{p}}$ and $\omega_{\mathbf{q}}'$ are induced from the same ambient symplectic form $\omega$ by the formula (\ref{o2f}). 
\end{proof}

For an open neighborhood of $\mathbf{x},$ $\overset{\circ}{W_{\mathbf{x}}} \subset \mathcal{S}_{\mathbf{p},i},$ we consider two-forms
\[
\omega_{\mathbf{p},W_{\mathbf{x}}} := \pi^*_{i}(\omega_{\mathbf{p}}|_{\overset{\circ}{W_{\mathbf{x}}}}) \in \Omega^2\left(W_{\mathbf{x}}\right)
\]
and 
\[
\omega'_{\mathbf{q}, W'_{\phi_{\mathbf{pq}}(\mathbf{x})}} := \pi_{i'}^{\prime *}\left(\omega_{\mathbf{q}}'\big|_{\overset{\circ}{W}'_{\phi_{\mathbf{pq}}(\mathbf{x})}}\right) \in \Omega^2\left({W}'_{\phi_{\mathbf{pq}}(\mathbf{x})}\right)
\]
that are presymplectic by construction.

Let
\[
\pi_{\mathbf{pq,x}} : W'_{\phi_{\mathbf{pq}}(\mathbf{x})} \twoheadrightarrow \phi_{\mathbf{pq}} \left(W_{\mathbf{x}}\right)
\]
be an implicitly chosen projection for the embedding $\phi_{\mathbf{pq}}$ and
\[
\phi_{\mathbf{pq}} \left(W_{\mathbf{x}}\right) \hookrightarrow W_{\mathbf{x}}
\]
the obvious embeddinig. We then denote 
\begin{equation}\label{tf0x}
T{\mathcal{F}}^{'0}_{\phi_{\mathbf{p}\mathbf{q}(\mathbf{x})}} := \ker\left( \left(\phi^{-1}_{\mathbf{pq}} \circ \pi_{\mathbf{pq,x}}\right)^*\left(\omega_{\mathbf{p},W_{\mathbf{x}}}\right)\right) \subset TU'_{\mathbf{q}}|_{W'_{\phi_{\mathbf{pq}}(\mathbf{x})}}. 
\end{equation} 

The following corollary plays a useful role in the construction of $L_{\infty}$-coordinate changes in Section 4.
\begin{cor}\label{fctfv} We have an identification of vector bundles over $\phi_{\mathbf{p}\mathbf{q}}(W_{\mathbf{x}})$:
\[
T\mathcal{F}^{'0}_{\phi_{\mathbf{p}\mathbf{q}}(\mathbf{x})}\mid_{\phi_{\mathbf{pq}}(W_{\mathbf{x}})} \simeq (\phi^{-1}_{\mathbf{p}\mathbf{q}})^{*}T\mathcal{F}_{\mathbf{x}} \oplus V
\]
for some ($\dim U_{\mathbf{q}} - \dim U_{\mathbf{p}}$)-dimensional vector space $V.$ 
\end{cor}

\begin{proof}
By Lemma \ref{kwpq}, we know that $(\phi^{-1}_{\mathbf{pq,x}})^*(T\mathcal{F}_{\mathbf{x}}) \subset T{\mathcal{F}}^{'0}_{\phi_{\mathbf{p}\mathbf{q}}(\mathbf{x})}|_{\phi_{\mathbf{p}\mathbf{q}}(W_{\mathbf{x}})}.$
\end{proof}

\section{$L_{\infty}$-coordinate changes}

In this section, which may be regarded as the core of this paper, we provide a construction of the $L_{\infty}$-component of coordinate changes for the moduli space $\mathcal{M}_{k+1}(L, \beta)$, including all the technicalities. As a consequence, we show that $\mathcal{M}_{k+1}(L, \beta)$ can be understood as an $L_{\infty}$-Kuranishi space. The prerequisite for this section is the content of Appendices B and C on the local $L_{\infty}[1]$-algebras and the $L_{\infty}[1]$-structures arising from V-algebras, respectively.

\subsection{$L_{\infty}$-coordinate changes}
The construction of the $L_{\infty}$-component coordinate change $\widehat{\phi}_{\mathbf{p}\mathbf{q}} = \left\{\widehat{\phi}_{\mathbf{p}\mathbf{q}, \mathbf{x}}\right\}_{\mathbf{x} \in s_{\mathbf{p}}^{-1}(0)}$ is now in order.

\ \

\noindent
($\textit{The }L_{\infty}\textit{-coordinate change}$) 
For each zero point $\mathbf{x} \in s_{\mathbf{p}}^{-1}(0),$ our $L_{\infty}$-component coordinate change is given by the $L_{\infty}[1]$-morphism
\begin{equation}\nonumber
\widehat{\phi}_{\mathbf{p}\mathbf{q},\mathbf{x}} : (\mathcal{C}'_{\mathbf{q}, \phi_{\mathbf{p}\mathbf{q}}(\mathbf{x})})_{\phi_{\mathbf{p}\mathbf{q}}} \rightarrow \mathcal{C}_{\mathbf{p},\mathbf{x}},
\end{equation}
where 
\[
(\mathcal{C}'_{\mathbf{q}, \phi_{\mathbf{p}\mathbf{q}}(\mathbf{x})})_{\phi_{\mathbf{p}\mathbf{q}}} :=  C^{\infty}_{\phi_{\mathbf{p}\mathbf{q}}}(W_{\mathbf{x}}) \otimes_{C^{\infty}(W_{\mathbf{x}})} \mathcal{C}'_{\mathbf{q}, \phi_{\mathbf{p}\mathbf{q}}(\mathbf{x})}
\]
is the completion of $\mathcal{C}'_{\mathbf{q}, \phi_{\mathbf{p}\mathbf{q}}(\mathbf{x})}$ at the image of $\phi_{\mathbf{p}\mathbf{q}}.$ Here we consider the inverse limit
\begin{equation}\label{ivlmt}
C_{\phi}^{\infty}(W) := \lim_{\longleftarrow} C^{\infty}(W)/I_{\phi}^{j} \cdot C^{\infty}(W)
\end{equation}
for the ideal $I_{\phi}^{j} := \{f \in C^{\infty}(W) \mid f|_{\mathrm{Im}\phi} \equiv 0\}.$ See Definition \ref{ladef} for more details including the $L_{\infty}[1]$-structure on $(\mathcal{C}'_{\mathbf{q}, \phi_{\mathbf{p}\mathbf{q}}(\mathbf{x})})_{\phi_{\mathbf{p}\mathbf{q}}}$. 

Then $\widehat{\phi}_{\mathbf{p}\mathbf{q},\mathbf{x}}$ is defined by the composition 
\begin{equation}\label{1phi4}
\begin{split}
\widehat{\phi}_{\mathbf{p}\mathbf{q}, \mathbf{x}} &:= \widehat{\eta}_{\mathbf{p}\mathbf{q}, \mathbf{x}} \circ \widehat{\kappa}_{\mathbf{p}\mathbf{q}, \mathbf{x}}\\
 &: (\mathcal{C}'_{\mathbf{q}, \phi_{\mathbf{p}\mathbf{q}}(\mathbf{x})})_{\phi_{\mathbf{p}\mathbf{q}}} \xrightarrow{\widehat{\kappa}_{\mathbf{p}\mathbf{q}, \mathbf{x}}} ({\mathcal{C}}^{'0}_{\mathbf{q}, \phi_{\mathbf{p}\mathbf{q}}(\mathbf{x})})_{\phi_{\mathbf{p}\mathbf{q}}} \xrightarrow{\widehat{\eta}_{\mathbf{p}\mathbf{q}, \mathbf{x}}} \mathcal{C}_{\mathbf{p}, \mathbf{x}}
\end{split}
\end{equation}
whose components are soon defined. Here $({\mathcal{C}}^{'0}_{\mathbf{q}, \phi_{\mathbf{p}\mathbf{q}}(\mathbf{x})})_{\phi_{\mathbf{p}\mathbf{q}}}$ denotes the $L_{\infty}[1]$-algebra given by
\[
({\mathcal{C}}^{'0}_{\mathbf{q},_{\phi_{\mathbf{p}\mathbf{q}}(\mathbf{x})}})_{\phi_{\mathbf{p}\mathbf{q}}} := \left(\bigwedge{}^{-\bullet}\Gamma(E^{'*}_{\mathbf{q}}|_{W'_{{\phi_{\mathbf{p}\mathbf{q}}(\mathbf{x})}}})\right)_{\phi_{\mathbf{p}\mathbf{q}}} \oplus \Omega_{\mathrm{aug},{\phi_{\mathbf{p}\mathbf{q}}}}^{\bullet +1}({\mathcal{F}}^{'0}_{\phi_{\mathbf{p}\mathbf{q}}(\mathbf{x})})
\]
where $\left(\bigwedge{}^{-\bullet}\Gamma(E^{'*}_{\mathbf{q}}|_{W'_{{\phi_{\mathbf{p}\mathbf{q}}(\mathbf{x})}}})\right)_{\phi_{\mathbf{p}\mathbf{q}}} $ is the Koszul complex completed at $\mathrm{Im}\phi_{\mathbf{p}\mathbf{q}}.$ And $\Omega_{\mathrm{aug},{\phi_{\mathbf{p}\mathbf{q}}}}^{\bullet +1}({\mathcal{F}}^{'0}_{\phi_{\mathbf{p}\mathbf{q}}(\mathbf{x})})$ is the augmented foliation de Rham complex determined by the regular foliation 
\[
T{\mathcal{F}}^{'0}_{\phi_{\mathbf{p}\mathbf{q}(\mathbf{x})}} := \ker\left( \left(\phi^{-1}_{\mathbf{pq}} \circ \pi_{\mathbf{pq,x}}\right)^*\left(\omega_{\mathbf{p},W_{\mathbf{x}}}\right)\right)
\]
of (\ref{tf0x}). The $L_{\infty}[1]$-algebra structure on $\Omega_{\mathrm{aug},{\phi_{\mathbf{p}\mathbf{q}}}}^{\bullet +1}({\mathcal{F}}^{'0}_{\phi_{\mathbf{p}\mathbf{q}}(\mathbf{x})})$ depends on the choice of splitting, but it only makes an isomorphic difference by Corollary \ref{liiifdr} (iv) .

With respect to the Koszul and the de Rham parts, $\widehat{\phi}_{\mathbf{p}\mathbf{q},\mathbf{x}}$ decomposes as 
\[
\widehat{\phi}_{\mathbf{p}\mathbf{q},\mathbf{x}} := \widehat{\phi}^{\mathrm{K}}_{\mathbf{p}\mathbf{q},\mathbf{x}} \oplus \widehat{\phi}_{\mathbf{p}\mathbf{q},\mathbf{x}}^{\mathrm{dR}}.
\]
We define $\widehat{\phi}^{\mathrm{K}}_{\mathbf{p}\mathbf{q},\mathbf{x}}$ similarly as the Koszul component introduced in Proposition \ref{afec}. Namely, 
\[
\widehat{\phi}^{\mathrm{K}}_{\mathbf{p}\mathbf{q},\mathbf{x}} : \left(\bigwedge\nolimits^{-\bullet}\Gamma(E^{'*}_{\mathbf{q}}|_{W'_{{\phi_{\mathbf{p}\mathbf{q}}(\mathbf{x})}}})\right)_{\phi_{\mathbf{p}\mathbf{q}}} \rightarrow \bigwedge\nolimits^{-\bullet}\Gamma(E^*_{\mathbf{p}}|_{W_{\mathbf{x}}})
\]
is defined by the compositions of the following maps
\[
\begin{split}
\left(\bigwedge\nolimits^{-\bullet}\Gamma(E^{'*}_{\mathbf{q}}|_{W'_{{\phi_{\mathbf{p}\mathbf{q}}(\mathbf{x})}}})\right)_{\phi_{\mathbf{p}\mathbf{q}}} &\xrightarrow{(1), \simeq} \left(\bigwedge\nolimits^{-\bullet}\Gamma(E^{'*}_{\mathbf{q}}|_{W'_{{\phi_{\mathbf{p'}\mathbf{q}}(\mathbf{x})}}})\right)_{\phi_{\mathbf{p'}\mathbf{q}}} \xrightarrow{(2), \simeq} \bigwedge\nolimits^{-\bullet}\Gamma(E^*_{\mathbf{p'}}|_{W'_{\mathbf{x}}})\\ 
&\xrightarrow{(3), \simeq} \bigwedge\nolimits^{-\bullet}\Gamma(E^*_{\mathbf{p'}}|_{W_{\mathbf{x}}})_{i_{\mathbf{p'p}}} \xrightarrow{(4),\simeq} \bigwedge\nolimits^{-\bullet}\Gamma(E^*_{\mathbf{p}}|_{W_{\mathbf{x}}}),
\end{split}
\]
where we use the notations in Subsection \ref{bccd} and denote $W'_{\mathbf{x}} := W_{\mathbf{x}} \cap U_{\mathbf{p'q}}$ and $i_{\mathbf{p'p}} : W'_{\mathbf{x}} \hookrightarrow W_{\mathbf{x}}.$ The $L_{\infty}[1]$-quasi-isomorphisms (1) through (4) are given as follows: (1) is defined to be the $L_{\infty}[1]$-morphism induced by the injection 
\[
C^{\infty}_{\phi_{\mathbf{pq}}}(W'_{\phi_\mathbf{pq}(\mathbf{x})}) \hookrightarrow C^{\infty}_{\phi_{\mathbf{p'q}}}(W'_{\phi_\mathbf{p'q}(\mathbf{x})}),
\]
which is again induced from $I_{\phi_{\mathbf{p'}\mathbf{q}}} \rightarrow I_{\phi_{\mathbf{p}\mathbf{q}}}$ and $\mathrm{Im}{\phi_{\mathbf{p}\mathbf{q}}} \hookrightarrow \mathrm{Im}{\phi_{\mathbf{p'}\mathbf{q}}}$ with the observation $\phi_{\mathbf{pq}}(\mathbf{x}) = \phi_{\mathbf{p'q}}(\mathbf{x}).$

Indeed, it fits in the following commutative diagram:
\[
\begin{tikzcd}
\left(\bigwedge\nolimits^{-\bullet}\Gamma(E^{'*}_{\mathbf{q}}|_{W'_{{\phi_{\mathbf{p}\mathbf{q}}(\mathbf{x})}}})\right)_{\phi_{\mathbf{p}\mathbf{q}}} \arrow{rr}{(1)}  & {} & \left(\bigwedge\nolimits^{-\bullet}\Gamma(E^{'*}_{\mathbf{q}}|_{W'_{{\phi_{\mathbf{p'}\mathbf{q}}(\mathbf{x})}}})\right)_{\phi_{\mathbf{p'}\mathbf{q}}}\\
{} &  \bigwedge\nolimits^{-\bullet}\Gamma\left(E^{'*}_{\mathbf{q}}|_{W'_{{\phi_{\mathbf{p}\mathbf{q}}(\mathbf{x})}}}\right). \arrow{ul}{\widehat{\varepsilon}^{\mathrm{K}}_{\mathbf{q},\phi_{\mathbf{pq}}(\mathbf{x}),\phi_{\mathbf{pq}}}} \arrow{ur}[swap]{\widehat{\varepsilon}^{\mathrm{K}}_{\mathbf{q},\phi_{\mathbf{p'q}}(\mathbf{x}),\phi_{\mathbf{p'q}}}} & {}
\end{tikzcd}
\]
Then by Lemma  \ref{cochqim} and Corollary \ref{coinle}, we know ${\widehat{\varepsilon}^{\mathrm{K}}_{\mathbf{q},\phi_{\mathbf{pq}}(\mathbf{x}),\phi_{\mathbf{pq}}}}$ and ${\widehat{\varepsilon}^{\mathrm{K}}_{\mathbf{q},\phi_{\mathbf{p'q}}(\mathbf{x}),\phi_{\mathbf{p'q}}}}$ are quasi-isomorphism. Thus, ${\widehat{\varepsilon}^{\mathrm{K}}_{\mathbf{p'p}}}$ is also a quasi-isomorphism.
The $L_{\infty}[1]$-quasi-isomorphisms (2), (3), and (4) are obtained from Proposition \ref{afec}, Lemma \ref{cochqim}, and Lemma \ref{itavsteai}, respectively.

Then it remains to construct 
\begin{equation}\nonumber
\widehat{\phi}_{\mathbf{p}\mathbf{q}, \mathbf{x}}^{\mathrm{dR}} :  \Omega_{\mathrm{aug},{\phi_{\mathbf{p}\mathbf{q}}}}^{\bullet +1}\left(\mathcal{F}'_{\phi_{\mathbf{p}\mathbf{q}}(\mathbf{x})}\right) \rightarrow \Omega_{\mathrm{aug}}^{\bullet +1}(\mathcal{F}_{x}).
\end{equation}
It is again given by the following composition:
\begin{equation}\label{phi4}
 \Omega_{\mathrm{aug}}^{\bullet +1}(\mathcal{F}_{\mathbf{x}}) \xrightarrow{\widehat{\phi}^{\mathrm{dR},1}_{\mathbf{p}\mathbf{q},\mathbf{x}}} \Omega_{\mathrm{aug},{\phi_{\mathbf{p}\mathbf{q}}}}^{\bullet +1}\left(\mathcal{F}^{'0}_{\phi_{\mathbf{pq}}(\mathbf{x})}\right) \xrightarrow{\widehat{\phi}^{\mathrm{\mathrm{dR}},2}_{\mathbf{p}\mathbf{q},\mathbf{x}}}  \Omega_{\mathrm{aug},{\phi_{\mathbf{p}\mathbf{q}}}}^{\bullet +1}\left(\mathcal{F}'_{\phi_{\mathbf{p}\mathbf{q}}(\mathbf{x})}\right),
\end{equation}
i.e., $\widehat{\phi}^{\mathrm{dR}}_{\mathbf{p}\mathbf{q}, \mathbf{x}} := \widehat{\phi}^{\mathrm{dR},2}_{\mathbf{p}\mathbf{q}, \mathbf{x}} \circ \widehat{\phi}^{\mathrm{dR},1}_{\mathbf{p}\mathbf{q}, \mathbf{x}}.$ 

Our definitions of $\widehat{\phi}_{\mathbf{p}\mathbf{q}, \mathbf{x}}^{\mathrm{dR},1}$ and $\widehat{\phi}_{\mathbf{p}\mathbf{q}, \mathbf{x}}^{\mathrm{dR},2}$ proceed by considering them as homotopy inverses of some other $L_{\infty}[1]$-morphisms.

\noindent
(\textit{The map} $\widehat{\phi}_{\mathbf{p}\mathbf{q}, \mathbf{x}}^{\mathrm{dR},1}$) 
We first consider a family of $\mathbb{R}$-linear maps ${\widehat{\eta}}_{\mathbf{pq},\mathbf{x}} := \{{\widehat{\eta}}_{\mathbf{pq},\mathbf{x},k}\}_{k \geq 1},$
\begin{equation}\nonumber
{\widehat{\eta}}_{\mathbf{pq},\mathbf{x},k} : \ \Omega^{\bullet+1} \left( \mathcal{F}_{\mathbf{x}} \right)^{\otimes k}
\to \Omega_{\phi_{\mathbf{p}\mathbf{q}}}^{\bullet+1} \left({\mathcal{F}^{'0}_{\phi_{\mathbf{p}\mathbf{q}}(\mathbf{x})}} \right)
\end{equation}
defined by
\begin{equation}\nonumber
{\widehat{\eta}}_{\mathbf{pq}, \mathbf{x}, k}({\xi}_1, \dots, {\xi}_k) :=
\begin{cases}
1 \otimes \overline{{\xi}_1} &\text{ if } k=1, \\
(0,0) &\text{ if } k \geq 2,
\end{cases}
\end{equation}
where we denote 
\[
\overline{\xi} := \left( \phi_{\mathbf{p}\mathbf{q}}^{-1} \circ {\pi}_{{\mathbf{pq,x}}}\right)^* (\xi) \in \Omega^{\bullet+1}\left({\mathcal{F}}^{'0}_{\phi_{\mathbf{p}\mathbf{q}}(\mathbf{x})}\right).
\]

\begin{lem} 
${\widehat{\eta}}_{\mathbf{pq},\mathbf{x}} := \{{\widehat{\eta}}_{\mathbf{pq},\mathbf{x},k} \}_{k \geq 1}$ is an $L_\infty[1]$-quasi-isomorphism.
\end{lem}

\begin{proof}
Consider the following commutative diagram of bundles:
\begin{equation}\label{35diag}
\begin{tikzcd}
TT^*\mathcal{F}_{\mathbf{x}} \arrow{r}{((\phi^{-1}_{\mathbf{p}\mathbf{q}})^*)_*} \arrow{d} & T\left((\phi_{\mathbf{p}\mathbf{q}}^{-1})^*T^*\mathcal{F}_{\mathbf{x}}\right) \arrow[r, hook, "\widetilde{i}_*"] \arrow{d} & TT^*\mathcal{F}^{'0}_{\phi_{\mathbf{p}\mathbf{q}}(\mathbf{x})}|_{T^*\mathcal{F}^{'0}_{\phi_{\mathbf{p}\mathbf{q}}(\mathbf{x})}|_{\phi_{\mathbf{p}\mathbf{q}(W_{\mathbf{x}})}}} \arrow{r}{(\pi_{\mathbf{pq,x}}^*)_*} \arrow{d}{} &  TT^*\mathcal{F}^{'0}_{\phi_{\mathbf{p}\mathbf{q}}(\mathbf{x})}|_{W'_{\phi_{\mathbf{pq}}(\mathbf{x})}} \arrow{d}{}\\
T^*\mathcal{F}_{\mathbf{x}} \arrow{r}{(\phi_{\mathbf{pq}}^{-1})^*} \arrow{d} & (\phi_{\mathbf{pq}}^{-1})^*T^*\mathcal{F}_{\mathbf{x}} \arrow[r, hook, "\widetilde{i}"] \arrow{d} & {T^*\mathcal{F}^{'0}_{\phi_{\mathbf{p}\mathbf{q}}(\mathbf{x})}|_{\phi_{\mathbf{p}\mathbf{q}}(W_{\mathbf{x}})}} \arrow{r}{\pi_{\mathbf{pq,x}}^*} \arrow{d} & T^*\mathcal{F}^{'0}_{\phi_{\mathbf{p}\mathbf{q}}(\mathbf{x})}|_{W'_{\phi_{\mathbf{pq}}(\mathbf{x})}} \arrow{d} \\
W_{\mathbf{x}} \arrow{r}{\phi_{\mathbf{pq}}} & \phi_{\mathbf{p}\mathbf{q}}(W_{\mathbf{x}}) \arrow{r}{=} & \phi_{\mathbf{p}\mathbf{q}(W_{\mathbf{x}})} & W'_{\phi_{\mathbf{pq}}(\mathbf{x})} \arrow[swap]{l}{\pi_{\mathbf{pq,x}}},\\
\end{tikzcd}
\end{equation}
where $\widetilde{i}$ and $\widetilde{i}_*$ denote the obvious inclusion map and the map induced from it, respectively, obtained from Corollary \ref{fctfv}. Observe that from the above commuting diagram, we know that
\begin{equation}\nonumber
(\pi^*_{\mathbf{pq,x}})_* \circ \tilde{i}^* \circ \big((\phi_{\mathbf{p}\mathbf{q}}^{-1})^*\big)_*|_{T^*\mathcal{F}_{\mathbf{x}}} = \pi_{\mathbf{pq,x}}'^* \circ \widetilde{i} \circ (\phi_{\mathbf{p}\mathbf{q}}^{-1})^*.
\end{equation}

Furthermore, we obtain the following diagram for the corresponding V-algebras:
\begin{equation}\label{cdgr1}
\begin{aligned}
&\begin{tikzcd}
\lim\limits_{\longleftarrow}\frac{\Gamma(\bigwedge^{\bullet +1}TT^*\mathcal{F}_{\mathbf{x}}) }{I^n \cdot \Gamma(\bigwedge^{\bullet +1}TT^*\mathcal{F}_{\mathbf{x}})} \arrow{r}{((\phi^{-1}_{\mathbf{p}\mathbf{q}})^*)_*} \arrow{d}{\Pi} &\lim\limits_{\longleftarrow}\frac{\Gamma(\bigwedge^{\bullet +1}T(\phi_{\mathbf{pq}}^{-1})^*T^*\mathcal{F}_{\mathbf{x}}) }{I^n \cdot \Gamma(\bigwedge^{\bullet +1}T(\phi_{\mathbf{pq}}^{-1})^*T^*\mathcal{F}_{\mathbf{x}})} \cdots \arrow{d} \\
\Gamma\left(\bigwedge^{\bullet+1}T^*\mathcal{F}_{\mathbf{x}}\right)\arrow{r}{(\phi_{\mathbf{pq}}^{-1})^*} & \Gamma\left(\bigwedge^{\bullet +1}(\phi_{\mathbf{pq}}^{-1})^*T^*\mathcal{F}_{\mathbf{x}}\right) \cdots
\end{tikzcd}
\\
& \quad \quad \quad \begin{tikzcd}
{\cdots} \arrow[r, hook, "{\widetilde{i}_*}"] &\lim\limits_{\longleftarrow}\frac{\Gamma(\bigwedge^{\bullet +1}TT^*\mathcal{F}^{'0}_{\phi_{\mathbf{p}\mathbf{q}}(\mathbf{x})}|_{T^*\mathcal{F}^{'0}_{\phi_{\mathbf{p}\mathbf{q}}(\mathbf{x})}|_{\phi_{\mathbf{p}\mathbf{q}}(W_{\mathbf{x}})}})}{I^n \cdot \Gamma(\bigwedge^{\bullet +1}TT^*\mathcal{F}^{'0}_{\phi_{\mathbf{p}\mathbf{q}}(\mathbf{x})}|_{T^*\mathcal{F}^{'0}_{\phi_{\mathbf{p}\mathbf{q}}(\mathbf{x})}|_{\phi_{\mathbf{p}\mathbf{q}}(W_{\mathbf{x}})}})} \arrow{r}{(\pi_{\mathbf{pq,x}}^*)_*} \arrow{d}{} &  
\lim\limits_{\longleftarrow}\frac{\Gamma(\bigwedge^{\bullet +1}TT^*\mathcal{F}^{'0}_{W'_{\phi_{\mathbf{pq}}(\mathbf{x})}})}{I^n \cdot \Gamma(\bigwedge^{\bullet +1}TT^*\mathcal{F}^{'0}_{W'_{\phi_{\mathbf{pq}}(\mathbf{x})}})} \arrow{d}{\Pi'}\\
{\cdots} \arrow[r, hook,"{\widetilde{i}}"] & \Gamma\left(\bigwedge^{\bullet+1}T^*T^*\mathcal{F}^{'0}_{\phi_{\mathbf{p}\mathbf{q}}(\mathbf{x})}|_{T^*\mathcal{F}^{'0}_{\phi_{\mathbf{p}\mathbf{q}}(\mathbf{x})}|_{\phi_{\mathbf{p}\mathbf{q}}(W_{\mathbf{x}})}}\right) \arrow{r}{\pi_{\mathbf{pq,x}}^*} & \Gamma\left(\bigwedge^{\bullet+1}T^*\mathcal{F}^{'0}_{\phi_{\mathbf{p}\mathbf{q}}(\mathbf{x})}|_{W'_{\phi_{\mathbf{pq}}(\mathbf{x})}}\right).
\end{tikzcd}
\end{aligned}
\end{equation}
See Definition \ref{Voronov1} and (\ref{dasfda}) for the definition of V-algebras and the related notations. Here the top horizontal line of the graded Lie algebras is given by the fact that the maps $((\phi^{-1}_{\mathbf{p}\mathbf{q}})^*)_*, \widetilde{i}_*,$ and $(\pi_{\mathbf{pq,x}}^*)_*$  in (\ref{35diag}) are bundle maps. The bottom line consists of the abelian subalgebras. $I$'s are the ideals of the functions on the tangent bundles $TT^*\mathcal{F}_{\mathbf{x}}, T\left((\phi_{\mathbf{p}\mathbf{q}}^{-1})^*T^*\mathcal{F}_{\mathbf{x}}\right), TT^*\mathcal{F}^{'0}_{\phi_{\mathbf{p}\mathbf{q}}(\mathbf{x})}|_{T^*\mathcal{F}^{'0}_{\phi_{\mathbf{p}\mathbf{q}}(\mathbf{x})}|_{\phi_{\mathbf{p}\mathbf{q}(W_{\mathbf{x}})}}},$ and $TT^*\mathcal{F}^{'0}_{\phi_{\mathbf{p}\mathbf{q}}(\mathbf{x})}|_{W'_{\phi_{\mathbf{pq}}(\mathbf{x})}}$ that vanish on the zero-sections, respectively. We use the same symbol $I$ by abuse of notation.
 
The two Poisson structures 
\begin{equation}\nonumber
\begin{cases}
P &\in \lim\limits_{\longleftarrow}\frac{\Gamma(\bigwedge^{\bullet +1}TT^*\mathcal{F}_{\mathbf{x}}) }{I^n \cdot \Gamma(\bigwedge^{\bullet +1}TT^*\mathcal{F}_{\mathbf{x}})},\\
P^{'0} &\in 
\lim\limits_{\longleftarrow}\frac{\Gamma(\bigwedge^{\bullet +1}TT^*\mathcal{F}^{'0}_{W'_{\phi_{\mathbf{pq}}(\mathbf{x})}})}{I^n \cdot \Gamma(\bigwedge^{\bullet +1}TT^*\mathcal{F}^{'0}_{W'_{\phi_{\mathbf{pq}}(\mathbf{x})}})}
\end{cases}
\end{equation}
are induced from the presymplectic structures on 
\[
(W_{\mathbf{x}}, \omega_{\mathbf{p},W_{\mathbf{x}}})
\]
and 
\[
\left(W'_{\phi_{\mathbf{pq}}(\mathbf{x})}, (\phi^{-1}_{\mathbf{pq,x}} \circ \pi_{\mathbf{pq,x}})^*(\omega_{\mathbf{p},W_{\mathbf{x}}})\right)
\] 
as in (\ref{pico}), respectively.

For the $L_{\infty}$-relation, we need to show that
\begin{equation}\nonumber
\begin{split}
l_{\phi_{\mathbf{pq}},k}^{'\mathcal{F}^{'0}}\big(\eta_{\phi_{\mathbf{pq}}(\mathbf{x}),1}(\xi_1), \cdots,& \eta_{\phi_{\mathbf{pq}}(\mathbf{x}),1}({\xi}_k)\big) = l_{\phi_{\mathbf{pq}},k}^{\mathcal{F}^{'0}}(1 \otimes \overline{{\xi}_1}, \cdots, 1 \otimes \overline{{\xi}_k})
\\ 
= 1 \otimes l_k^{'\mathcal{F}^{'0}}(\overline{{\xi}_1}, \cdots, \overline{{\xi}_k})& \overset{*}{=} 1 \otimes \overline{l_k^{\mathcal{F}}({\xi}_1, \dots, {\xi}_k)} = \widehat{\eta}_{\mathbf{x},1}\big(l_k^{\mathcal{F}}(\overline{\xi_1}, \dots, \overline{\xi_k})\big),
\end{split}
\end{equation}
which follows once we verify that
\begin{equation}\label{lkfol}
l_k'^{\mathcal{F}^{'0}}(\overline{{\xi}_1}, \dots, \overline{{\xi}_k}) \overset{*}{=} \overline{l_k^{\mathcal{F}}({\xi}_1, \dots, {\xi}_k)}
\end{equation}
holds. Here $\left\{l_{\phi_{\mathbf{pq}},k}^{\mathcal{F}^{'0}}\right\}$ is the $L_{\infty}[1]$-structure on the completion discussed in Definition~\ref{ladef}. 

\begin{claim}
(\ref{lkfol}) holds.
\end{claim}

\begin{proof}
Applying the formula (\ref{dlpk}), we obtain:

\begin{equation}\nonumber
\begin{split}
{l'}_k^{\mathcal{F}^{'0}}(\overline{\xi_1}, \dots, \overline{\xi_k}) =& \Pi^{'} \left[ \cdots \left[P^{'0}, \big((\pi^*_{\mathbf{pq,x}})_* \circ \tilde{i}^* \circ (\phi_{\mathbf{p}\mathbf{q}}^{-1})^*\big)_*({\xi}_1)\right], \cdots, \big((\pi^*_{\mathbf{pq,x}})_* \circ \tilde{i}^* \circ (\phi^{-1}_{\mathbf{p}\mathbf{q}})^*\big)_*({\xi}_k)\right]\\
\overset{(1)}{=} & {\Pi}^{'} \bigg[ \cdots \left[ (\pi_{\mathbf{pq,x}}^*)_*(P^{'0}|_{T^*\mathcal{F}^{'0}_{\phi_{\mathbf{p}\mathbf{q}}(\mathbf{x})}|_{\phi_{\mathbf{p}\mathbf{q}}(\mathbf{x})}}),\big((\pi^*_{\mathbf{pq,x}})_* \circ \tilde{i}^* \circ (\phi_{\mathbf{p}\mathbf{q}}^{-1})^*\big)_*({\xi}_1)\right],\\
&\quad \quad \quad \quad \quad \quad \quad \quad \quad \quad \quad \quad \quad \quad\quad \quad \cdots, \big((\pi_{\mathbf{pq,x}}^*)_* \circ \tilde{i}^* \circ (\phi_{\mathbf{p}\mathbf{q}}^{-1})^*\big)_*({\xi}_k)\bigg]\\
\overset{(2)}{=} & {\Pi}^{'} \Big[\cdots \left[ \big((\pi^*_{\mathbf{pq,x}})_* \circ \tilde{i}^* \circ (\phi_{\mathbf{p}\mathbf{q}}^{-1})^*\big)_*P,  \left((\pi^*_{\mathbf{pq,x}})_* \circ \tilde{i}^* \circ (\phi_{\mathbf{p}\mathbf{q}}^{-1})^*\right)_*({\xi}_1)\right],\\
&\quad \quad \quad \quad \quad \quad \quad \quad \quad \quad \quad \quad \quad \quad\quad \quad  \cdots, \big((\pi^*_{\mathbf{pq,x}})_* \circ \tilde{i}^* \circ (\phi^{-1}_{\mathbf{p}\mathbf{q}})^*\big)_*({\xi}_k)\Big]\\
\overset{(3)}{=} & {\Pi}^{'} \left((\pi_{\mathbf{pq,x}}^*)_* \circ \tilde{i}^* \circ (\phi^{-1}_{\mathbf{p}\mathbf{q}})^*\right)_* \left[ \cdots [ P, {\xi}_1], \dots, {\xi}_k \right] \\
\overset{(4)}{=} & (\pi^*_{\mathbf{pq,x}})_* \circ \tilde{i}^* \circ (\phi^{-1}_{\mathbf{p}\mathbf{q}})^* \circ {\Pi}[ \cdots [ P, {\xi}_1], \dots, {\xi}_k ] \\
= & (\pi^*_{\mathbf{pq,x}})_* \circ \tilde{i}^* \circ (\phi^{-1}_{\mathbf{p}\mathbf{q}})^*  \big( l^{\mathcal{F}}_k({\xi}_1, \dots, {\xi}_k)\big) = \overline{l_k^{\mathcal{F}}({\xi}_1, \dots, {\xi}_k)}.
\end{split}
\end{equation}

We explain how we obtain the equalities $(1)$ through $(4)$:

\begin{itemize}
  \item[(1)] All $((\pi_{\mathbf{pq,x}}^* \circ \tilde{i}^* \circ (\phi_{\mathbf{p}\mathbf{q}}^{-1})^*)_*({\xi}_i)$'s are constant in the fiber direction.
  \item[(2)] From the formula (\ref{pico}), it is not difficult to show that the two Poisson structures are related by:
\[
 P^{'0}|_{T^*\mathcal{F}^{'0}|_{\phi_{\mathbf{p}\mathbf{q}}(W_{\mathbf{x}})}}
  =  (\tilde{i}^*)_{*} \circ \big((\phi_{\mathbf{p}\mathbf{q}}^{-1})^*\big)_*(P) + \overbrace{\left( \sum\limits_{\gamma'} \frac{\partial}{\partial q'_{\gamma'}} \wedge \frac{\partial}{\partial p^{'\gamma'}}  \right),}^{\text{fiber direction components}}
\]
and for the same reason as (1), the repeated bracket vanishes for the components $\sum\limits_{\gamma'} \frac{\partial}{\partial q'_{\gamma'}} \wedge \frac{\partial}{\partial p^{'\gamma'}}$ in the fiber direction. 
  \item[(3)] The Nijenhuis–Schouten bracket commutes with pushforwards.
  \item[(4)] From the commutative diagram $(\ref{cdgr1})$, we have
\[
  \Pi' \circ \big( \pi_{\mathbf{pq,x}}^* \circ \tilde{i} \circ (\phi_{\mathbf{p}\mathbf{q}}^{-1})^*\big)_* = \pi_{\mathbf{pq,x}}^* \circ \tilde{i} \circ (\phi_{\mathbf{p}\mathbf{q}}^{-1})^* \circ \Pi.
\]
\end{itemize}
\end{proof}
Note that $\eta_{\mathbf{pq},\mathbf{x},k}$ is quasi-isomorphic, as both the domain and the target are acyclic (cf. Proposition \ref{augomega}). This proves that $\{\widehat{\eta}_{\mathbf{pq},\mathbf{x},k} \}$ is an $L_\infty[1]$-quasi-isomorphism. 
\end{proof}

We then define
\begin{equation}\nonumber
\widehat{\phi}_{\mathbf{p}\mathbf{q}, \mathbf{x}}^{\mathrm{dR},1} :  \Omega_{\mathrm{aug},\phi_{\mathbf{p}\mathbf{q}}}^{\bullet+1} \left({\mathcal{F}^{'0}_{\phi_{\mathbf{p}\mathbf{q}}}}|_{W_{\mathbf{x}}} \right) \rightarrow \Omega_{\mathrm{aug}}^{\bullet+1} \left( \mathcal{F}_{\mathbf{x}} \right)
\end{equation}
by
\begin{equation}\nonumber
\widehat{\phi}^{\mathrm{dR},1}_{\mathbf{pq}, \mathbf{x}} := \text{ a homotopy inverse of } \widehat{\eta}_{\mathbf{pq},\mathbf{x}},
\end{equation}
using Theorem \ref{wht}.

\subsection{Family of presymplectic forms}

To obtain the map $\widehat{\phi}^{\mathrm{dR},2}_{\mathbf{pq}, \mathbf{x}}$, we need to connect 

\[
\omega'_{\mathbf{q}, W'_{\phi_{\mathbf{pq}}(\mathbf{x})}} 
:= \pi_{i'}^{\prime *}\left(\omega_{\mathbf{q}}'\big|_{\overset{\circ}{W}'_{\phi_{\mathbf{pq}}(\mathbf{x})}}\right)  \in \Omega^2\left(W'_{\phi_{\mathbf{pq}}(\mathbf{x})}\right)
\]
and 
\[
\pi_{\mathbf{pq,x}}^* \circ (\phi_{\mathbf{pq}}^{-1})^*(\omega_{\mathbf{p},W_{\mathbf{x}}}) := \pi_{\mathbf{pq,x}}^* \circ (\phi_{\mathbf{pq}}^{-1})^* \circ \pi_i^*\left( \omega_{\mathbf{p}}|_{\overset{\circ}{W_{\mathbf{x}}}} \right) \in \Omega^2\left(W'_{\phi_{\mathbf{pq}}(\mathbf{x})}\right)
\]
with a family of presymplectic forms, so that Corollary \ref{indliso} yields the desired result.

\noindent
(\textit{The map} $\widehat{\phi}_{\mathbf{p}\mathbf{q}, \mathbf{x}}^{\mathrm{dR},2}$) Our plan is to consider concatenation of the families of presymplectic forms on $W'_{\phi_{\mathbf{pq}}(\mathbf{x})}$, (A), (B), and (C):
\[
\begin{tikzcd}
\omega'_{\mathbf{q}, W'_{\phi_{\mathbf{pq}}(\mathbf{x})}} \arrow[r, "(\mathrm{A})", rightsquigarrow]  & \pi_i^{' *} \circ \overset{\circ}{\pi}_{\mathbf{pq,x}}^{'*} 
\left(\omega_{\mathbf{q}}' \big|_{\overset{\circ}{W}'_{\phi_{\mathbf{pq}}(\mathbf{x})} \cap \phi_{\mathbf{pq}}(\overset{\circ}{W}_{\mathbf{x}})}\right) \arrow[d, "(\mathrm{B})", rightsquigarrow] \\
\pi_{\mathbf{pq,x}}^* \circ (\phi_{\mathbf{pq}}^{-1})^*(\omega_{\mathbf{p},W_{\mathbf{x}}}) & \pi_{\mathbf{pq,x}}^* \circ (\phi_{\mathbf{pq}}^{-1})^* \circ \pi_i^{*} \circ \overset{\circ}{\pi}_{\mathbf{pq,x}}^* 
\left(\omega_{\mathbf{p}} \big|_{\phi_{\mathbf{pq}}^{-1}(\overset{\circ}{W}'_{\phi_{\mathbf{pq}}(\mathbf{x})} \cap \phi_{\mathbf{pq}}(\overset{\circ}{W}_{\mathbf{x}}))}\right), \arrow[l, "(\mathrm{C})", rightsquigarrow]  
\end{tikzcd}
\]
where 
\begin{equation}\nonumber
\begin{cases}
\pi_{\mathbf{pq,x}} : W'_{\phi_{\mathbf{pq}}(\mathbf{x})} \twoheadrightarrow \phi_{\mathbf{pq}} \left(W_{\mathbf{x}}\right),\\
\overset{\circ}{\pi}_{\mathbf{pq,x}}^{'*} :  \overset{\circ}{W}'_{\phi_{\mathbf{pq}}(\mathbf{x})} \twoheadrightarrow {\overset{\circ}{W}'_{\phi_{\mathbf{pq}}(\mathbf{x})} \cap \phi_{\mathbf{pq}}(\overset{\circ}{W}_{\mathbf{x}})}\\
\overset{\circ}{\pi}_{\mathbf{pq,x}} : \overset{\circ}{W}_{\mathbf{x}} \twoheadrightarrow \phi_{\mathbf{pq}}^{-1}(\overset{\circ}{W}'_{\mathbf{pq}(\mathbf{x})} \cap \phi_{\mathbf{pq}}(\overset{\circ}{W}_{\mathbf{x}}))
\end{cases}
\end{equation}
denote choices of projections between the contractible spaces, and
\begin{equation}\nonumber
\begin{cases}
\pi_{i} : N_i \rightarrow \mathcal{S}_{\mathbf{p},i},\\
\pi'_{i} : N'_{i'} \rightarrow \mathcal{S}'_{\mathbf{q},i'}
\end{cases}
\end{equation}
are the associated projection for the tubular neighborhoods.

\begin{lem}\label{eajekl}
The two-forms
\[
\pi_i^{' *} \circ \overset{\circ}{\pi}_{\mathbf{pq,x}}^{'*} 
\left(\omega_{\mathbf{q}}' \big|_{\overset{\circ}{W}'_{\phi_{\mathbf{pq}}(\mathbf{x})} \cap \phi_{\mathbf{pq}}(\overset{\circ}{W}_{\mathbf{x}})}\right)   \in \Omega^2\left(W'_{\phi_{\mathbf{pq}}(\mathbf{x})}\right)
\]
and
\[
\pi_{\mathbf{pq,x}}^* \circ (\phi_{\mathbf{pq}}^{-1})^* \circ \pi_i^{*} \circ \overset{\circ}{\pi}_{\mathbf{pq,x}}^* \left(\omega_{\mathbf{p}} \big|_{\phi_{\mathbf{pq}}^{-1}(\overset{\circ}{W}'_{\mathbf{pq}(\mathbf{x})} \cap \phi_{\mathbf{pq}}(\overset{\circ}{W}_{\mathbf{x}}))}\right)   \in \Omega^2\left(W'_{\phi_{\mathbf{pq}}(\mathbf{x})}\right)
\]
are presymplectic.
\end{lem}

\begin{proof}
Writing the restrictions as pullbacks under the embeddings, we see that the closedness is obvious. The constant rank conditions is a consequence of the restrictions to (subspaces) of the strata   

\[
\phi_{\mathbf{pq}}^{-1}(\overset{\circ}{W}'_{\mathbf{pq}(\mathbf{x})} \cap \phi_{\mathbf{pq}}(\overset{\circ}{W}_{\mathbf{x}})) \subset \overset{\circ}{W}_{\mathbf{x}} \subset \mathcal{S}_{\mathbf{p},i} \text{ (for some }i),
\]

\[
\overset{\circ}{W}'_{\phi_{\mathbf{pq}}(\mathbf{x})} \cap \phi_{\mathbf{pq}}(\overset{\circ}{W}_{\mathbf{x}}) \subset \overset{\circ}{W}'_{\phi_{\mathbf{pq}}(\mathbf{x})} \subset \mathcal{S}'_{\mathbf{q},i'} \text{ (for some }i').
\]
\end{proof}

The family (A) is a 1-parameter family of presymplectic forms on $W'_{\phi_{\mathbf{pq}}(\mathbf{x})}$ of the same rank (cf. Lemma \ref{lemdw}). Similarly, (C) is also such a family. Then by Corollary \ref{indliso}, we obtain $L_{\infty}[1]$-isomorphisms
\begin{equation}\label{ggdef}
\begin{split}
\widehat{\gamma}'_{\mathbf{pq},\mathbf{x}} &: \Omega^{\bullet+1}(\mathcal{F}'_{\phi_{\mathbf{pq}}(\mathbf{x})}(\omega'_{\mathbf{q},W_{\phi_{\mathbf{pq}}(\mathbf{x})}})) 
\xrightarrow{\simeq}
\Omega^{\bullet + 1}\left(\mathcal{F}'_{\phi_{\mathbf{pq}}(\mathbf{x})}\left( \pi_i^{'*} \circ \overset{\circ}{\pi}^{'*}_{\mathbf{pq},\mathbf{x}}\left(\omega_{\mathbf{q}}' \big|_{\overset{\circ}{W}'_{\phi_{\mathbf{pq}}(\mathbf{x})} \cap \phi_{\mathbf{pq}}(W_{\mathbf{x}})}\right)\right)\right),\\
\widehat{\gamma}_{\mathbf{pq},\mathbf{x}} &: \Omega^{\bullet+1}(\mathcal{F}'_{\phi_{\mathbf{pq}(\mathbf{x})}}(\pi_{\mathbf{pq,x}}^* \circ (\phi_{\mathbf{pq}}^{-1})^*(\omega_{\mathbf{p},W_{\mathbf{x}}}))\\
& \quad \quad \quad \quad \xrightarrow{\simeq}
\Omega^{\bullet+1}\left(\mathcal{F}'_{\phi_{\mathbf{pq}}(\mathbf{x})}\left(\pi_{\mathbf{pq,x}}^* \circ (\phi_{\mathbf{pq}}^{-1})^* \circ \pi_i^{*} \circ \overset{\circ}{\pi}_{\mathbf{pq,x}}^* 
\left(\omega_{\mathbf{p}} \big|_{\phi_{\mathbf{pq}}^{-1}(\overset{\circ}{W}'_{\phi_{\mathbf{pq}}(\mathbf{x})} \cap \phi_{\mathbf{pq}}(\overset{\circ}{W}_{\mathbf{x}}))}\right)\right)\right),
\end{split}
\end{equation}
where $\mathcal{F}_{\phi_{\mathbf{pq}}(\mathbf{x})}'(\cdots)$ stands for the foliation arising from the presymplectic form $(\cdots).$
 
For the family (B), we recall the presymplectic version of the Darboux theorem, whose proof can be found in Theorem 2.1 of \cite{GLRR}, for example.
\begin{thm}[Presymplectic Darboux theorem]  
Let $(M,\omega)$ be a presymplectic manifold of dimension $2m+k$ and of rank $2m$.  
Then there exists a local coordinate system at each point of $M$
\[
\{x_1,\ldots,x_m, \; x^{'1},\ldots,x^{'m}, \; q^1,\ldots,q^k\}
\]
such that $\omega$ is written as
\[
\omega = \sum_{i=1}^m dx_i \wedge dx^{'
i}.
\]
In this system, the kernel of $\omega$ is spanned as
\[
\ker \omega = \mathrm{span}\left\{ \frac{\partial}{\partial q^1}, \ldots, \frac{\partial}{\partial q^k} \right\}.
\]
\end{thm}
The special choice of coordinates in the preceding theorem affect the $L_{\infty}[1]$-algebras that they determine; however, they only make isomorphic changes by Lemma \ref{liiifdr} (v).

Note that we can take \textit{the same} Darboux coordinates for the presymplectic forms in Lemma \ref{eajekl}
\[
\pi_i^{' *} \circ \overset{\circ}{\pi}_{\mathbf{pq,x}}^{'*} 
\left(\omega_{\mathbf{q}}' \big|_{\overset{\circ}{W}'_{\mathbf{pq}(\mathbf{x})} \cap \phi_{\mathbf{pq}}(\overset{\circ}{W}_{\mathbf{x}})}\right) \in \Omega^2(W'_{\phi_\mathbf{pq}(\mathbf{x})})
\] 
and
\[
\pi_{\mathbf{pq,x}}^* \circ (\phi_{\mathbf{pq}}^{-1})^* \circ \pi_i^{*} \circ \overset{\circ}{\pi}_{\mathbf{pq,x}}^* 
\left(\omega_{\mathbf{p}} \big|_{\phi_{\mathbf{pq}}^{-1}(\overset{\circ}{W}'_{\mathbf{pq}(\mathbf{x})} \cap \phi_{\mathbf{pq}}(\overset{\circ}{W}_{\mathbf{x}}))}\right) \in \Omega^2(W'_{\phi_\mathbf{pq}(\mathbf{x})}).
\]
by taking sufficiently small $\overset{\circ}{W}_{\mathbf{x}}$ and  $\overset{\circ}{W}'_{\mathbf{pq}(\mathbf{x})}$ if necessary, so that they are expressed in the following forms of skew symmetric matrices, respectively:
\[
\begin{split}
\pi_i^{' *} \circ \overset{\circ}{\pi}_{\mathbf{pq,x}}^{'*} 
\left(\omega_{\mathbf{q}}' \big|_{\overset{\circ}{W}'_{\mathbf{pq}(\mathbf{x})} \cap \phi_{\mathbf{pq}}(\overset{\circ}{W}_{\mathbf{x}})}\right) &\text{ in the Darboux coordinates }\\
&=\left(\begin{array}{ c | c | c| c}
S_{2m} & O_{k \times 2m} & O_{2m' \times 2m} & O_{k' \times 2m} \\
\hline
O_{2m \times k} & O_{k \times k} & O_{m' \times m} & O_{k' \times k}  \\
\hline
O_{2m \times 2m'} & O_{k \times 2m'} & S_{2m'} & O_{k' \times 2m'} \\
\hline
O_{2m' \times k'} & O_{k \times k'}  & O_{2m' \times k'}& O_{k' \times k'}
\end{array}\right).\\
\end{split}
\]
\[
\begin{split}
\pi_{\mathbf{pq,x}}^* \circ (\phi_{\mathbf{pq}}^{-1})^* \circ \pi_i^{*} \circ \overset{\circ}{\pi}_{\mathbf{pq,x}}^* 
&\left(\omega_{\mathbf{p}} \big|_{\phi_{\mathbf{pq}}^{-1}(\overset{\circ}{W}'_{\mathbf{pq}(\mathbf{x})} \cap \phi_{\mathbf{pq}}(\overset{\circ}{W}_{\mathbf{x}}))}\right) \text{ in the Darboux coordinates } \\
&\quad \quad \quad \quad = \left(\begin{array}{ c | c | c| c}
S_{2m} & O_{k \times 2m} & O_{2m' \times 2m}& O_{k' \times 2m} \\
\hline
O_{2m \times k} & O_{k \times k} & O_{m' \times m} & O_{k' \times k}  \\
\hline
O_{2m \times 2m'} & O_{k \times 2m'} & O_{2m'} & O_{k' \times 2m'} \\
\hline
O_{2m' \times k'} & O_{k \times k'}  & O_{2m' \times k'}& O_{k' \times k'}
 \end{array}\right),
\end{split}
\]
where, for each $\ell \geq 0,$ we denote
\[
S_{2\ell} :=
\begin{pmatrix}
O_{\ell \times \ell} & -I_{\ell} \\
I_{\ell} & O_{\ell \times \ell}
\end{pmatrix}
\]
with $O_{\ell \times \ell'}$ and $I_{\ell}$ being the $\ell \times \ell'$ zero matrix and the $\ell \times \ell$ identity matrix, respectively. 

To join them with a family of presymplectic forms, we need the following theorem:

\begin{thm}\cite[Theorem 3.4]{HW}\label{hwthm}
Let $M$ be a closed manifold. Assume that presymplectic two-forms $\omega_0, \omega_1 \in \Omega^2(M)$ are joined by a path $\{\omega_t\}_{t\in [0,1]}$ of nondegenerate two-forms. Then $\omega_0$ and $\omega_1$ are homotopic through presymplectic forms.
\end{thm}

\begin{proof}[Proof-sketch]
We use the fact that there exists a homotopy equivalence
\[
S_{\mathrm{presymp}}(M,a) \hookrightarrow S_{\mathrm{nondeg}}(M)
\]
from the space of presymplectic forms on $M$ of fixed cohomology type $a$ to the space of nondegenerate two-forms. For the more details, see \cite{HW}.
\end{proof}

Since all the entries of the above-mentioned matrices are filled with constant functions, we can extend them to the closure of the open ball $W'_{\phi_{\mathbf{pq}(\mathbf{x})}}.$ We can now apply the preceding theorem to our situation.

At the moment, we assume that both $k$ and $k'$ are even numbers. We have
\[
\widetilde{\delta}_1(t) := \left(
\begin{array}{ c | c | c | c}
S_{2m} & O_{k \times m} &  O_{2m' \times 2m} & O_{k' \times 2m} \\
\hline
O_{m \times k} & (1-t)g \cdot S_{2k} & O_{m' \times m} & O_{k' \times k}  \\
\hline
 O_{2m \times 2m'} & O_{k \times 2m'} & S_{2m'} & O_{k' \times 2m'} \\
\hline
O_{2m' \times k'} & O_{k \times k'}  & O_{2m' \times k'}& (1-t)g' \cdot S_{2k'}
  \end{array} 
\right),
\]
and
\[
\widetilde{\delta}_2(t) =  \left(\begin{array}{ c | c | c| c}
S_{2m} & O_{k \times m} & O_{2m' \times 2m} & O_{k' \times 2m} \\
\hline
O_{m \times k} & (1-t)g \cdot S_{2k} & O_{m' \times m} & O_{k' \times k}  \\
\hline
 O_{2m \times 2m'} & O_{k \times 2m'} & (1-t)g \cdot S_{2m'} & O_{k' \times 2m'} \\
\hline
O_{2m' \times k'} & O_{k \times k'}  & O_{2m' \times k'}& (1-t)g' \cdot S_{2k'}
  \end{array}\right)
\]
for $t \in [0,1]$ and some positive functions $g, g'$ in $C^{\infty}(W'_{\phi_{\mathbf{pq}}(\mathbf{x})}; \mathbb{R}_{>0}).$

We now observe that the matrices $ \widetilde{\delta}_1(t)$ and $ \widetilde{\delta}_2(t)$ are nondegenerate at each $0 < t < 1,$ that is, $\det \widetilde{\delta}_1(t), \; \det \widetilde{\delta}_2(t) \neq 0, 
\quad 0 < t < 1,$
and consider their concatenation
\[
\widetilde{\delta}(t) := \widetilde{\delta}_1(t) \# \widetilde{\delta}_2(t) =  
\begin{cases}
\widetilde{\omega}'_1(1 - 2t), & 0 \leq t \leq \tfrac{1}{2}, \\[6pt]
\widetilde{\omega}'_2(2t - 1), & \tfrac{1}{2} \leq t \leq 1,
\end{cases}
\]
so that they determines a family of nondegenerate two-forms. Note that $\widetilde{\delta}(t)$ is not smooth in general at $t = \tfrac{1}{2}$. However,
we can locally deform it near $t = \tfrac{1}{2}$ to a smooth path,
while preserving the nondegeneracy condition at each $t$, by the fact that nondegeneracy is an open condition. We write $\widetilde{\omega}'(t)$ for the resulting smooth family. 

For the case when $k'$ is odd and $k$ is even we can use the following path:
\begin{equation}\label{mmqq1}
\left(\begin{array}{c|c}
\begin{array}{ c | c | c| c}
S_{2m} & O_{k \times m} & O_{2m' \times 2m} & O_{(k'-1) \times 2m} \\
\hline
O_{m \times k} & (1-t)g \cdot S_{2k} & O_{m' \times m} & O_{(k'-1) \times k}  \\
\hline
O_{2m \times 2m'} & O_{k \times 2m'} & \cdots & O_{(k'-1) \times 2m'} \\
\hline
O_{2m' \times (k'-1)} & O_{k \times (k'-1)}  & O_{2m' \times (k'-1)}& (1-t)g' \cdot S_{2(k'-1)}
  \end{array} & O_{1 \times (2(m + m') + k + k' -1)}\\
\hline
O_{(2(m + m') + k + k' -1) \times 1}& O_{1 \times 1}
\end{array}\right),
\end{equation}
where the upper left $(2(m + m') + k + k' -1) \times (2(m + m') + k + k' -1)$ block is nondegenerate. It is possible to apply the same method to connect the two presymplectic forms using the family (\ref{mmqq1}) without changing the other blocks. The other two cases can be treated in exactly the same way, so we omit them.

We then obtain the family (B) by the following corollary.

\begin{cor}
There exists a smooth family of presymplectic forms denoted by
\[
\left\{\widetilde{\omega}'(t)\right\}_{t \in [0,1]}
\]
that connects 
\[
\widetilde{\omega}'(0) := \pi_i^{' *} \circ \overset{\circ}{\pi}_{\mathbf{pq,x}}^{'*} 
\left(\omega_{\mathbf{q}}' \big|_{\overset{\circ}{W}'_{\mathbf{pq}(\mathbf{x})} \cap \phi_{\mathbf{pq}}(\overset{\circ}{W}_{\mathbf{x}})}\right) \in \Omega^2(W'_{\phi_\mathbf{pq}(\mathbf{x})}), 
\]
and
\[
\widetilde{\omega}'(1) := \pi_{\mathbf{pq,x}}^* \circ (\phi_{\mathbf{pq}}^{-1})^* \circ \pi_i^{*} \circ \overset{\circ}{\pi}_{\mathbf{pq,x}}^* 
\left(\omega_{\mathbf{p}} \big|_{\phi_{\mathbf{pq}}^{-1}(\overset{\circ}{W}'_{\phi_{\mathbf{pq}}(\mathbf{x})} \cap \phi_{\mathbf{pq}}(\overset{\circ}{W}_{\mathbf{x}}))}\right) \in \Omega^2(W'_{\phi_\mathbf{pq}(\mathbf{x})}).
\]
Moreover, we can take the number of times that the rank of $\widetilde{\omega}'(t)$ changes on $[0,1]$ to be finite.
\end{cor}

\begin{proof}
By the discussion in the previous paragraph and Theorem \ref{hwthm}, we only need to verify the finiteness, which is straightforward once we notice that if the given path for some $g$ and $g'$ does not give us the desired finiteness, then we can choose $g = g' =1,$ which makes the family $\left\{\widetilde{\omega}^{'}(t)\right\}_{t \in (0,1)}$ consist of closed nondegenerate two-forms, that is, symplectic forms for the case of even $k$ and $k'.$ Even when either of $k$ or $k'$ is an odd number, we can use the same trick for the upper left block of the matrix (\ref{mmqq1}).
\end{proof}

Denote by $t_0 = 0 < t_1 < \cdots < t_N < t_{N+1} = 1$ the numbers where the nullity of $\widetilde{\omega}'(t)$ (in the preceding corollary) jumps. If we draw a graph of nullity $\widetilde{\omega}(t)$ versus $t \in [0,1],$ its shape over each interval $[t_i, t_{i+1}]$ falls into one of the following four types by the upper semi-continuity of the nullity function.

\begin{figure}[h]
    \centering
\begin{tikzpicture}[>=stealth,scale=1.2]

\node at (-0.05,3) {$(1)\ \mathrm{rk}(\ker \widetilde{\omega}'(t))$};

\draw[->,thick] (-0.5,0) -- (-0.5,2.8);
\draw[thick] (1,1.1) -- (1.4,1.1);
\draw[thick] (0,0.8) -- (-0.4,0.8);

\draw[dashed] (0,0.1) -- (0,2.7);
\draw[dashed] (1,0.1) -- (1,2.7);

\node at (0,-0.3) {$t_i$};
\node at (1,-0.3) {$t_{i+1}$};

\draw[thick] (0,1.5) -- (1,1.5);
\filldraw (0,1.5) circle (2pt);
\filldraw (1,1.5) circle (2pt);

\draw[->,line width=2pt,yellow!40!green] (0.05,1.8) -- (0.95,1.8);

\draw (0,0.8) circle (2pt);
\draw (1,1.1) circle (2pt);

\node at (2.45,3) {$(2)\ \mathrm{rk}(\ker \widetilde{\omega}'(t))$};

\draw[->,thick] (2,0) -- (2,2.8);
\draw[thick] (3.5,2.1) -- (3.9,2.1);
\draw[thick] (2.5,0.8) -- (2.1,0.8);

\draw[dashed] (2.5,0.1) -- (2.5,2.7);
\draw[dashed] (3.5,0.1) -- (3.5,2.7);

\node at (2.5,-0.3) {$t_i$};
\node at (3.5,-0.3) {$t_{i+1}$};

\draw[thick] (2.5,1.5) -- (3.5,1.5);
\filldraw (2.5,1.5) circle (2pt);
\filldraw (3.5,2.1) circle (2pt);

\draw[->,line width=2pt,yellow!40!green] (2.55,1.8) -- (3.45,1.8);

\draw (2.5,0.8) circle (2pt);
\draw (3.5,1.5) circle (2pt);

\node at (4.95,3) {$(3)\ \mathrm{rk}(\ker \widetilde{\omega}'(t))$};

\draw[->,thick] (4.5,0) -- (4.5,2.8);
\draw[thick] (6,0.8) -- (6.4,0.8);
\draw[thick] (5,2.3) -- (4.6,2.3);

\draw[dashed] (5,0.1) -- (5,2.7);
\draw[dashed] (6,0.1) -- (6,2.7);

\node at (5,-0.3) {$t_i$};
\node at (6,-0.3) {$t_{i+1}$};

\draw[thick] (5,1.3) -- (6,1.3);
\filldraw (5,2.3) circle (2pt);
\filldraw (6,1.3) circle (2pt);

\draw[->,line width=2pt,yellow!40!green] (5.95,1.5) -- (5.05,1.5);

\draw[->,line width=2pt,red!40!] (5.05,1.8) -- (5.95,1.8);

\node at (5.45,2) {$\textrm{h-inv.}$};

\draw (5,1.3) circle (2pt);
\draw (6,0.8) circle (2pt);

\node at (7.45,3) {$(4)\ \mathrm{rk}(\ker \widetilde{\omega}'(t))$};

\draw[->,thick] (7,0) -- (7,2.8);
\draw[thick] (8.75,1.5) -- (9.15,1.5);
\draw[thick] (7.5,2.3) -- (7.1,2.3);

\draw[dashed] (7.5,0.1) -- (7.5,2.7);
\draw[dashed] (8.75,0.1) -- (8.75,2.7);

\node at (7.5,-0.3) {$t_i$};
\node at (8.5,-0.3) {$t_{i+1}$};
\node at (8.2, 0.05) {$\frac{t_{i} + t_{i+1}}{2}$};

\draw[thick] (7.5,0.8) -- (8.75,0.8);
\filldraw (7.5,2.3) circle (2pt);
\filldraw (8.75,1.5) circle (2pt);

\draw[->,line width=2pt,yellow!40!green] (8.075,1.15) -- (7.55,1.15);
\draw[->,line width=2pt,yellow!40!green] (8.175,1.15) -- (8.69,1.15);

\draw[line width=1pt,yellow!40!green] (8.125,0.3) -- (8.125,2.4);

\draw[->,line width=2pt,red!40!] (7.55,1.65) -- (8.075,1.65);

\node at (7.7,1.85) {$\textrm{h-inv.}$};

\draw (7.5,0.8) circle (2pt);
\draw (8.75,0.8) circle (2pt);

\end{tikzpicture}

\caption{4 types of the nullity graphs over [$t_i, t_{i+1}$]}

\end{figure}
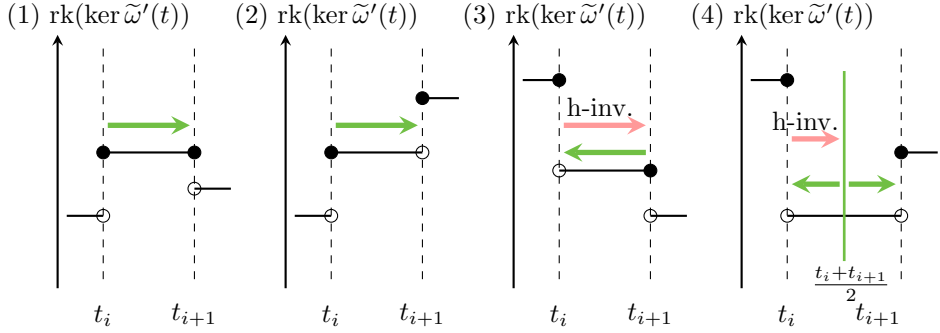
\noindent In Figure 1, the green arrows represents the direction that the $L_{\infty}[1]$-morphisms are constructed, while the pink ones mean that we take homotopy inverses of the corresponding quasi-isomorphisms. 

We denote
\begin{equation}\nonumber
A_{(\ell)} : = \{ 0 < i < N \mid [t_i, t_{i+1}] \text{ corresponds to the case } (\ell) \}
\end{equation}
for $\ell \in \{1,2,3,4\}.$
\begin{enumerate}
\item  If $i \in A_{(\ell)},$ the family $\big\{\widetilde{\omega}'(t)\big\}_{t \in [t_i, t_{i+1}]}$ determines an $L_{\infty}[1]$-isomorphism
\[
\widehat{\kappa}_{\mathbf{pq},\mathbf{x}}^{(i)} : \Omega^{\bullet +1}(\mathcal{F}'_{t_i,{\phi_{\mathbf{p}\mathbf{q}}(\mathbf{x})}})_{\phi_{\mathbf{p}\mathbf{q}}} \rightarrow \Omega^{\bullet +1}(\mathcal{F}'_{t_{i+1},{\phi_{\mathbf{p}\mathbf{q}}(\mathbf{x})}})_{\phi_{\mathbf{p}\mathbf{q}}}
\]
by Corollary \ref{indliso}, hence its augmented version (obtained by Lemma \ref{auglmo}) as well
\[
\widehat{\kappa}_{\mathbf{pq},\mathbf{x}}^{(i)} : \Omega^{\bullet +1}_{\mathrm{aug},\phi_{\mathbf{p}\mathbf{q}}}(\mathcal{F}'_{t_i,{\phi_{\mathbf{p}\mathbf{q}}(\mathbf{x})}}) \rightarrow \Omega^{\bullet +1}_{\mathrm{aug},\phi_{\mathbf{p}\mathbf{q}}}(\mathcal{F}'_{t_{i+1},{\phi_{\mathbf{p}\mathbf{q}}(\mathbf{x})}}).
\]

Here, both sides are $L_{\infty}[1]$-algebras that arise from the corresponding presymplectic structures $\widetilde{\omega}'(t_i)$ and $\widetilde{\omega}'(t_{i+1})$ together with the splittings (in (\ref{wdecomp})) $G'(t_i)$ and $G'(t_{i+1}),$ respectively.

\item By the construction described below, we obtain the induced morphisms from the family $\big\{\widetilde{\omega}'(t)\big\}_{t \in [t_i, t_{i+1}]},$ 
\[
\widehat{\kappa}_{\mathbf{pq},\mathbf{x}}^{(i)} : \Omega^{\bullet +1}(\mathcal{F}'_{t_i, \phi_{\mathbf{p}\mathbf{q}}(\mathbf{x})})_{\phi_{\mathbf{p}\mathbf{q}}} \rightarrow \Omega^{\bullet +1}(\mathcal{F}'_{t_{i+1}, \phi_{\mathbf{p}\mathbf{q}}(\mathbf{x})})_{\phi_{\mathbf{p}\mathbf{q}}},
\]
and its augmented version by Lemma \ref{auglmo}
\[
\widehat{\kappa}_{\mathbf{pq},\mathbf{x}}^{(i)} : \Omega^{\bullet +1}_{\mathrm{aug},{\phi_{\mathbf{p}\mathbf{q}}}}(\mathcal{F}'_{t_i,\phi_{\mathbf{p}\mathbf{q}}(\mathbf{x})}) \rightarrow \Omega^{\bullet +1}_{\mathrm{aug},{\phi_{\mathbf{p}\mathbf{q}}}}(\mathcal{F}'_{t_{i+1},\phi_{\mathbf{p}\mathbf{q}}(\mathbf{x})}).
\]
Since both the domain and the target are acyclic, it induces the zero map on the cohomology, hence is an quasi-isomorphism.  More details are provided in Subsection \ref{tcsi}.

\item In the same way as the case $(2),$ but in the opposite direction, we have a quasi-isomorphism from $\Omega^{\bullet +1}_{\mathrm{aug}, \phi_{\mathbf{p}\mathbf{q}}}(\mathcal{F}'_{t_{i+1},\phi_{\mathbf{p}\mathbf{q}}(\mathbf{x})})$ to $\Omega^{\bullet +1}_{\mathrm{aug}, \phi_{\mathbf{p}\mathbf{q}}}(\mathcal{F}'_{t_{i},\phi_{\mathbf{p}\mathbf{q}}(\mathbf{x})}).$ Then we take its homotopy inverse and denote it by:
\[
\widehat{\kappa}_{\mathbf{pq},\mathbf{x}}^{(i)} : \Omega^{\bullet +1}_{\mathrm{aug}, \phi_{\mathbf{p}\mathbf{q}}}(\mathcal{F}'_{t_{i},\phi_{\mathbf{p}\mathbf{q}}(\mathbf{x})}) \rightarrow \Omega^{\bullet +1}_{\mathrm{aug}, \phi_{\mathbf{p}\mathbf{q}}}(\mathcal{F}'_{t_{i+1},\phi_{\mathbf{p}\mathbf{q}}(\mathbf{x})}).
\]

\item We split the family into two: (a) $\big\{\widetilde{\omega}'(t)\big\}_{t \in [t_i, \frac{t_i + t_{i+1}}{2}]},$ (b) $\big\{\widetilde{\omega}'(t)\big\}_{t \in [\frac{t_i + t_{i+1}}{2}, t_{i+1}]}.$ For  (b), we proceed in the same way as in $(3),$ and for (a), as in $(2)$ to obtain quasi-isomorphisms
\[
\widehat{\kappa}_{\mathbf{pq},\mathbf{x},(a)}^{(i)} : \Omega^{\bullet +1}_{\mathrm{aug},\phi_{\mathbf{p}\mathbf{q}}}({\mathcal{F}}'_{t_i,\phi_{\mathbf{p}\mathbf{q}}(\mathbf{x})}) \rightarrow \Omega^{\bullet +1}_{\mathrm{aug}, \phi_{\mathbf{p}\mathbf{q}}}(\underline{\mathcal{F}}'_{t_{i},\phi_{\mathbf{p}\mathbf{q}}(\mathbf{x})})
\]
and
\[
\widehat{\kappa}_{\mathbf{pq},\mathbf{x},(b)}^{(i)} : \Omega^{\bullet +1}_{\mathrm{aug}, \phi_{\mathbf{p}\mathbf{q}}}(\underline{\mathcal{F}}'_{t_i, \phi_{\mathbf{p}\mathbf{q}}(\mathbf{x})}) \rightarrow \Omega^{\bullet +1}_{\mathrm{aug}, \phi_{\mathbf{p}\mathbf{q}}}({\mathcal{F}}'_{t_{i+1},\phi_{\mathbf{p}\mathbf{q}}(\mathbf{x})}),
\]
where $ \Omega^{\bullet +1}_{\mathrm{aug},\phi_{\mathbf{p}\mathbf{q}}}(\underline{\mathcal{F}}'_{t_i, \phi_{\mathbf{p}\mathbf{q}}(\mathbf{x})})$ denotes the augmented localization of the $L_{\infty}[1]$-algebra determined by the presymplectic form $\widetilde{\omega}'(\frac{t_i+t_{i+1}}{2}).$ Then we define
\[
\widehat{\kappa}_{\mathbf{pq},\mathbf{x}}^{(i)} := \widehat{\kappa}_{\mathbf{pq},\mathbf{x},{\text{(b)}}}^{(i)} \circ \widehat{\kappa}_{\mathbf{pq},\mathbf{x},{\text{(a)}}}^{(i)}
\]
that is a quasi-isomorphism by construction.
\end{enumerate}

\subsection{The case (2)}\label{tcsi}
Among the four types of non-continuity in the previous subsection, we focus on the case (2), that is, the case when $\text{rk}(T^*\mathcal{F}_{\mathbf{q}}^{'(i)}|_{W'_{\phi_{\mathbf{p}\mathbf{q}}(\mathbf{x})}})$ remains constant at all $t \in [t_i, t_{i+1})$ and increases by $1$ at $t=t_{i+1}$. The cases (3) and (4) can be treated with minor modifications, so we leave them to the reader as exercises. We remark that constructions in this subsection depend heavily on the content in Appendix C. 

At each $t \in [t_i, t_{i+1}]$, we consider a family of the normal components, that is, a family
$$
\big\{G_{\mathbf{q}}^{(i)}(t)\big\}_{t \in [t_i, t_{i+1}]}
$$
of subbundles of $TU'_{\mathbf{q}}|_{W'_{\phi_{\mathbf{p}\mathbf{q}}(\mathbf{x})}}$
that smoothly depends on $t$ and satisfies
\begin{equation*}\label{end}
TU'_{\mathbf{q}}|_{W'_{\phi_{\mathbf{p}\mathbf{q}}(\mathbf{x})}} = T \mathcal{F}_{t,W'_{\phi_{\mathbf{p}\mathbf{q}}(\mathbf{x})}}^{'(i)} \oplus G_{t, W'_{\phi_{\mathbf{p}\mathbf{q}}(\mathbf{x})}}^{(i)},
\end{equation*}
for the foliation tangent bundle $T F_{\mathbf{q},W_{\phi_{\mathbf{p}\mathbf{q}}(\mathbf{x})}}^{(i)}(t) := \ker\bigl(\omega_{\mathbf{q},W_{\phi_{\mathbf{p}\mathbf{q}}(\mathbf{x})}}^{(i)}(t)\bigr)$ for each $i$.  We emphasize that both $\mathrm{rk}\,T\mathcal{F}_{\mathbf{q},W_{\phi_{\mathbf{p}\mathbf{q}}(\mathbf{x})}}^{(i)}(t)$ and $\mathrm{rk}\,G_{\mathbf{q},W_{\phi_{\mathbf{p}\mathbf{q}}(\mathbf{x})}}^{(i)}(t)$ are not continuous at $t = t_{i+1}.$

At $t\in [t_i,t_{i+1})$, we have
\begin{equation}\label{txfgx1}
\begin{split}
\Gamma(T\mathcal{F}^{'(i)}_{t,{\phi_{\mathbf{p}\mathbf{q}}(\mathbf{x})}}) &= \text{span}_{C^{\infty}(W'_{{\phi_{\mathbf{p}\mathbf{q}}(\mathbf{x})}})} \Bigg\{ \frac{\partial}{\partial \mathbf{q}^{(i)}_{t,1}}, \cdots, \frac{\partial}{\partial \mathbf{q}^{(i)}_{t,n-k}} \Bigg\}, \\
\Gamma(G^{(i)}_{t,\mathbf{x}}) &= \text{span}_{C^{\infty}(W'_{{\phi_{\mathbf{p}\mathbf{q}}(\mathbf{x})}})}  \Bigg\{ \frac{\partial}{\partial \mathbf{y}^{(i)}_{t,j}} + \sum\limits_{\alpha =1}^{n-k} \mathbf{R}^{(i),\alpha}_{t,j} \frac{\partial}{\partial \mathbf{q}^{(i)}_{t,\alpha}} \Bigg\}_{1 \leq j \leq k}.
\end{split}
\end{equation}

The Poisson structure with respect to the presymplectic form $\widetilde{\omega}^{'(i)}(t)$ is given by
\[
\mathbf{P}^{'(i)}_t
= \sum\limits_{j,j'} \frac{1}{2}\,\widetilde{\omega}^{'(i)}_{jj'}\bigl(t\bigr)\,\mathbf{e}^{(i),j}_t\wedge \mathbf{e}^{(i),j'}_t
+ \sum\limits_{\alpha} \frac{\partial}{\partial \mathbf{p}^{(i),}_{t,\alpha}}\wedge \frac{\partial}{\partial \mathbf{q}^{(i),\alpha}_{t}},
\]
where we denote
\[
\mathbf{e}_t^{(i),j} := \frac{\partial}{\partial \mathbf{y}^{(i)}_{t,j}} + \sum\limits_{\alpha}\mathbf{R}_{t,j}^{(i),\alpha} \frac{\partial}{\partial \mathbf{q}^{(i),\alpha}} - \sum\limits_{\beta, \nu}\mathbf{p}^
{(i)}_{t,\beta} \frac{\partial \mathbf{R}^{(i),\beta}_{t,j}}{\partial \mathbf{q}^{(i),\nu}_t} \frac{\partial}{\partial \mathbf{p}^{(i)}_{t,\nu}},
\]
where $\mathbf{R}_{t,j}^{(i),\alpha}$ is from (\ref{txfgx1}). See Example \ref{gtyemb} for more details.

At $t = t_{i+1}$, new kernel directions appear. For simplicity of presentation, we assume that the number of the new directions is $1$. The other cases can be treated in a similar manner.

\begin{equation}\label{txfgx}
\begin{split}
\Gamma(T\mathcal{F}^{'(i)}_{t_{i+1},{\phi_{\mathbf{p}\mathbf{q}}(\mathbf{x})}}) &= \text{span}_{C^{\infty}(W'_{{\phi_{\mathbf{p}\mathbf{q}}(\mathbf{x})}})}  \Bigg\{ \frac{\partial}{\partial \mathbf{q}^{(i)}_{t_{i+1},1}}, \cdots, \frac{\partial}{\partial \mathbf{q}^{(i)}_{t_{i+1},n-k}},  \overbrace{\frac{\partial}{\partial \mathbf{q}^{(i)}_{t_{i+1},n-k+1}}}^{\text{the newly appearing kernel direction}}\Bigg\}, \\
\Gamma(G^{(i)}_{t_{i+1},\mathbf{x}}) &= \text{span}_{C^{\infty}(W'_{{\phi_{\mathbf{p}\mathbf{q}}(\mathbf{x})}})}  \left\{ \frac{\partial}{\partial \mathbf{y}^{(i)}_{t_{i+1},j}} + \sum\limits_{\alpha =1}^{n-k'} \mathbf{R}^{(i),\alpha}_{t_{i+1},j} \frac{\partial}{\partial \mathbf{q}^{(i)}_{t_{i+1},\alpha}} \right\}_{1 \leq j \leq k'}
\end{split}
\end{equation}

Their effect on the Poisson structure is given by
\begin{equation}\label{paosdfjil}
\mathbf{P}^{'(i)}_{t_{i+1}} = \overbrace{\sum\limits_{j,j'} \frac{1}{2}\widetilde{\omega}^{'(i)}_{jj'}\bigl(t_{i+1}\bigr) \mathbf{e}^{(i),j}_{t_{i+1}} \wedge \mathbf{e} ^{(i),j'}_{t_{i+1}} + \sum\limits_{\alpha} \frac{\partial}{\partial \mathbf{p}^{(i)}_{t_{i+1},\alpha}}\wedge \frac{\partial}{\partial \mathbf{q}_{t_{i+1}}^{(i),\alpha}}}^{ \overset{\circ}{\mathbf{P}}^{'(i)}_{t_{i+1}} := } + \overbrace{ \frac{\partial}{\partial \mathbf{p}'_\gamma}\wedge \frac{\partial}{\partial \mathbf{q}^{'\gamma}}}^{\text{the newly appearing term}},
\end{equation}
and
\begin{equation}\nonumber
\begin{split}
\mathbf{e}_{t_{i+1}}^{(i),j} := \frac{\partial}{\partial \mathbf{y}^{(i)}_{t_{i+1},j}} + \sum\limits_{\alpha} \mathbf{R}_{t_{i+1},j}^{(i),\alpha} \frac{\partial}{\partial \mathbf{q}^{(i),\alpha}_{t_{i+1}}} - &\sum\limits_{\beta, \nu}\mathbf{p}_{\beta} \frac{\partial \mathbf{R}^{(i),\beta}_{t_{i+1},j}}{\partial \mathbf{q}^{(i),\nu}_{t_{i+1}}} \frac{\partial}{\partial \mathbf{p}^{(i)}_{t_{i+1}, \nu}}\\
& (\text{including the newly appearing term}),
\end{split}
\end{equation}
where $\mathbf{R}^{\alpha}_j$ is from (\ref{txfgx}). We write
\[
{\overset{\circ}{\mathbf{P}}^{'(i)}_{t_{i+1}}} := \sum\limits_{j,j'} \frac{1}{2}\widetilde{\omega}^{'(i)}_{jj'}\bigl(t_{i+1}\bigr) \mathbf{e}^{(i),j}_{t_{i+1}} \wedge \mathbf{e} ^{(i),j'}_{t_{i+1}} + \sum\limits_{\alpha} \frac{\partial}{\partial \mathbf{p}^{(i)}_{t_{i+1},\alpha}}\wedge \frac{\partial}{\partial \mathbf{q}_{t_{i+1}}^{(i),\alpha}}
\]
for the first braced term in (\ref{paosdfjil}).

At $t = t_{i+1}-\epsilon$ for sufficiently small positive $\epsilon,$ we have
\[
{\mathbf{P}}^{'(i)}_{t_{i+1}-\epsilon} =
\frac{1}{2} \sum_{j,j'=1}^{m+1} \widetilde{\omega}^{'(i)}_{jj'}(t_{i+1}-\epsilon) \mathbf{e}^{(i),j}_{t_{i+1}-\epsilon} \wedge \mathbf{e}^{(i),j'}_{t_{i+1}-\epsilon}
+ \sum_{\alpha=1}^{m+1} \frac{\partial}{\partial \mathbf{p}_{t_{i+1}-\epsilon, \alpha}^{(i)}} \wedge \frac{\partial}{\partial \mathbf{q}_{t_{i+1}-\epsilon}^{(i), \alpha}},
\]
and we write
\[
\overset{\circ}{\mathbf{P}}^{'(i)}_{t_{i+1}-\epsilon} := {\mathbf{P}}^{'(i)}_{t_{i+1}-\epsilon} - \frac{1}{2}\sum_{j=1}^{m+1} \widetilde{\omega}^{'(i)}_{m+1,j}(t_{i+1}-\epsilon) (\cdots).
\]
In this case, we can actually take the coordinate system $\left( \{\mathbf{y}^{(i)}_{t_{i+1},j}\}_j, \{\mathbf{q}^{(i),\alpha}_{t_{i+1}}\}_{\alpha}\right)$ at $t= t_{i+1}$, so that it exhibits the following limiting behavior under $\epsilon \rightarrow 0:$
\begin{equation}
\begin{cases}
\frac{\partial}{\partial \mathbf{y}^{(i)}_{t_{i+1} - \epsilon,j}}
\overset{\epsilon \rightarrow 0}{\longmapsto} \frac{\partial}{\partial \mathbf{y}^{(i)}_{t_{i+1},j}}, j =1, \cdots, m,\\ 
\frac{\partial}{\partial  \mathbf{y}^{(i)}_{t_{i+1}- \epsilon,j}}
\overset{\epsilon \rightarrow 0}{\longmapsto} \frac{\partial}{\partial  \mathbf{q}^{(i), k+1}_{t_{i+1}- \epsilon}},\\
\frac{\partial}{\partial  \mathbf{q}^{(i), \alpha}_{t_{i+1} - \epsilon}}
\overset{\epsilon \rightarrow 0}{\longmapsto} \frac{\partial}{\partial  \mathbf{q}^{(i),\alpha}_{t_{i+1}}}, \alpha = 1, \cdots, k,\\
\frac{\partial}{\partial  \mathbf{p}^{(i)}_{t_{i+1}-\epsilon, \alpha}}
\overset{\epsilon \rightarrow 0}{\longmapsto} \frac{\partial}{\partial  \mathbf{p}^{(i)}_{t_{i+1}-\epsilon, \alpha}}, \alpha = 1, \cdots, k,
\end{cases}
\end{equation}
which is possible up to isomorphic changes and without loss of generality (cf. Lemma \ref{liiifdr} (iv)).

Note that $\widetilde{\omega}^{'(i)}(t)$ is closed, being a presymplectic form, and its kernel for $t \in [t_i, t_{i+1})$  is of constant rank by construction. The closedness implies that
\[
\left[\mathbf{P}_t^{'(i)},\mathbf{P}_t^{'(i)}\right] = 0.
\]
Hence $\mathbf{P}_t^{'(i)}$ as a Maurer-Cartan element determines an $L_\infty[1]$-algebra. In other words, we obtain a family of V-algebras
\[
\mathcal{V}^{(i)}_t = \left(\mathfrak{h}^{(i)}_t,\mathfrak{a}^{(i)}_t,\Pi^{(i)}_t\right), \ t \in [t_i, t_{i+1}]
\]
together with Poisson structures $\mathbf{P}_t^{'(i)}\in (\mathfrak{h}^{(i)}_t)^1$, where we denote
\begin{equation}\label{valgit}
\begin{cases}
\mathfrak{h}^{(i)}_t := \lim\limits_{\longleftarrow} \frac{\Gamma(T^*\mathcal{F}_{\mathbf{q},W'_{\phi_{\mathbf{p}\mathbf{q}}(\mathbf{x})}}^{'(i)}, \bigwedge^{\bullet +1}TT^*\mathcal{F}_{\mathbf{q},W'_{\phi_{\mathbf{p}\mathbf{q}}(\mathbf{x})}}^{'(i)})}{{I}^n \cdot \Gamma(T^*\mathcal{F}_{\mathbf{q},W'_{\phi_{\mathbf{p}\mathbf{q}}(\mathbf{x})}}^{'(i)}, \bigwedge^{\bullet +1}TT^*\mathcal{F}_{\mathbf{q},W'_{\phi_{\mathbf{p}\mathbf{q}}(\mathbf{x})}}^{'(i)})},\\
\mathfrak{a}^{(i)}_t := \Gamma(W'_{\phi_{\mathbf{p}\mathbf{q}}(\mathbf{x})}, \bigwedge^{\bullet +1}T^*\mathcal{F}_{\mathbf{q},W'_{\phi_{\mathbf{p}\mathbf{q}}(\mathbf{x})}}^{'(i)}),\\
\Pi^{(i)}_t : \mathfrak{h}^{(i)}_t \to \mathfrak{a}^{(i)}_t, \text{ the projection to the subspace.}
\end{cases}
\end{equation}
Here, $\mathcal{V}^{(i)}_t$ and $\mathbf{P}^{'(i)}_t$ fail to be continuous
at $t,$ where $\text{rk}(T\mathcal{F}^{(i)}(t))$ jumps.

From the V-algebra $\mathcal{V}^{(i)}_t$ with the Poisson structure $\mathbf{P}^{'(i)}(t)$, we can also consider its localization at $\mathrm{Im}\phi:$
\begin{equation}\nonumber
\CV^{(i)}_{t,\phi} := (\mathfrak{h}^{(i)}_{t,\phi}, \mathfrak{a}^{(i)}_{t,\phi}, \Pi^{(i)}_{t,\phi}),
\end{equation}
where we denote
\begin{equation}\nonumber
\begin{cases}
\mathfrak{h}^{(i)}_{t,\phi} &:= {C^{\infty}_{\phi_{\mathbf{pq}}}(W'_{{\phi_{\mathbf{p}\mathbf{q}}(\mathbf{x})}})} \otimes_{{C^{\infty}(W'_{{\phi_{\mathbf{p}\mathbf{q}}(\mathbf{x})}})}} \mathfrak{h}_t^{(i)},\\ \mathfrak{a}^{(i)}_{t,\phi} &:= {C^{\infty}_{\phi_{\mathbf{pq}}}(W'_{{\phi_{\mathbf{p}\mathbf{q}}(\mathbf{x})}})} \otimes_{{C^{\infty}(W'_{{\phi_{\mathbf{p}\mathbf{q}}(\mathbf{x})}})}}  \mathfrak{a}^{(i)}_t,\\ \Pi^{(i)}_{t, \phi} &:= \mathrm{id}_{C_{\phi_{\mathbf{pq}}}^{\infty}(W'_{\phi_{\mathbf{pq}}(\mathbf{x})})} \otimes_{C^{\infty}(W'_{\phi_{\mathbf{pq}}(\mathbf{x})})} \Pi_t^{(i)},
\end{cases}
\end{equation}
together with 
\[
\mathbf{P}^{'(i)}_{t,\phi} := 1 \otimes \mathbf{P}^{'(i)}_t.
\]

For $\epsilon \ll 1,$ $\mathfrak{h}^{(i)}_{t,\phi}$ determines a 1-parameter family of (localized) V-algebras with Maurer-Cartan elements,
\begin{equation}\nonumber
\begin{split}
\big(\mathfrak{h}_{\phi}(t), \mathfrak{a}_{\phi}(t), \Pi_\phi(t) \big) :
\big(&\mathfrak{h}_{\phi}(t_i), \mathfrak{a}_{\phi}(t_i), \Pi_\phi(t_i) \big) \rightsquigarrow
\big(\mathfrak{h}_{\phi}(t_{i+1} - \epsilon), \mathfrak{a}_{\phi}(t_{i+1} - \epsilon), \Pi_\phi(t_{i+1} - \epsilon) \big)\\
&\text{together with } \mathbf{P}^{'(i)}_{\phi,t_i} \quad\quad\quad\quad\quad \text{together with } \mathbf{P}^{'(i)}_{\phi,t_{i+1} - \epsilon}
\end{split}
\end{equation}
with $\mathfrak{a}_{\phi}(t_{i+1} -t_i - \epsilon) = U_{t_{i+1}-t_i-\epsilon}\left(\mathfrak{a}_{\phi}(0)\right).$

We have
\begin{equation}
\begin{split}
\ker\Pi^{(i)}_{t,\phi}& = \ker \left(\mathrm{id}_{C_{\phi_{\mathbf{pq}}}^{\infty}(W'_{\phi_{\mathbf{pq}}(\mathbf{x})})} \otimes_{C^{\infty}(W'_{\phi_{\mathbf{pq}}(\mathbf{x})})} \Pi_t^{(i)}\right)\\
&\simeq C^{\infty}_{\phi_{\mathbf{pq}}}(W'_{\phi_{\mathbf{pq}}(\mathbf{x})}) \otimes \ker\Pi^{(i)}_{t} \simeq C^{\infty}_{\phi_{\mathbf{pq}}}(W'_{\phi_{\mathbf{pq}}(\mathbf{x})}) \otimes \ker\Pi^{(i)}_{0}\\ 
&\simeq \ker \left(\mathrm{id}_{C_{\phi_{\mathbf{pq}}}^{\infty}(W'_{\phi_{\mathbf{pq}}(\mathbf{x})})} \otimes_{C^{\infty}(W'_{\phi_{\mathbf{pq}}(\mathbf{x})})}\Pi^{(i)}_{0} \right) = \ker\Pi^{(i)}_{0, \phi}
\end{split}
\end{equation}
from $\CV_{\phi}$ and  $\mathbf{P}^{'(i)}_{t, \phi}$ with Corollary \ref{indliso}. Then we obtain an $L_{\infty}[1]$-isomorphism
\[
\overset{\circ}{\kappa}_{\mathbf{p}\mathbf{q},\mathbf{x}}^{(i)} : \Omega^{\bullet +1}(\mathcal{F}^{'(i)}_{t_i,{\phi_{\mathbf{p}\mathbf{q}}(\mathbf{x})}})_{\phi_{\mathbf{p}\mathbf{q}}} \to {U}(t_{i+1} - t_{i} - \epsilon)\left(\Omega^{\bullet +1}(\mathcal{F}^{'(i)}_{t_i,{\phi_{\mathbf{p}\mathbf{q}}(\mathbf{x})}})_{\phi_{\mathbf{p}\mathbf{q}}}\right).
\]
Here ${U}(t_{i+1} - t_{i} - \epsilon)$ denotes the isomorphism obtained from Corollary \ref{indliso}.

We then define
\[
\widehat{\theta}^{(i)}_k :  {U}(t_{i+1} -t_i - \epsilon)\left(\Omega^{\bullet +1}(\mathcal{F}^{'(i)}_{t_i,{\phi_{\mathbf{p}\mathbf{q}}(\mathbf{x})}})_{\phi_{\mathbf{p}\mathbf{q}}}\right)^{\otimes k} \rightarrow \Omega^{\bullet +1}(\mathcal{F}^{'(i)}_{t_{i+1},{\phi_{\mathbf{p}\mathbf{q}}(\mathbf{x})}})_{\phi_{\mathbf{p}\mathbf{q}}}
\]
by
\[
\widehat{\theta}^{(i)}_k(\xi_1^{t_{i+1} - \epsilon}, \cdots, \xi_k^{t_{i+1} - \epsilon}) :=
\begin{cases}
\left(\lim\limits_{\epsilon \rightarrow 0} \xi_1^{t_{i+1} - \epsilon}\right)\Big|_{\mathbf{q}_{t_{i+1}}^{(i),k+1} = 0} & \text{if } k=1,\\
0, & \text{if } k \geq 2.
\end{cases}
\]
whose well-definedness is easy to see from the setting of the case (2).

\begin{claim}
$\widehat{\theta}^{(i)}:= \{\widehat{\theta}^{(i)}_k\}_{k\geq 1}$ is an $L_\infty[1]$-morphism.
\end{claim}

\begin{proof}
It suffices to show that
\begin{equation}
\widehat{\theta}^{(i)}_1\bigl(l_k^{t_{i+1}-\epsilon}(\xi_1^\epsilon, \cdots, \xi_k^\epsilon)\bigr)
= l_k^{t_{i+1}}\bigl(\widehat{\theta}_1^{(i)}(\xi_1^\epsilon), \cdots, \widehat{\theta}_1^{(i)}(\xi_k^\epsilon)\bigr).
\end{equation}
The left-hand side is given by:
\[
\begin{split}
\widehat{\theta}^{(i)}_1\bigl(l_k^{t_{i+1}-\epsilon}(\xi_1^\epsilon, \cdots, \xi_k^\epsilon)\bigr) &= \widehat{\theta}^{(i)}_1 \left(\left[\cdots \left [\mathbf{P}^{'(i)}_{t_{i+1}-\epsilon}, \xi_1^\epsilon \right], \cdots, \xi_l^\epsilon \right]\right)\\
&= \widehat{\theta}^{(i)}_1 \left(\sum_j \widetilde{\omega}_{m+1,j}^{\epsilon'} (\cdots)\right)
+ \widehat{\theta}^{(i)}_1 \left(\left[ \cdots \left[\overset{\circ}{\mathbf{P}}^{'(i)}_{t_{i+1}-\epsilon}, \xi_1 \right], \cdots, \xi_k \right]\right)\\
&\overset{(1)}{=}\widehat{\theta}^{(i)}_1\left(\left[ \cdots \left[\overset{\circ}{\mathbf{P}}^{'(i)}_{t_{i+1}-\epsilon}, \xi_1\right], \cdots, \xi_k\right]\right).
\end{split}
\]
Here, for the equality $(1),$ we use $\lim_{\epsilon \to 0} \widetilde{\omega}'_{m+1,j}(t_{i+1} - \epsilon) = 0,$ which follows from the fact that 
\[
\frac{\partial}{\partial \mathbf{q}^{(i),k+1}_{t_{i+1}}} \in \ker \lim\limits_{\epsilon \rightarrow 0}\widetilde{\omega}'_{m+1,j}(t_{i+1}-\epsilon)
\]
for all $j.$ 

The right-hand side is given by:
\[
\begin{split}
l_k^1\bigl(\widehat{\theta}_1^{(i)}(\xi_1^\epsilon), \cdots, \widehat{\theta}_1^{(i)}(\xi_k^\epsilon)\bigr) &= \left[ \cdots \left[ \mathbf{P}^{'(i)}_{t_{i+1}}, \widehat{\theta}^{(i)}_1(\xi_1^\epsilon)\right], \cdots, \widehat{\theta}^{(i)}_1(\xi_k^\epsilon)\right]\\
&= \left[ \cdots \left[  {\mathbf{P}}^{'(i)}_{t_{i+1}},  \left(\lim_{\epsilon \to 0} \xi_1^\epsilon\right)|_{q_{t_{i+1}}^{(i),k+1}=0}\right], \cdots,  \left(\lim_{\epsilon \to 0} \xi_k^\epsilon\right)|_{q_{t_{i+1}}^{(i),k+1}=0}\right]\\
&\overset{(2)}{=} \left[ \cdots \left[ \overset{\circ}{\mathbf{P}}^{'(i)}_{t_{i+1}}, \left(\lim_{\epsilon \to 0} \xi_1^\epsilon\right)|_{q_{t_{i+1}}^{(i),k+1}=0}\right], \cdots,  \left(\lim_{\epsilon \to 0} \xi_k^\epsilon\right)|_{q_{t_{i+1}}^{(i),k+1}=0}\right]\\
&\overset{(3)}{=} \left(\left[ \cdots \left[\overset{\circ}{\mathbf{P}}^{'(i)}_{t_{i+1}}, \left(\lim_{\epsilon \to 0} \xi_1^\epsilon\right)\right], \cdots, \left(\lim_{\epsilon \to 0} \xi_k^\epsilon\right)\right]\right)\Bigg|_{q_{t_{i+1}}^{(i),k+1}=0}\\
&\overset{(4)}{=}\left(\lim_{\epsilon \to 0}\left[ \cdots \left[\overset{\circ}{\mathbf{P}}^{'(i)}_{t_{i+1}-\epsilon}, \xi_1^\epsilon\right], \cdots, \xi_k^\epsilon\right]\right)\Bigg|_{q_{t_{i+1}}^{(i),k+1}=0}\\
&= \widehat{\theta}_1\left(\left[ \cdots \left[\overset{\circ}{\mathbf{P}}^{'(i)}_{t_{i+1} - \epsilon}, \xi_1^\epsilon\right], \cdots, \xi_k^\epsilon\right]\right).
\end{split}
\]
We explain how we obtain the equalities (2) through (4): (2) is a consequence of the fact that $ (\lim\limits_{\epsilon \to 0} \xi_1^\epsilon)|_{\mathbf{q}_{t_{i+1}}^{(i),k+1}=0}$ has no $q_{t_{i+1}}^{(i),k+1}$-dependence. For (3), we use the fact that $\overset{\circ}{\mathbf{P}}^{'(i)}_{t_{i+1}}$ has no term that involves the differentiation $\frac{\partial}{\partial \mathbf{q}_{t_{i+1}}^{(i),k+1}}.$ (4) follows when we take sufficiently small $\epsilon$ and a different parametrization ${\mathbf{P}}^{'(i)}_{t}$ in $t$ if necessary, so that we achieve the desired uniform convergence as $\epsilon \rightarrow 0.$ Thus we can assume that the bracket is interchangeable with the limit $\epsilon \rightarrow 0.$
\end{proof}

Now we define
\[
\widehat{\kappa}_{\mathbf{p}\mathbf{q},\mathbf{x}}^{(i)} : \Omega^{\bullet +1}(\mathcal{F}^{'(i)}_{t_i,{\phi_{\mathbf{p}\mathbf{q}}(\mathbf{x})}})_{\phi_{\mathbf{p}\mathbf{q}}} \to \Omega^{\bullet +1}(\mathcal{F}^{'(i)}_{t_{i+1},{\phi_{\mathbf{p}\mathbf{q}}(\mathbf{x})}})_{\phi_{\mathbf{p}\mathbf{q}}}
\]
by
\[
\widehat{\kappa}_{\mathbf{p}\mathbf{q},\mathbf{x}}^{(i)} := \widehat{\theta}_{\mathbf{p}\mathbf{q},\mathbf{x}}^{(i)} \circ \overset{\circ}{\kappa}_{\mathbf{p}\mathbf{q},\mathbf{x}}^{(i)} .
\]

At the same $t,$ but with different choices of splitting, the resulting two $L_{\infty}[1]$-algebras are related by the isomorphism of Lemma \ref{liiifdr} (iv):
\[
\widehat{\tau}_{\mathbf{p}\mathbf{q},\mathbf{x}}^{(i)} : \Omega^{\bullet +1}(\mathcal{F}^{'(i)}_{t_{i+1},{\phi_{\mathbf{p}\mathbf{q}}(\mathbf{x})}})_{\phi_{\mathbf{p}\mathbf{q}}} \xrightarrow{\simeq} \Omega^{\bullet +1}(\mathcal{F}^{'(i+1)}_{t_{i+1},{\phi_{\mathbf{p}\mathbf{q}}(\mathbf{x})}})_{\phi_{\mathbf{p}\mathbf{q}}}.
\]

We then obtain their augmented versions (written in the same notation) by Proposition \ref{htphi}:
\[
\widehat{\kappa}_{\mathbf{p}\mathbf{q},\mathbf{x}}^{(i)} : \Omega_{\mathrm{aug},{\phi_{\mathbf{p}\mathbf{q}}}}^{\bullet +1}(\mathcal{F}^{'(i)}_{t_i,{\phi_{\mathbf{p}\mathbf{q}}(\mathbf{x})}}) \to \Omega_{\mathrm{aug},{\phi_{\mathbf{p}\mathbf{q}}}}^{\bullet +1}(\mathcal{F}^{'(i)}_{t_{i+1},{\phi_{\mathbf{p}\mathbf{q}}(\mathbf{x})}}),
\]

\[
\widehat{\tau}_{\mathbf{p}\mathbf{q},\mathbf{x}}^{(i)} : \Omega_{\mathrm{aug},{\phi_{\mathbf{p}\mathbf{q}}}}^{\bullet +1}(\mathcal{F}^{'(i)}_{t_{i+1},{\phi_{\mathbf{p}\mathbf{q}}(\mathbf{x})}}) \xrightarrow{\simeq} \Omega_{\mathrm{aug},{\phi_{\mathbf{p}\mathbf{q}}}}^{\bullet +1}(\mathcal{F}^{'(i+1)}_{t_{i+1},{\phi_{\mathbf{p}\mathbf{q}}(\mathbf{x})}}).
\]

Finally, our second component of the de Rham coordinate change $\widehat{\phi}^{\mathrm{dR},2}_{\mathbf{pq}, \mathbf{x}}$ is given by:
\begin{equation}\nonumber
\begin{split}
\widehat{\phi}^{\mathrm{dR},2}_{\mathbf{pq}, \mathbf{x}}
:= &\text{ a homotopy inverse of }\\
& \quad (\widehat{\gamma}_{\mathbf{pq,x}})^{-1} \circ \widehat{\tau}_{\mathbf{pq},\mathbf{x}}^{(N)} \circ \widehat{\kappa}^{(N)}_{\mathbf{pq},\mathbf{x}}\circ \widehat{\tau}_{\mathbf{pq},\mathbf{x}}^{(N-1)}\circ \widehat{\kappa}^{(N-1)}_{\mathbf{pq},\mathbf{x}}\circ \cdots \circ \widehat{\tau}_{\mathbf{pq},\mathbf{x}}^{(1)} \circ \widehat{\kappa}_{\mathbf{pq},\mathbf{x}}^{(1)}\circ \widehat{\gamma}'_{\mathbf{pq,x}},
\end{split}
\end{equation}
whose existence is guaranteed by the quasi-isomorphism property of the maps $\widehat{\kappa}_{\mathbf{pq},\mathbf{x}}^{(i)}$'s and $\widehat{\tau}^{(i)}_{\mathbf{pq},\mathbf{x}}$'s, and by virtue of the Whitehead theorem (Theorem \ref{wht}). Also, note that $\widehat{\gamma}_{\mathbf{pq,x}}$ and $\widehat{\gamma}'_{\mathbf{pq,x}}$ defined in (\ref{ggdef}) are $L_{\infty}[1]$-isomorphisms, so we can take theire inverses.

\begin{defn}
We can write our \textit{de Rham part coordinate change} as:
\[
\widehat{\phi}^{\mathrm{dR}}_{\mathbf{pq}, \mathbf{x}}:\ \Omega_{\mathrm{aug},{\phi}_{\mathbf{pq}}}^{\bullet +1}(\mathcal{F}'_{\mathbf{q}, W_{\mathbf{p}\mathbf{q}}(\mathbf{x})}) \rightarrow \Omega_{\mathrm{aug}}^{\bullet}(\mathcal{F}_{\mathbf{p}, \mathbf{x}})
\]
as
\begin{equation}\nonumber
\widehat{\phi}^{\mathrm{dR}}_{\mathbf{pq}, \mathbf{x}}
:= \widehat{\phi}^{\mathrm{dR},1}_{\mathbf{pq}, \mathbf{x}} \circ \widehat{\phi}^{\mathrm{dR},2}_{\mathbf{pq}, \mathbf{x}}.
\end{equation}
\end{defn}

In fact, we have almost shown:
\begin{thm}
$\widehat{\phi}_{\mathbf{p}\mathbf{q},\mathbf{x}}$ is a quasi-isomorphism for each $\mathbf{p}, \mathbf{q},$ and $\mathbf{x}$.
\end{thm}

\begin{proof}
According to (\cite[Lemma 2.32]{Kim1} and) Proposition \ref{afec}, the Koszul part map $\widehat{\phi}^{\mathrm{K}}_{\mathbf{p}\mathbf{q},\mathbf{x}}$ is a quasi-isomorphism, which is also the case with the de Rham part $\widehat{\phi}^{\mathrm{dR}}_{\mathbf{p}\mathbf{q},\mathbf{x}}$ by construction.
\end{proof}

\subsection{$\mathcal{M}_{k+1}(L, \beta)$ as an $L_{\infty}$-Kuranishi space}

In this subsection, we prove that the moduli space $\mathcal{M}_{k+1}(L, \beta)$ is indeed an $L_{\infty}$-Kuranishi space.

\begin{prop}
The tuple 
\[
\Phi_{\mathbf{p}\mathbf{q}} = \left(U_{\mathbf{pq}}, \phi_{\mathbf{p}\mathbf{q}}, \left\{\widehat{\phi}_{\mathbf{p}\mathbf{q}, \mathbf{x}}\right\}\right)
\]
for $\mathbf{p}, \mathbf{q} \in \mathcal{M}_{k+1}(L, \beta)$ with $\mathrm{Im}  \psi_{\mathbf{p}} \cap \mathrm{Im}\psi_{\mathbf{q}} \neq \emptyset$ is a coordinate change for Kuranishi charts from $\mathcal{U}_{\mathbf{\mathbf{p}}}$ to $\mathcal{U}_{\mathbf{\mathbf{q}}}.$ 
\end{prop}

\begin{proof}
The conditions (i) through (iv) of Definition \ref{kstr} are all for the base components, which are already shown in Theorem 8.32 of \cite{FOOO3}.
\end{proof}

With regard to $\Phi_{\mathbf{p}\mathbf{q}},$ the following lemma highlights its favorable property, which will play an important role in Section \ref{morphsec}:

\begin{lem}\label{cochqim}
For each pair $\mathbf{p}', \mathbf{q} \in \mathcal{M}_{k+1}(L, \beta)$ with $\mathbf{p}' \in \mathrm{Im}\psi_{\mathbf{q}},$ the $L_{\infty}[1]$-morphism
\[
\widehat{\varepsilon}_{\mathbf{q},\phi_{\mathbf{p'q}}(\mathbf{x}),\phi_{\mathbf{p'q}}} : \mathcal{C}_{\mathbf{q},\phi_{\mathbf{p'q}}(\mathbf{x})} \rightarrow \mathcal{C}_{\mathbf{q},\phi_{\mathbf{p'q}}(\mathbf{x}),\phi_{\mathbf{p'q}}}
\]
for each $\mathbf{x} \in s_{\mathbf{p}'}^{-1}(0)$ induced from the FOOO coordinate change $\Phi_{\mathbf{p'q}} = \left(\phi_{\mathbf{p'q}}, \overline{\phi}_{\mathbf{p'q}} \right)$ is a quasi-isomorphism.
\end{lem}

\begin{proof}
We first note that for each such pair $\mathbf{p', q},$ and for $m := \dim U_\mathbf{q} - \dim U_{\mathbf{p'}},$ there exists a morphism of charts
\[
\widetilde{\Phi}_{\mathbf{p'q}} = \left(U_{\mathbf{p'q}} \times \mathbb{R}^m, \widetilde{\phi}_{\mathbf{p'q}}, \left\{\widehat{\widetilde{\phi}}_{\mathbf{p'q},\mathbf{x}}\right\}\right) : \mathcal{U}_{\mathbf{p}'} \times \mathbb{R}^m|_{U_{\mathbf{p'q}} \times \mathbb{R}^m} \rightarrow \mathcal{U}_{\mathbf{q}}
\]
with the following properties:
\begin{enumerate}[label = (\roman*)]
\item $\widetilde{\phi}_{\mathbf{p'q}} : U_{\mathbf{p'q}} \times \mathbb{R}^m \hookrightarrow U_{\mathbf{q}}$ is an open embedding,
\item $\widetilde{\phi}_{\mathbf{p'q}}|_{U_{\mathbf{p'q}} \times \{0\}} \equiv \phi_{\mathbf{p'q}},$
\item $\overline{\phi}_{\mathbf{p'q}} : E_{\mathbf{p'}} \times \mathbb{R}^m|_{U_{\mathbf{p'q}}\times \mathbb{R}^m} \hookrightarrow E_{\mathbf{q}} \text{ is a bundle embedding,}$
\end{enumerate}
The existence of such a number $m$ is guaranteed under the assumption of the contractibility of $U_{\mathbf{p'}}$ and $U_{\mathbf{p'q}}$

Observe that the pair $(\widetilde{\phi}_{\mathbf{p'q}}, \widetilde{\overline{\phi}}_{p'q})$ determines an FOOO embedding with the tangent bundle condition satisfied:
\[
\left[d_{(\mathbf{x},0)}(s_{\mathbf{p}'} \times \mathrm{id}_{\mathbb{R}^m})\right] : \frac{T_{\widetilde{\phi}_{\mathbf{p'q}}(\mathbf{x},0)}U_{\mathbf{q}}}{\widetilde{\phi}_{\mathbf{p'q}*}\left(T_{(\mathbf{x},0)}(U_{\mathbf{p}'}\times \mathbb{R}^m)\right)} \xrightarrow{\simeq} \frac{E_{\mathbf{q}, \widetilde{\phi}_{\mathbf{p'q}}(\mathbf{x},0)}}{\overline{\phi}_{\mathbf{p'q}}(E_{\mathbf{p}',\mathbf{x}} \times \mathbb{R}^m)}.
\]
at each $\mathbf{x} \in s^{-1}_{\mathbf{p}'}(0).$ Then Proposition \ref{afec} implies the quasi-isomorphicity of the Koszul part of $\widehat{\widetilde{\phi}}_{\mathbf{p'q},\mathbf{x}}$ for each $\mathbf{x},$ with the additional conditions (iv) and (v) in Condition \ref{addcond} satisfied. Since the de Rham part $L_{\infty}[1]$-morphism, being a map between acyclic $L_{\infty}[1]$-algebras, is automatically quasi-isomorphic, $\widehat{\widetilde{\phi}}_{\mathbf{p'q},\mathbf{x}}$ itself is a quasi-isomorphism. Moreover, $\widetilde{\phi}_{\mathbf{p'q}}|_{W_{\mathbf{x}}} : W_{\mathbf{x}} \times \mathbb{R}^m \rightarrow W_{\widetilde{\phi}_{\mathbf{p'q}}(\mathbf{x})}$ can be assume surjective by taking the smaller neighborhoods, if necessary. Thus we add to the above properties:
\begin{enumerate}[label = (\roman*), start = 4]
\item $\widehat{\widetilde{\phi}}_{\mathbf{p'q},\mathbf{x}} : \mathcal{C}_{\mathbf{q},\widetilde{\phi}_{\mathbf{p'q}}(\mathbf{x},0)} = \mathcal{C}_{\mathbf{q},\widetilde{\phi}_{\mathbf{p'q}}(\mathbf{x},0), \widetilde{\phi}_{\mathbf{p'q}}} \xrightarrow{\simeq} \mathcal{C}_{\mathbf{p'},(\mathbf{x},0)}^{\mathbb{R}^m}.$
\end{enumerate}

Considering the way $\widehat{\pi}_{(x,0)}, \widehat{\phi}_{\mathbf{p'q},\mathbf{x}}$ and $\widehat{\phi}_{\mathbf{p'q},\mathbf{x}}$ are defined, it is straightforward to see that we have the commutative (up to $L_{\infty}[1]$-homotopy) diagram
\begin{equation}
\begin{tikzcd}
 \mathcal{C}_{\mathbf{p'},(\mathbf{x},0)}^{\mathbb{R}^m}  \arrow{r}{\widehat{\pi}^{\mathrm{K}}_{(\mathbf{x},0)},\simeq} \arrow{d}[swap]{\widehat{\widetilde{\phi}}^{{\mathrm{K}},-1}_{\mathbf{p'q},\mathbf{x}}, \simeq} & \mathcal{C}_{\mathbf{p}',\mathbf{x}}\arrow{r}{\widehat{\phi}_{\mathbf{p'q},\mathbf{x}}^{{\mathrm{K}},-1}, \simeq} & \mathcal{C}_{\mathbf{q},\phi_{\mathbf{p'q}}(\mathbf{x}),\phi_{\mathbf{p'q}}}\\
\mathcal{C}_{\mathbf{q},\widetilde{\phi}_{\mathbf{p'q}}(\mathbf{x},0), \widetilde{\phi}_\mathbf{p'q}} \arrow{r}{=} &  \mathcal{C}_{\mathbf{q},\widetilde{\phi}_{\mathbf{p'q}}(\mathbf{x},0)}  \arrow{r}{=} & \mathcal{C}_{\mathbf{q},\phi_{\mathbf{p'q}}(\mathbf{x})} \arrow{u}[swap]{"\widehat{\varepsilon}^{\mathrm{K}}_{\mathbf{q},\phi_{\mathbf{p'q}}(\mathbf{x}),\phi_{\mathbf{p'q}}}"}
\end{tikzcd}
\end{equation}
that consists of the \textit{Koszul} part morphisms only. Since all the other $L_{\infty}[1]$-morphisms are quasi-isomorphic, so is $\widehat{\varepsilon}^{\mathrm{K}}_{\mathbf{q},\phi_{\mathbf{p'q}}(\mathbf{x}),\phi_{\mathbf{p'q}}}.$ The de Rham part morphism is also a quasi-isomorphism, being an $L_{\infty}[1]$-morphism between acyclic algebras.
\end{proof}

\begin{cor}\label{coinle}
For each pair $\mathbf{p}, \mathbf{q} \in \mathcal{M}_{k+1}(L, \beta)$ with $\mathrm{Im}\psi_{\mathbf{p}} \cap \mathrm{Im}\psi_{\mathbf{q}} \neq \emptyset,$ the $L_{\infty}[1]$-morphism
\[
\widehat{\varepsilon}_{\mathbf{q},\phi_{\mathbf{pq}}(\mathbf{x}),\phi_{\mathbf{pq}}} : \mathcal{C}_{\mathbf{q},\phi_{\mathbf{pq}}(\mathbf{x})} \rightarrow \mathcal{C}_{\mathbf{q},\phi_{\mathbf{pq}}(\mathbf{x}),\phi_{\mathbf{pq}}}
\]
for each $\mathbf{x} \in s_{\mathbf{p }}^{-1}(0) \cap U_{\mathbf{pq}}$
induced from the coordinate change $\Phi_{\mathbf{pq}}$ of $L_{\infty}$-Kuranishi space is a quasi-isomorphism.
\end{cor}
\begin{proof}
We can apply the proof of the previous lemma with the smaller presymplectic neighborhood $W_{\mathbf{pq,x}} := W_{\mathbf{p'q,x}} \cap U_{\mathbf{p'p}}$ for each zero point $\mathbf{x} \in s_{\mathbf{p }}^{-1}(0) \cap U_{\mathbf{pq}}.$
\end{proof}

We now state our main result in this section:
\begin{thm-defn}\label{krnsthm}
$\mathcal{M}_{k+1}(L, \beta) := \left( \mathcal{M}_{k+1}(L, \beta), \left[\left\{\mathcal{U}_{\mathbf{p}}\}, \{ \Phi_{\mathbf{p} \mathbf{q}} \right\}\right] \right)$ is an $L_{\infty}$-Kuranishi space, which we call \textit{the moduli space of pseudoholomorphic disks.}
\end{thm-defn}

\section{Morphisms of Kuranishi spaces $\mathcal{M}_{k+1}(L, \beta)$}\label{morphsec}

In this section, we present two examples of morphisms concerning the $L_{\infty}$-Kuranishi space $\mathcal{M}_{k+1}(L, \beta)$: the evaluation and forgetful morphisms.

\subsection{Evaluation morphisms}
Recall that the topological moduli space $\mathcal{M}_{k+1}(L, \beta)$ allows a natural map that evaluates at the boundary marked points
\begin{equation}\nonumber
\begin{split}
\mathrm{ev}_{i} : \mathcal{M}_{k+1}(L, \beta) &\rightarrow L, \ \ i = 0, 1, \cdots, k,\\
\big[\big((\Sigma, \vec{z}), u\big)\big] &\mapsto u(z_i).
\end{split}
\end{equation}
In this subsection, we would like to lift $\mathrm{ev}_{i}$ to the Kuranishi space level, where the Kuranishi space structures on the domain and the target are as in Theorem-Definition \ref{krnsthm} and Example \ref{man}, respectively. 

In Section \ref{ksmod}, we considered the local base chart
\begin{equation}\nonumber
U_{[((\Sigma, \vec{z}), u)]} := \{ \mathbf{x} \in \mathscr{U}_{[((\Sigma, \vec{z}), u)]} \mid \overline{\partial}u_{\mathbf{x}} \in E_{[((\Sigma, \vec{z}), u)]}(\mathbf{x}) \}.
\end{equation}
Using this, we define
\begin{equation}\nonumber
\begin{split}
\mathrm{ev}_{i, [((\Sigma, \vec{z}), u)]} : U_{[((\Sigma, \vec{z}), u)]} &\rightarrow \mathbb{R}^n_{u(z_i)},\\
\mathbf{x} := \big((\Sigma', \vec{z}'), u'\big) &\mapsto u'(z'_i).
\end{split}
\end{equation}
for $\mathbf{x} \in \mathscr{U}_{\mathbf{p}} \subset \mathcal{X}_{k+1}(L, \beta).$ Here, $\mathbb{R}^n_{u(z_i)}$ is the Euclidean model of the manifold $L^n$ at $u(z_i)$ (cf. Example \ref{man}).

From now on, we write $\mathbf{p}$ for $\big[\big((\Sigma, \vec{z}), u\big)\big]$ and $\mathrm{ev}_i({\mathbf{p}})$ for $u(z_i)$ and similarly for other cases. 

Note that the atlas $\widehat{\mathcal{U}}$ on the moduli can be replaced with the equivalent $\widehat{\mathcal{U}} \times V$ with larger dimension of the base $U_{\mathbf{p}} \times V$ with $\dim V \geq \dim L,$ if necessary. Then we extend the base chart map $\mathrm{ev}_{i,\mathbf{p}}$ to $U_{\mathbf{p}} \times V$ properly, so that
\[
{\mathrm{ev}}_{i,\mathbf{p}} : U_{\mathbf{p}} \times V \rightarrow U'_{\mathrm{ev}_i(\mathbf{p})}
\]
is surjective for each $\mathbf{p}$, hence so is $\mathrm{ev}_{i,\mathbf{p}}|_{W_{\mathbf{x}}} : W_{\mathbf{x}} \times V \rightarrow W'_{{\mathrm{ev}_{i,\mathbf{p}}}(\mathbf{x})}$ for every $\mathbf{x} \in (s_{\mathbf{p}} \times \mathrm{id}_V)^{-1}(0) \simeq s_{\mathbf{p}}^{-1}(0).$ (Such an extension always exists under the assumption that $U_{\mathbf{p}}$ is contractible.) As a result, we can assume that
\[
\mathcal{C}'_{\mathrm{ev}_i(\mathbf{p}), \mathrm{ev}_i(\mathbf{x}), \mathrm{ev}_{i, \mathbf{p}}} = \mathcal{C}'_{\mathrm{ev}_i(\mathbf{p}), \mathrm{ev}_i(\mathbf{x})}
\]
(cf. Lemma \ref{surjcp}).

The $L_\infty$-component
\[
\widehat{\mathrm{ev}}_{i,\mathbf{p},\mathbf{x}} : \mathcal{C}'_{\mathrm{ev}_i(\mathbf{p}), \mathrm{ev}_i(\mathbf{x}), \mathrm{ev}_{i,\mathbf{p}}} \rightarrow \mathcal{C}_{\mathbf{p},\mathbf{x}}
\]
 is defined as follows. Note that the domain here is given by
\begin{equation}\label{evmodo}
\mathcal{C}'_{\mathrm{ev}_i(\mathbf{p}), \mathrm{ev}_i(\mathbf{x}), \mathrm{ev}_{i,\mathbf{p}}} = \mathcal{C}'_{\mathrm{ev}_i(\mathbf{p}), \mathrm{ev}_i(\mathbf{x})} = \Omega_{\mathrm{aug}}^{\bullet+1}(W_{\mathrm{ev}_i(\mathbf{x})})
\end{equation}
for an open neighborhood $W_{\mathrm{ev}_i(\mathbf{x})}$ of $\mathrm{ev}_i(\mathbf{x})$ in $L$, while the target is
\[
\mathcal{C}_{\mathbf{p},\mathbf{x}} = \bigwedge\nolimits^{-\bullet} \Gamma(E^*|_{W_{\mathbf{x}}}) \oplus \Omega_{\mathrm{aug}}^{\bullet+1}(\mathcal{F}_{\mathbf{x}}).
\]
In (\ref{evmodo}), the first equation holds as a consequence of the above surjectivity assumption. The second one follows from the choice of Kuranishi space structure for smooth manifolds in Example \ref{man}. We then define
$$\widehat{\mathrm{ev}}_{i,\mathbf{p},\mathbf{x},k}: \Omega^{\bullet+1}_{\text{aug}}(W_{\mathrm{ev}_i(\mathbf{x})}) \rightarrow \bigwedge\nolimits^{-\bullet}\Gamma(E^*|_{W_{\mathbf{x}}}) \oplus \Omega_{\text{aug}}^{\bullet+1}(W_{\mathbf{x}}; \mathcal{F}_{\mathbf{x}})$$
by 
\begin{equation}\nonumber
\widehat{\mathrm{ev}}_{i,\mathbf{p},\mathbf{x},k}(\xi_1, \ldots, \xi_k) 
:=\begin{cases}
\mathrm{ev}_{i,\mathbf{p}}^*(\xi_1) &\text{if } k=1,\\
0 &\text{if } k \geq 2.
\end{cases}
\end{equation}
\begin{lem}\label{levipzx}
$\widehat{\mathrm{ev}}_{i,\mathbf{p},\mathbf{x}} := \{\widehat{\mathrm{ev}}_{i,\mathbf{p},\mathbf{x},k}\}_{k \geq 1}$ is an $L_\infty[1]$-morphism.
\end{lem}

\begin{proof}
From the definition of the base evaluation map $\mathrm{ev}_{i,\mathbf{p}}$, we first observe that the pulled back differential form $\mathrm{ev}_{i,\mathbf{p}}^*(\xi) \in \Omega^*(U_{\mathbf{p}})$ satisfies the following property for any $\xi \in \Omega^*(L)$: The directional derivative $d \mathrm{ev}_{i,\mathbf{p}}^*(\xi) (Y)$ vanishes for every vector field $Y \in \Gamma(TU_{\mathbf{p}})$ such that its restriction $Y_{\mathbf{y}} \in C^{\infty}(\Sigma, u_{\mathbf{y}}^*TX)$ is zero at the marked point $z_i \in \partial \Sigma,$ that is, $Y_{\mathbf{y}}|_{z_i} = 0$. Hence (as far as the foliation differentiation of $\mathrm{ev}_{i,\mathbf{p}}^*(\xi)$ is concerned) we only need to consider the directions of the vector fields $Y$ whose value at each $\mathbf{y} \in U_p$ is supported on an open neighborhood arbitrarily near to $z_i.$ But the closed two-form $\omega_{\mathbf{p}}$ evaluated at such a vector field $Y$,
\[
\omega_{\mathbf{p}}(Y, -) = \left\{\int \omega(Y_{\mathbf{y}}, -)\mathrm{dvol}_\Sigma\right\}_{\mathbf{y} \in U_{\mathbf{p}}} \text{ for } Y \in \Gamma(TU_{\mathbf{p}})
\]
vanishes by the Lagrangian boundary condition for $u_{\mathbf{y}}: \Sigma \to X$. In other words, we have $\{\text{Such vector fields }Y\text{'s}\} \subset \Gamma(T\mathcal{F}_{\mathbf{x}}).$ This implies that the foliation differential can be regarded simply as the ordinary one in our case. Thus, we have for $k=1$,
\begin{equation}\nonumber
\begin{split}
l_1\left(\widehat{\mathrm{ev}}_{i,\mathbf{p}, \mathbf{x}, 1}(\xi_1)\right) &= \Pi\left[P_{\mathbf{x}}, {\mathrm{ev}}_{i,\mathbf{p}}^*(\xi_1)\right] \\ &= d\left(\mathrm{ev}_{i,\mathbf{p}}^*(\xi_1)\right) =\mathrm{ev}_{i,\mathbf{p}}^*(d\xi_1) = \widehat{\mathrm{ev}}_{i,\mathbf{p}, \mathbf{x}, 1}\left(l^{\mathrm{man}}_1(\xi_1)\right)
\end{split}
\end{equation}
from Lemma \ref{liiifdr}, where $P_{\mathbf{x}}$ denotes the Poisson structure on the presymplectic neighborhood $W_{\mathbf{x}}.$

For $k \geq 2$, we have
\begin{equation}\label{evk2}
l_k\left(\widehat{\mathrm{ev}}_{i, \mathbf{p}, \mathbf{x}, 1}(\xi_1), \cdots, \widehat{\mathrm{ev}}_{i, \mathbf{p}, \mathbf{x}, 1}(\xi_k)\right) = \Pi\left[\cdots \left[P_{\mathbf{x}}, \mathrm{ev}_{i, \mathbf{p}}^*(\xi_j)\right], \cdots, \mathrm{ev}_{i, \mathbf{p}}^*(\xi_k)\right] = 0
\end{equation}
by Lemma \ref{strlinf} (ii). Then (\ref{evk2}) further equals
\[
0 = \widehat{\mathrm{ev}}_{i, \mathbf{p}, \mathbf{x}, 1}(0) = \widehat{\mathrm{ev}}_{i, \mathbf{p}, \mathbf{x}, 1}\left(l^{\text{man}}_k(\xi_1, \cdots, \xi_k)\right),
\]
which amounts to the $L_{\infty}[1]$-relation.
\end{proof}

\begin{thm-defn}[Evaluation morphisms]
The equivalence class 
\[\mathrm{Ev}_i := \left[\left(\widehat{\mathcal{U}}_{\mathbf{p}}, \widehat{\mathcal{U}}^{\mathrm{man}}_{\mathrm{ev}_i(\mathbf{p})}, \mathrm{ev}_i, \{\mathrm{ev}_{i,\mathbf{p}}\}, \{\widehat{\mathrm{ev}}_{i,\mathbf{p},\mathbf{x}}\} \right)\right]
\] 
defines a morphism of Kuranishi spaces
\[ 
\mathrm{Ev}_i : \mathcal{M}_{k+1}(L, \beta) \rightarrow L
\]
for each $i$, which we call the \textit{evaluation morphism} of the moduli space $\mathcal{M}_{k+1}(L, \beta).$
\end{thm-defn}

\begin{proof}
First, we know that the map $\mathrm{ev}_i$ is continuous (cf. \cite[Lemma 7.10]{FOOO3}). We further show that the axioms (i) and (ii) in Definition \ref{morphismkur}. 

(i) $\psi_{\mathrm{ev}_i(\mathbf{p})}' \circ \mathrm{ev}_{i, \mathbf{p}} = \mathrm{ev}_i \circ \psi_{\mathbf{p}}$ on $s_{\mathbf{p}}^{-1}(0)$ is straightforward to verify because the homeomorphism $\psi_{\mathbf{p}}$ is given simply by $\mathbf{x} \mapsto \mathbf{x}.$

(ii) Consider a pair of points $\mathbf{p}, \mathbf{q} \in X$ with $\mathrm{Im} \psi_{\mathbf{p}} \cap \mathrm{Im} \psi_{\mathbf{q}} \neq \emptyset.$  Since $L$ is a manifold, we have $s_{\mathrm{ev}_{i} (\mathbf{p})}^{-1}(0) = U_{\mathrm{ev}_{i}(\mathbf{p})}.$ Then using (i), we also have
\begin{equation}\label{peppe}
\psi'_{\mathrm{ev}_{i}(\mathbf{p})} \circ \mathrm{ev}_{i,\mathbf{p}} = \mathrm{ev}_i \circ \psi_{\mathbf{p}} \overset{(1)}{=} \mathrm{ev}_i \circ \psi_{\mathbf{q}} \circ \phi_{\mathbf{p}\mathbf{q}} = \psi'_{\mathrm{ev}_{i}(\mathbf{q})} \circ \mathrm{ev}_{i,\mathbf{q}} \circ \phi_{\mathbf{p}\mathbf{q}},
\end{equation}
where the equality (1) follows from the definitions of the coordinate change $\phi_{\mathbf{p}\mathbf{q}} : U_{\mathbf{p}\mathbf{q}} \hookrightarrow U_{\mathbf{q}}$ and the maps $\psi_{\mathbf{p}}, \psi_{\mathbf{q}} : \mathbf{x} \mapsto \mathbf{x}$ on $s^{-1}_{\mathbf{p}}(0) \cap U_{\mathbf{p}\mathbf{q}}.$ It follows that
\begin{equation}\nonumber
\phi'_{\mathrm{ev}_{i}(\mathbf{p})\mathrm{ev}_i(\mathbf{q})} \circ \mathrm{ev}_{i, \mathbf{p}} \overset{(2)}{=} \psi^{'-1}_{\mathrm{ev}_i(\mathbf{q})} \circ \psi_{\mathrm{ev}_{i}(\mathbf{p})}^{'} \circ \mathrm{ev}_{i, \mathbf{p}} \overset{(3)}{=}  \psi^{'-1}_{\mathrm{ev}_i(\mathbf{q})} \circ \psi'_{\mathrm{ev}_{i}(\mathbf{q})} \circ \mathrm{ev}_{i,\mathbf{q}} \circ \phi_{\mathbf{p}\mathbf{q}} = \mathrm{ev}_{i,\mathbf{q}} \circ \phi_{\mathbf{p}\mathbf{q}},
\end{equation}
where the equalities (2) and (3) are the consequences of the definition of coordinate change for manifolds
\[
\phi_{\mathrm{ev}_i(\mathbf{p})\mathrm{ev}_i(\mathbf{q})} := \psi^{'-1}_{\mathrm{ev}_i(\mathbf{q})} \circ \psi^{'}_{\mathrm{ev}_i(\mathbf{p})}
\]
and the relation (\ref{peppe}), respectively.
The $(\Gamma_{\mathbf{p}}, \Gamma_{\mathrm{ev}_i(\mathbf{p})})$-equivariance of $\mathrm{ev}_{i, \mathbf{p}}$ follows from the fact that the automorphisms in $\textrm{Aut}(\mathbf{p})$ preserve the marked points, and that $\Gamma_{\mathrm{ev}_i(\mathbf{p})}$ is trivial.

(iii) Since one can assume the surjectivity of $\mathrm{ev}_{i, \mathbf{p}}|_{W_{\mathbf{x}}}$, we only need to show the homotopy commutativity of the following diagram:

\begin{equation}\label{octa3}
\begin{aligned}
\begin{tikzcd}
\mathcal{C}_{\mathbf{q}, \phi_{\mathbf{pq}}(\mathbf{x})} \arrow{d}[swap]{\widehat{\varepsilon}_{\mathbf{q}, \phi_{\mathbf{pq}}(\mathbf{x}), \phi_{\mathbf{pq}}}, \simeq} & \mathcal{C}'_{\mathrm{ev}_i(\mathbf{q}), \mathrm{ev}_{\mathbf{q}} \circ \phi_{\mathbf{pq}}(\mathbf{x}), \mathrm{ev}_{i, \mathbf{q}}} \arrow{l}[swap]{\widehat{\mathrm{ev}}_{i, \mathbf{q}, \phi_{\mathbf{pq}}(\mathbf{x})}} \cdots&  \quad \quad \quad \quad \quad \quad \quad \quad \quad \quad \quad \quad \quad \quad \quad \quad \quad \quad \quad \quad \quad \quad \quad \quad \quad \quad \quad \quad \quad \quad \quad \quad \quad \quad \quad \quad \quad \quad \quad \quad \\
\mathcal{C}_{\mathbf{q}, \phi_{\mathbf{pq}}(\mathbf{x}), \phi_{\mathbf{pq}}} \arrow{d}[swap]{\widehat{\phi}_{\mathbf{pq},\mathbf{x}}, \simeq} & {} & \quad \quad \quad \quad \quad \quad \quad \quad \quad \quad \quad \quad \quad \quad \quad \quad \quad \quad \quad \quad \quad \quad \quad \quad \quad \quad \quad \quad \quad \quad \quad \quad \quad \quad \quad \quad \quad \quad \quad \quad  \\
\mathcal{C}_{\mathbf{p}, \mathbf{x}} & \mathcal{C}'_{\mathrm{ev}_i(\mathbf{p}),\mathrm{ev}_{i,\mathbf{p}}(\mathbf{x}),\mathrm{ev}_{i,\mathbf{p}}} \arrow{l}[swap]{{\widehat{\mathrm{ev}}_{i, \mathbf{p},\mathbf{x}}}} \cdots & \quad \quad \quad \quad \quad \quad \quad \quad \quad \quad \quad \quad \quad \quad \quad \quad \quad \quad \quad \quad \quad \quad \quad \quad \quad \quad \quad \quad \quad \quad \quad \quad \quad \quad \quad \quad \quad \quad \quad \quad \\
\end{tikzcd}\\
\begin{tikzcd}
\cdots & \mathcal{C}'_{\mathrm{ev}_i(\mathbf{q}),\mathrm{ev}_{i,\mathbf{q}} \circ \phi_{\mathbf{pq}}(\mathbf{x})} = \Omega^{\bullet+1}_{\mathrm{aug}}(W_{\mathrm{ev}_{i, \mathbf{q}} \circ \phi_{\mathbf{pq}}(\mathbf{x})}) = \mathcal{C}'_{\mathrm{ev}_i(\mathbf{q}), \phi'_{\mathrm{ev}_i(\mathbf{p})\mathrm{ev}_i(\mathbf{q})} \circ \mathrm{ev}_{i,\mathbf{p}}(\mathbf{x})} \arrow{l}[swap]{\widehat{\varepsilon}_{\mathrm{ev}_i(\mathbf{q}), \mathrm{ev}_{i, \mathbf{q}}\circ \phi_{\mathbf{pq}}(\mathbf{x}), \mathrm{ev}_{i,\mathbf{q}},} \simeq} \arrow{d}{ \widehat{\varepsilon}_{\mathrm{ev}_i(\mathbf{q}),\phi'_{\mathrm{ev}_i(\mathbf{p})\mathrm{ev}_i(\mathbf{q})}\circ \mathrm{ev}_{i, \mathbf{p}}(\mathrm{x}), \phi'_{\mathrm{ev}_i(\mathbf{p})\mathrm{ev}_i(\mathbf{q})},} \simeq} & {} & \quad \quad \quad \quad \quad \quad \quad \quad \quad \quad \quad \quad \quad \quad \quad \quad\quad \quad \quad \quad \\
{} & \mathcal{C}'_{\mathrm{ev}_i(\mathbf{q}),\phi'_{\mathrm{ev}_i(\mathbf{p})\mathrm{ev}_i(\mathbf{q})}\circ \mathrm{ev}_{i,\mathbf{p}}(\mathbf{x}), \phi'_{\mathrm{ev}_i(\mathbf{p})\mathrm{ev}_i(\mathbf{q})}} \arrow{d}{\widehat{\phi}'_{\mathrm{ev}_i(\mathbf{p})\mathrm{ev}_i(\mathbf{q}), \mathrm{ev}_i(\mathbf{x})}, \simeq} & {} & \quad \quad \quad \quad \quad \quad \quad \quad \quad \quad \quad \quad \quad \quad \quad \quad \quad \quad \quad \quad\\
\cdots & \mathcal{C}'_{\mathrm{ev}_i(\mathbf{p}),\mathrm{ev}_{i,\mathbf{p}}(\mathbf{x})} = \Omega^{\bullet+1}_{\mathrm{aug}}(W_{\mathrm{ev}_{i,\mathbf{p}}(\mathbf{x})}) \arrow{l}[swap]{\widehat{\varepsilon}_{\mathrm{ev}_i(\mathbf{p}),\mathrm{ev}_{i,\mathbf{p}}(\mathbf{x}),\mathrm{ev}_{i,\mathbf{p}}}, \simeq} & {} & \quad \quad \quad \quad \quad \quad \quad \quad \quad \quad \quad \quad \quad \quad \quad \quad \quad \quad \quad \quad\\
\end{tikzcd}
\end{aligned}
\end{equation}
This diagram may look complicated; however, checking the homotopy commutativity is almost trivial. The common domain of the two $L_{\infty}[1]$-morphisms under consideration has no Koszul part, so we only need to examine the de Rham part of the morphisms. Hence the $L_{\infty}[1]$-morphisms
\[
\widehat{\phi}_{\mathbf{p}\mathbf{q},\mathbf{x}} \circ \widehat{\varepsilon}_{\mathbf{q}, \phi_{\mathbf{pq}}(\mathbf{x}), \phi_{\mathbf{pq}}} \circ \widehat{\mathrm{ev}}_{i,\mathbf{q},\phi_{\mathbf{pq}}(\mathbf{x})}\circ \widehat{\varepsilon}_{\mathrm{ev}_i(\mathbf{q}), \mathrm{ev}_{i, \mathbf{q}}\circ \phi_{\mathbf{pq}}(\mathbf{x}), \mathrm{ev}_{i,\mathbf{q}}}
\]
and 
\[
\widehat{\mathrm{ev}}_{i, \mathbf{p},\mathbf{x}} \circ \widehat{\varepsilon}_{\mathrm{ev}_i(\mathbf{p}),\mathrm{ev}_{i,\mathbf{p}}(\mathbf{x}),\mathrm{ev}_{i,\mathbf{p}}} \circ \widehat{\phi}'_{\mathrm{ev}_i(\mathbf{p})\mathrm{ev}_i(\mathbf{q}),\mathrm{ev}_i(\mathbf{x})} \circ \widehat{\varepsilon}_{\mathrm{ev}_i(\mathbf{q}),\phi'_{\mathrm{ev}_i(\mathbf{p})\mathrm{ev}_i(\mathbf{q})}\circ \mathrm{ev}_{i, \mathbf{p}}(\mathrm{x}), \phi'_{\mathrm{ev}_i(\mathbf{p})\mathrm{ev}_i(\mathbf{q})}}
\] 
are defined from an acyclic $L_{\infty}[1]$-algebra (i.e., $\Omega_{\mathrm{aug}}^{\bullet +1}(W_{\mathrm{ev}_{i, \mathbf{q}} \circ \phi_{\mathbf{pq}}(\mathbf{x})})$) to another, hence are quasi-isomorphisms. Then we know from Corollary \ref{anhp} that there exists an $L_{\infty}[1]$-homotopy between them.
\end{proof}

\subsection{Forgetful morphisms}
Suppose that we are given the family \textit{topological} moduli spaces 
\[
\{\mathcal{M}_{k+1}(L, \beta)\}_{k \geq 0},
\]
for each $0 \leq i \leq k,$ we have the forgetful map
\[
\mathrm{ft}_i : \mathcal{M}_{k+1}(L, \beta), \rightarrow \mathcal{M}_{k}(L, \beta)
\]
that forgets the $i$-th marked point. By forgetting the marked points, some components may become unstable. Then we shrink it and the resulting (equivalence class of) map is defined to be the value of $\mathrm{ft}_i.$ In this subsection, we show that it can be given an interpretation as an $L_{\infty}$-Kuranishi morphism.

First, using the Kuranishi space structure on $\mathcal{M}_{k+1}(L, \beta),$ constructed in Section \ref{ksmod}, we provide its description with respect to the local data.

The following lemma is an $L_{\infty}$-version of \cite[Lemma 7.3.8]{FOOO2}.
\begin{lem}\label{idks}
Let $\mathbf{p} \in \mathcal{M}_{k}(L, \beta)$ and $\mathbf{p}^+ \in \mathcal{M}_{k+1}(L, \beta)$ be points on the moduli spaces that satisfy $\mathrm{ft}_i(\mathbf{p}^+) = \mathbf{p}.$ Then Kuranishi charts $\mathcal{U}_{\mathbf{p}} = (U_{\mathbf{p}}, E_{\mathbf{p}}, s_{\mathbf{p}}, \Gamma_{\mathbf{p}}, \psi_{\mathbf{p}}) $ at $\mathbf{p}$ and $\mathcal{U}_{\mathbf{p}^+}=(U_{\mathbf{p}^+}, E_{\mathbf{p}^+}, s_{\mathbf{p}^+}, \Gamma^+_{\mathbf{p}}, \psi_{\mathbf{p}^+})$ at $\mathbf{p}^+$ can be taken in such a way that the following hold.
\begin{enumerate}[label = (\roman*)]
\item ${U}_{\mathbf{p^+}} \simeq U_{\mathbf{p}} \times \mathcal{W}_{\mathbf{p}} \times \mathcal{W}'_{\mathbf{p}^+},$ where $\mathcal{W}_{\mathbf{p}} \subset \mathbb{R}$ is an open interval containing 0, and $\mathcal{W}'_{\mathbf{p^+}}$ is an open neighborhood of $\mathbb{R}^{c(\mathbf{p}^+)},$ where the non-negative integer $c(\mathbf{p}^+)$ is given by
\begin{equation}\nonumber
c(\mathbf{p}^+) := \begin{cases} 1 & \text{ if } \mathrm{ft}_i(\mathbf{p}^+) \text{ is unstable,}\\
0 & \text{ otherwise.}
\end{cases}
\end{equation}
\item The closed two-form $\omega_{\mathbf{p}^+}$ is given in the same way as $\omega_{\mathbf{p}}$ in (\ref{o2f}).
\item The isotropy group $\Gamma_{\mathbf{p}^+} \simeq \Gamma_{\mathbf{p}}$ on $V_{\mathbf{p}^+}$ acts trivially on $\mathcal{W}_{\mathbf{p}} \times \mathcal{W}'_{\mathbf{p}^+}.$ The action on $U_{\mathbf{p}}$ coincides with that by $\Gamma_{\mathbf{p}}$ on $U_{\mathbf{p}}.$ 
\item ${E}_{\mathbf{p}^+} \simeq \bigoplus\limits_{i} {E}_i,$ where ${E}_i$ is a finite dimensional subspace of $\Gamma(\Sigma, T^*\Sigma^{0,1} \otimes w^{*}TM),$ each element of which consists of sections supported on a compact subset of the $i$-th irreducible component of $\Sigma.$
\item $ {E}_{\mathbf{p}^+}|_{U_{\mathbf{p}} \times \mathcal{W}_{\mathbf{p}} \times \{0\}} \simeq \pi^*({E}_{\mathbf{p}}) \oplus \mathbb{R}^{c(\mathbf{p}^+)},$ where $\pi : U_{\mathbf{p}} \times \mathcal{W}_{\mathbf{p}} \rightarrow U_{\mathbf{p}}$ is the projection, and $\mathbb{R}^{c(\mathbf{p}^+)}$ is the trivial vector bundle on ${U}_{\mathbf{p^+}}$ for the positive integer $c(\mathbf{p}^+)$ in (i).
\item $s_{\mathbf{p}^+}|_{U_{\mathbf{p}} \times \mathcal{W}_{\mathbf{p}} \times \{0\}} \simeq \pi^*s_{\mathbf{p}} \oplus \{0\},$ so that $\mathrm{Im}\left(s_{\mathbf{p}^+}|_{U_{\mathbf{p}} \times \mathcal{W}_{\mathbf{p}} \times \{0\}}\right) \subset \pi^*({E}_{\mathbf{p}}).$
\item The differential map $ds_{\mathbf{p}^+}|_{U_{\mathbf{p}} \times \mathcal{W}_{\mathbf{p}} \times \{0\}, \vec{0}}$ is given by the obvious embedding $T_0\mathcal{W}'_{\mathbf{p}^+} \hookrightarrow T_0 \mathbb{R}^{c(\mathbf{p}^+)} \simeq \mathbb{R}^{c(\mathbf{p}^+)} \subset {E}_{\mathbf{p}^+}|_{U_{\mathbf{p}} \times W_{\mathbf{p}} \times \{0\},\vec{0}}.$
\item The local (presymplectic) neighborhood is given by ${W}_{\mathbf{x}^+} := {W}_{\mathbf{x}} \times \mathcal{W}_{\mathbf{p}} \times \mathcal{W}'_{\mathbf{p}^+}.$
\item $s_{\mathbf{p}^+}(0) = 0,$ where $0 \in U_{\mathbf{p}^+}$ is the point such that its class of the isotropy group action gives $\psi_{\mathbf{p}^+}^+([0]) = \mathbf{p}^+.$
\item The evaluation maps $\mathrm{ev}_0^+ : U_{\mathbf{p}^+} \rightarrow L$ and $\mathrm{ev}_0 : U_{\mathbf{p}} \rightarrow L$ satisfy $\mathrm{ev}^+_0 = ev \circ \pi',$ where $\pi' : U^+_{\mathbf{p}^+} \rightarrow U_{\mathbf{p}}$ is the projection from (i).
\end{enumerate}
\end{lem}

\begin{proof}[Proof-sketch]
The proofs of all parts except (ii) and (viii) are given in \cite[Lemma 7.3.8]{FOOO2}, while (ii) and (viii) follow directly from the definitions.
\end{proof}

Roughly speaking, $\mathcal{W}_{\mathbf{p}}$ in (i) amounts to the location of the marked point that is forgotten, while the $\mathcal{W}'_{\mathbf{p}^+}$-direction in $TU_{\mathbf{p}^+}$ amounts to the $\Gamma_{\mathbf{p}^+}$-orbits of the map with one marked point being removed, which shrinks to a point after stabilization. 

Lemma \ref{idks} allows us to have the following projection for each $\mathbf{p}$
\begin{equation}\nonumber
\mathrm{ft}_{i,\mathbf{p}} : {U}^+_{\mathbf{p^+}} \simeq U_{\mathbf{p}} \times \mathcal{W}_{\mathbf{p}} \times \mathcal{W}'_{\mathbf{p}^+} \twoheadrightarrow U_{\mathbf{p}},\\
\end{equation}
as our base component map.

For the $L_{\infty}$-component map, we need the following lemma.
\begin{lem}\label{agssagaga}
For $\mathbf{x}^+ \in (s^+_{\mathbf{p}^+})^{-1}(0)$ and $\mathrm{ft}_{i, \mathbf{p}}(\mathbf{x}^+) = \mathbf{x} \in s_{\mathbf{p}}^{-1}(0),$ we have  a decomposition of the foliation tangent bundle
\[
T\mathcal{F}^+_{\mathbf{x}^+}|_{W_{\mathbf{x}^+}} \simeq \mathrm{ft}_{i, \mathbf{p}}^*T\mathcal{F}_{\mathbf{x}}|_{W_{\mathbf{x}^+}} \oplus \mathcal{W}_{\mathbf{p}} \oplus \mathcal{W }'_{\mathbf{p}^+}, 
\]
\end{lem}

\begin{proof}
Consider the restriction of the tangent bundle to the local neighborhood $W_{\mathbf{x}^+}$,
\[
W_{\mathbf{x}^+} \times \mathcal{W}_{\mathbf{p}} \subset TU_{\mathbf{p}^+}|_{W_{\mathbf{x}^+}}.
\]
We claim that
\[
\mathcal{W}_{\mathbf{p}} \subset \ker \omega_{\mathbf{p}^+}|_{{W}_{\mathbf{x}^+}} = T\mathcal{F}^+_{\mathbf{x}^+}|_{{W}_{\mathbf{x}^+}}.
\]
In fact, infinitesimal changes in the $\mathcal{W}_{\mathbf{p}}$-direction make no difference to the closed two-form $\omega_{\mathbf{p}}$ (and $\mathrm{ft}_{i,\mathbf{p}}^* \omega_{\mathbf{p}}$). This is because the location of boundary marked points is irrelevant to the way in which $\omega_{\mathbf{p}}$ is defined.

Since the crossing terms for the closed two-form are all zero, we have
\[
\mathrm{ft}_{i,\mathbf{p}}^*T\mathcal{F}_{\mathbf{x}}|_{{W}_{\mathbf{x}^+}} \subset T\mathcal{F}^+_{\mathbf{x}^+}|_{{W}_{\mathbf{x}^+}}.
\]
Then, over a sufficiently small contractible $W_{\mathbf{x}^+}$, we obtain the desired decomposition (as an isomorphism between trivial bundles) from the following relation for the ranks and dimensions:
\[
\mathrm{rk}\left(T\mathcal{F}^+_{\mathbf{x}^+}|_{W_{\mathbf{x}^+}}\right) - \mathrm{rk}\left(\mathrm{ft}_{i, \mathbf{p}}^*T\mathcal{F}_{\mathbf{x}}|_{W_{\mathbf{x}^+}}\right) \leq \dim W_{\mathbf{x}^+} - \dim W_{\mathbf{x}} = \dim \mathcal{W}_{\mathbf{p}} + \dim \mathcal{W'}_{\mathbf{p}^+}.
\]
\end{proof}

We now define the $L_\infty$-component map
\[
\widehat{\mathrm{ft}}_{i, \mathbf{p},\mathbf{x}}: \mathcal{C}_{\mathbf{p},\mathbf{x}, \mathrm{ft}_{i, \mathbf{p}}} \rightarrow \mathcal{C}^+_{\mathbf{p}^+, \mathbf{x}^+}
\]
for each $\mathbf{x} \in s^{-1}(0)$ and $\mathbf{x}^+ \in (s^+)^{-1}(0)$ with $\mathrm{ft}_{i, \mathbf{p}}(\mathbf{x}^+) = \mathbf{x}$. Here, we have
\[
\mathcal{C}_{\mathbf{p},\mathbf{x}, \mathrm{ft}_{\mathbf{p}}} = \mathcal{C}_{\mathbf{p},\mathbf{x}} = \bigwedge\nolimits^{-\bullet}\Gamma\left(E_{\mathbf{p}}^*|_{W_\mathbf{x}}\right) \oplus \Omega_{\text{aug}}^{\bullet+1}(\mathcal{F}_{\mathbf{x}})
\]
from the surjectivity of $\mathrm{ft}_{\mathbf{p}}$ and
\[
\mathcal{C}^+_{\mathbf{p}^+, \mathbf{x}^+} = \bigwedge\nolimits^{-\bullet}\Gamma\left(E_{\mathbf{p}^+}^*|_{W_{\mathbf{x}^+}}\right) \oplus \Omega_{\text{aug}}^{\bullet+1}(\mathcal{F}^+_{\mathbf{x}^+})
\]
with the $L_{\infty}[1]$-algebra structure $\l^+_{\mathbf{p}^+,\mathbf{x}^+} := \left\{l^+_{\mathbf{p}^+,\mathbf{x}^+,k}\right\}_{k \geq 1}.$

It is important to note that we can take the section $s_{\mathbf{p}^+}$ of the following special form:
\begin{equation}\label{sftsp}
s_{\mathbf{p}^+} = \mathrm{ft}_{\mathbf{p}}^* s_{\mathbf{p}},
\end{equation}
so that it depends only on $U_{\mathbf{p}}$. This choice can be justified by the fact that a pseudoholomorphic disk itself is determined independently of the marked points and that $ds_{\mathbf{p}^+}|_{U_{\mathbf{p}} \times \mathcal{W}_{\mathbf{p}} \times \{0\},\vec{0}}$ is an embedding by Lemma \ref{idks} (vii). Note that this embedding property implies 
\[
s_{\mathbf{p}^+}^{-1}(0)|_{U_{\mathbf{p}} \times \mathcal{W}_{\mathbf{p}}\times \mathcal{W}_{\mathbf{p}^+}} \subset s_{\mathbf{p}^+}^{-1}(0)|_{U_{\mathbf{p}} \times \mathcal{W}_{\mathbf{p}}\times \{0\}}
\]
for sufficiently small $\mathcal{W}'_{\mathbf{p}^+}$.

We then obtain the expression
\[
\mathcal{C}^+_{\mathbf{p}, \mathbf{x}^+} \simeq \bigwedge\nolimits^{-\bullet}\Gamma\left(\left(\pi^* E_{\mathbf{p}}^* \oplus \mathbb{R}^{c(\mathbf{p}^+)}\right)\Big|_{W_\mathbf{x}^+}\right) \oplus \Gamma_{\text{aug}}\left(\bigwedge\nolimits^{\bullet+1}\left(\mathrm{ft}_{\mathbf{p}}^* T\mathcal{F}_{\mathbf{x}}|_{{W}_{\mathbf{x}^+}} \oplus \mathcal{W}_{\mathbf{p}} \oplus\mathcal{W}'_{\mathbf{p}^+}\right)\right),
\]
where $\Gamma_{\mathrm{aug}}(\cdots)$ stands for the augmentation of the $L_{\infty}[1]$-algebra structure as in Proposition \ref{augomega}.
We now define
\[
\widehat{\mathrm{ft}}_{i, \mathbf{p},\mathbf{x},k}: \mathcal{C}_{\mathbf{p},\mathbf{x}}^{\otimes k} \rightarrow \mathcal{C}^+_{\mathbf{p}^+, \mathbf{x}^+}
\]
by 
\begin{equation}\nonumber
\widehat{\mathrm{ft}}_{i, \mathbf{p},\mathbf{x},k}\left((a_1, \xi_1), \cdots, (a_k, \xi_k)\right) := 
\begin{cases}
\left(\mathrm{ft}_{\mathbf{p}}^* a_1, 0\right) \oplus \left(\mathrm{ft}_{\mathbf{p}}^* \xi_1, 0, 0\right) & \text{if } k=1, \\ 
0 & \text{if } k \geq 2. 
\end{cases}
\end{equation}

\begin{prop}\label{qcqisom}
$\widehat{\mathrm{ft}}_{i, \mathbf{p}, \mathbf{x}} := \left\{\widehat{\mathrm{ft}}_{i, \mathbf{p}, \mathbf{x}, k}\right\}_{k \geq 1}$ is a quasi-isomorphic $L_{\infty}[1]$-morphism.
\end{prop}

\begin{proof}
We first show that $\widehat{\mathrm{ft}}_{i, \mathbf{p}, \mathbf{x}, 1}$ is a chain map; for $a \in \bigwedge^{-\bullet} \Gamma(E_{\mathbf{p},\mathbf{x}}^*|_{W_{\mathbf{x}}})$, $\xi \in \Omega_{\mathrm{aug}}^{\bullet+1}(\mathcal{F}_{\mathbf{x}})$, we have
\begin{equation}\nonumber
\begin{split}
l^+_{\mathbf{p}^+, \mathbf{x}^+,1} \circ \widehat{\mathrm{ft}}_{i, \mathbf{p}, \mathbf{x}, 1}(a, \xi) &= l_{\mathbf{p}^+, \mathbf{x}^+,1}^{+ \mathrm{K}}\left((\mathrm{ft}_{i, \mathbf{p}}^*a),0 )\right) \oplus l_{\mathbf{p}^+, \mathbf{x}^+, 1}^{+\mathrm{dR}}\left(\mathrm{ft}_{i, \mathbf{p}}^*\xi, 0 \right) \\
&{=} \iota_{s_{\mathbf{p}}^+|_{W_{\mathbf{x}^+}}}(\mathrm{ft}_{i,\mathbf{p}}^*a) \oplus d_{\mathcal{F}_{\mathbf{x}^+}^+}\left((\mathrm{ft}_{i, \mathbf{p}}^*\xi),0\right)\\ &\overset{(1)}{=} \left(\mathrm{ft}_{i, \mathbf{p}}^*(\iota_{s_{\mathbf{p}}}(a)),0 \right) \oplus \left(\mathrm{ft}_{i, \mathbf{p}}^*d_{\mathcal{F}_{\mathbf{x}}}(\xi), 0 \right) \\
&= \widehat{\mathrm{ft}}_{i, \mathbf{p}, \mathbf{x}, 1}\left(\iota_{s_{\mathbf{p}}}|_{W_{\mathbf{x}}}(a), d_{\mathcal{F}_{\mathbf{x}}}(\xi)\right) = \widehat{\mathrm{ft}}_{i, \mathbf{p}, \mathbf{x}, 1}\left(l_1(a, \xi)\right).
\end{split}
\end{equation}
Here the equality (1) holds for our choice of (\ref{sftsp}): $s_{\mathbf{p^+}} = \mathrm{ft}_{\mathbf{p}}^*s_{\mathbf{p}}$, so we have
\[
\left(\mathrm{ft}_{i, \mathbf{p}}^*(a)\right)(s_{\mathbf{p^+}}) = \mathrm{ft}_{i, \mathbf{p}}^*(a)\left(\mathrm{ft}_{i, \mathbf{p}}^*(s_{\mathbf{p}})\right) = \mathrm{ft}_{i, \mathbf{p}}^*\left(a(s_{\mathbf{p}})\right).
\]

Also, we have
\[
d_{\mathcal{F}_{\mathbf{x}^+}^+}\left(\mathrm{ft}_{i, \mathbf{p}}^*\xi\right) = \mathrm{ft}_{i, \mathbf{p}}^*d_{\mathcal{F}_{\mathbf{x}}}(\xi)
\]
from Lemma \ref{agssagaga} and the definition of $ft_{i, \mathbf{p}}.$

In general, we have for $k \geq 2$, $a_i \in \bigwedge^{-\bullet} \Gamma(E_{\mathbf{p}}^*|_{W_{\mathbf{x}}})$, $\xi_i \in \Omega^{\bullet+1}_{\mathrm{aug}}(\mathcal{F}_{\mathbf{x}}),$ and $1 \leq i \leq k$,
\begin{equation}\nonumber
\begin{split}
l^+_{\mathbf{p}^+,\mathbf{x}^+,k}&\left(\mathrm{ft}_{i, \mathbf{p},\mathbf{x},1}(a_1, \xi_1), \cdots, (a_k, \xi_k)\right) = l^+_{\mathbf{p}^+,\mathbf{x}^+,k}\left(\left(\pi^*(a_1), \mathrm{ft}_{i, \mathbf{p}}^*(\xi_1)\right), \cdots, \left(\pi^*(a_k), \mathrm{ft}_{i, \mathbf{p}}^*(\xi_k)\right)\right)\\
&= l^+_{\mathbf{p}^+,\mathbf{x}^+,_k}\big(\big(\pi^*(a_1), \cdots, \pi^*(a_k)\big) \oplus \big(\mathrm{ft}_{i, \mathbf{p}}^*(\xi_1), \cdots, \mathrm{ft}_{i, \mathbf{p}}^*(\xi_k)\big) \big)\\
&= \Pi \big[\cdots [P_{{\mathbf{x}}^+}, \mathrm{ft}_{i, \mathbf{p}}^*(\xi_1)], \cdots, \mathrm{ft}_{i, \mathbf{p}}^*(\xi_k) \big]\\
&\overset{(2)}{=} \Pi \big[\cdots \left[\mathrm{ft}_{i, \mathbf{p}}^*(P_{\mathbf{x}}), \mathrm{ft}_{i, \mathbf{p}}^*(\xi_1)\big], \cdots, \mathrm{ft}_{i, \mathbf{p}}^*(\xi_k)] \right]\\
&\quad \quad  \quad \quad  \quad \quad  \quad \quad \quad \quad  \quad \quad + \Pi \big[\cdots \left[P_{\mathbf{x}}^{\mathcal{W}_{\mathbf{p}}, \mathcal{W}'_{\mathbf{p}^+}}, \mathrm{ft}_{i, \mathbf{p}}^*(\xi_1)\right], \cdots, \mathrm{ft}_{i, \mathbf{p}}^*(\xi_k) \big]\\
&\overset{(3)}{=} \Pi \mathrm{ft}_{i, \mathbf{p}}^* \big( \big[\cdots [P_{\mathbf{x}}, \xi_1], \cdots, \xi_k] \big] \big) \overset{(4)}{=}  \mathrm{ft}_{i, \mathbf{p}}^* \Pi \big( \big[ \cdots [P_{\mathbf{x}}, \xi_1], \cdots, \xi_k \big] \big)\\
&= \mathrm{ft}_{i, \mathbf{p}}^* \big(l_{\mathbf{p},\mathbf{x},k}(\xi_1, \cdots, \xi_k)\big) = \mathrm{ft}_{i, \mathbf{p}}^* \big(l_{\mathbf{p},\mathbf{x},k}\big((a_1, \xi_1), \cdots, (a_k, \xi_k)\big)\big)\\
&= \widehat{\mathrm{ft}}_{i, \mathbf{p},\mathbf{x},1} \circ l_{\mathbf{p},\mathbf{x},k}\big((a_1, \xi_1), \cdots, (a_k, \xi_k)\big).
\end{split}
\end{equation}
Here, for the equality $(2)$, we use the decomposition of the Poisson structure
\[
P_{\mathbf{x}^+} = \mathrm{ft}_{\mathbf{p}}^*(P_{\mathbf{x}}) + P_{\mathbf{x}^+}^{\mathcal{W}_{\mathbf{p}}, \mathcal{W}'_{\mathbf{p}}},
\]
where $\mathrm{ft}_{\mathbf{p}}^+(P_{\mathbf{x}})$ (by abuse of notation) denotes the Poisson structure with respect to the presymplectic form $\mathrm{ft}_{i, \mathbf{p}}^*(\omega|_{W_{\mathbf{x}}})$ and $P_{\mathbf{x}^+}^{\mathcal{W}_{\mathbf{p}}, \mathcal{W}'_{\mathbf{p}^+}}$ the term that consists of the factors with the differentiations in the $\mathcal{W}_{\mathbf{p}}$- or $\mathcal{W}_{\mathbf{p}^+}'$-direction.
Now note that we have
\[
\left[P_{\mathbf{x}^+}^{\mathcal{W}_{\mathbf{p}}, \mathcal{W}'_{\mathbf{p}^+}}, \mathrm{ft}_{i, \mathbf{p}}^*(\xi_i)\right] = 0
\]
because $\mathrm{ft}_{\mathbf{p}}^*(\xi_i)$ is constant in the $\mathcal{W}_{\mathbf{p}}$- and $\mathcal{W}_{\mathbf{p}^+}'$-directions, hence we obtain the equality (3). (4) follows from a straightforward computation.

Finally, we show that $\widehat{\mathrm{ft}}_{i, \mathbf{p}, \mathbf{x}}$ is a quasi-isomorphism. Since $\widehat{\mathrm{ft}}_{i, \mathbf{p},\mathbf{x},1}$ is injective, it suffices to show that the quotient complex
\[
\frac{\mathcal{C}^+_{\mathbf{p}^+, \mathbf{x}^+}}{\widehat{\mathrm{ft}}_{i, \mathbf{p},\mathbf{x},1}(\mathcal{C}_{\mathbf{p},\mathbf{x}})}
\simeq \frac{\bigwedge^{-\bullet}\Gamma\left((\pi^* E_{\mathbf{p}}^* \oplus \mathbb{R}^c)|_{W_\mathbf{x}}\right)}{\widehat{\mathrm{ft}}_{i, \mathbf{p},\mathbf{x},1}^{\mathrm{K}}\left(\bigwedge^{-\bullet}\Gamma\left(E_{\mathbf{p}}^*|_{W_\mathbf{x}}\right)\right)} \oplus \frac{\Gamma_{\text{aug}}\left(\bigwedge^{\bullet+1}\left(\mathrm{ft}_{i,\mathbf{p}}^* T^*\mathcal{F}_{\mathbf{x}} \oplus \mathcal{W}_{\mathbf{p}} \oplus \mathcal{W}_{\mathbf{p}^+}'\right)\right)}{\widehat{\mathrm{ft}}_{i, \mathbf{p},\mathbf{x},1}^{\mathrm{dR}}\left(\Omega_{\text{aug}}^{\bullet+1}(\mathcal{F}_\mathbf{x})\right)}
\]
is acyclic. The de Rham part, being the quotient of acyclic chain complexes, is acyclic, where $\Gamma_{\mathrm{aug}}(\cdots)$ stands for the augmentation of the $L_{\infty}[1]$-algebra equipped with the $L_{\infty}[1]$-structure as in Proposition \ref{augomega}. For the Koszul part, the proof is essentially the same as of \cite[Lemma 2.32]{Kim1}, so we omit it.
\end{proof}

\begin{thm-defn}[Forgetful morphisms]\label{thqi}
The equivalence class
\[
\mathrm{Ft}_i:=\left[\left(\mathcal{U}_{\mathbf{p}^+}, \mathcal{U}_{\mathbf{p}}, {\mathrm{ft}_i}, \left\{{\mathrm{ft}_i}_{\mathbf{p}^+}\right\}, \left\{\widehat{\mathrm{ft}}_{i, \mathbf{p}^+, \mathbf{x}^+}\right\}\right)\right]
\]
defines a morphism of Kuranishi spaces
\begin{equation}\nonumber
{\mathrm{Ft}_i} : \left( \mathcal{M}_{k+1}(L, \beta), \left[\widehat{\mathcal{U}}^+\right] \right) \rightarrow \left( \mathcal{M}_{k}(L, \beta), \left[\widehat{\mathcal{U}}\right] \right)
\end{equation}
for each $i,$ which we call the $i$-th \textit{forgetful morphism} of the moduli space $\mathcal{M}_{k+1}(L, \beta).$
\end{thm-defn}

\begin{proof}
We show the compatibility with coordinate changes, 

First, we know that $\mathrm{ft}_{i}$ is continuous (cf. \cite[Definitoin 10.3]{FO}). We verify axioms (i) and (iii) of Definition \ref{morphismkur}.

(i) $\psi'_{\mathbf{p}} \circ \mathrm{ft}_{i, \mathbf{p}^+} = \mathrm{ft}_i \circ \psi_{\mathbf{p}^+}$ on $s^{-1}_{\mathbf{p}^+}(0)$ follows immediately from the definitions of $\mathrm{ft}_{i, \mathbf{p}^+}$ and $\mathrm{ft}_i.$

(ii) For $\mathbf{p}^+, \mathbf{q}^+$ with $\mathrm{Im} \psi_{\mathbf{p }^+} \cap \mathrm{Im} \psi_{\mathbf{q}^+}\neq \emptyset,$ the compatibility with respect to the base coordinate change, 
\begin{equation}\nonumber
\phi'_{\mathbf{p}\mathbf{q}} \circ \mathrm{ft}_{i, \mathbf{p}^+} = \mathrm{ft}_{i, \mathbf{q}^+} \circ \phi_{\mathbf{p}^+\mathbf{q}^+}
\end{equation}
follows immediately from the definitions of $\mathrm{ft}_{i, \mathbf{p}^+}$ and $\mathrm{ft}_{i, \mathbf{q}^+}.$ The $(\Gamma_{\mathbf{p}^+}, \Gamma_{\mathbf{p}})$-equivariance of $\mathrm{ft}_{i, \mathbf{p}}$ is an easy consequence of the group $\Gamma_{\mathbf{p}^+} = \Gamma_{\mathbf{p}}$ as in Lemma \ref{idks} (iii). 

(iii) For the $L_{\infty}[1]$-component, the diagram
\begin{equation}\label{octa4}
\begin{aligned}
\begin{tikzcd}
\mathcal{C}_{\mathbf{q}^+, \phi_{\mathbf{p^+q^+}}(\mathbf{x}^+)} \arrow{d}[swap]{\widehat{\varepsilon}_{\mathbf{q}^+, \phi_{\mathbf{p^+q^+}}(\mathbf{x}^+), \phi_{\mathbf{p^+q^+}}}, \simeq} & \mathcal{C}'_{\mathbf{q}, \mathrm{ft}_{i,\mathbf{q}^+} \circ \phi_{\mathbf{p^+q^+}}(\mathbf{x}^+), \mathrm{ft}_{i, \mathbf{q}^+}} \arrow{l}[swap]{\widehat{\mathrm{ft}}_{i, \mathbf{q}^+, \phi_{\mathbf{p^+q^+}}(\mathbf{x}^+)}, \simeq} \cdots & \quad \quad\quad \quad \quad \quad \quad \quad \quad \quad \quad \quad\\
\mathcal{C}_{\mathbf{q}^+, \phi_{\mathbf{p^+q^+}}(\mathbf{x}), \phi_{\mathbf{p^+q^+}}} \arrow{d}[swap]{\widehat{\phi}_{\mathbf{p^+q^+},\mathbf{x}}, \simeq} & {} & \quad \quad \quad \quad \quad \quad \quad \quad\quad \quad \quad \quad \\
\mathcal{C}_{\mathbf{p}^+, \mathbf{x}} & \mathcal{C}'_{\mathbf{p},\mathrm{ft}_{i,\mathbf{p}^+}(\mathbf{x}^+),\mathrm{ft}_{i,\mathbf{p}^+}} \arrow{l}[swap]{{\widehat{\mathrm{ft}}_{i, \mathbf{p}^+,\mathbf{x}}}, \simeq} \cdots & \quad \quad\quad \quad \quad \quad \quad \quad\quad \quad \quad \quad\\
\end{tikzcd}\\
\begin{tikzcd}
\quad \quad\quad \quad \quad \quad \quad \quad\quad \quad \quad \quad \cdots & \mathcal{C}'_{\mathbf{q},\mathrm{ft}_{i,\mathbf{q}^+} \circ \phi_{\mathbf{p^+q^+}}(\mathbf{x}^+)} = \mathcal{C}'_{\mathbf{q}, \phi'_{\mathbf{p}\mathbf{q}} \circ \mathrm{ft}_{i,\mathbf{p}^+}(\mathbf{x}^+)} \arrow{l}[swap]{\widehat{\varepsilon}_{\mathbf{q}, \mathrm{ft}_{i, \mathbf{q}^+}\circ \phi_{\mathbf{p^+q^+}}(\mathbf{x}^+), \mathrm{ft}_{i,\mathbf{q}^+}}, \simeq} \arrow{d}{\widehat{\varepsilon}_{\mathbf{q},\phi'_{\mathbf{p}\mathbf{q}}\circ \mathrm{ft}_{i, \mathbf{p}^+}(\mathrm{x}^+), \phi'_{\mathbf{p}\mathbf{q}}}, \simeq} &{}\\
\quad \quad\quad \quad \quad \quad \quad \quad\quad \quad \quad \quad & \mathcal{C}'_{\mathbf{q},\phi'_{\mathbf{p}\mathbf{q}}\circ \mathrm{ft}_{i,\mathbf{p}^+}(\mathbf{x}^+), \phi'_{\mathbf{p}\mathbf{q}}} \arrow{d}{\widehat{\phi}'_{\mathbf{p}\mathbf{q}}, \mathrm{ft}_{i,\mathbf{p}^+}(\mathbf{x}^+), \simeq} &{}\\
\quad \quad\quad \quad \quad \quad \quad \quad\quad \quad \quad \quad \cdots & \mathcal{C}'_{\mathbf{p},\mathrm{ft}_{i,\mathbf{p}^+}(\mathbf{x}^+)} \arrow{l}[swap]{\widehat{\varepsilon}_{\mathbf{p},\mathrm{ft}_{i,\mathbf{p}^+}(\mathbf{x}^+),\mathrm{ft}_{i,\mathbf{p}^+}}, \simeq} & {}\\
\end{tikzcd}
\end{aligned}
\end{equation}
commutes up to $L_{\infty}[1]$-homotopy. This diagram may look complicated; however, the homotopy commutativity follows immediately from the fact that both sides are quasi-isomorphisms by Proposition \ref{qcqisom} and Corollary \ref{coinle}. Then we can apply Corollary \ref{anhp} to obtain an $L_{\infty}[1]$-homotopy between them.
\end{proof}

Consider the forgetful morphism $\mathrm{Ft}_{i}$ that forgets the $i$-th marked point of each element in $\mathcal{M}_{k+1}(L, \beta)$ for $k \geq 1.$ For the 0-th evaluation morphism $\mathrm{Ev}_0^{(k+1)}$ and  $\mathrm{Ev}_0^{(k)}$ from  $\mathcal{M}_{k+1}(L, \beta)$ and  $\mathcal{M}_{k}(L, \beta)$ to $L,$ respectively, we have the following diagram.
\[
\begin{tikzcd}
\mathcal{M}_{k+1}(L, \beta) \arrow{rd}[swap]{\mathrm{Ev}_0^{(k+1)}} \arrow{rr}{\mathrm{Ft}_{i}} & {} & \mathcal{M}_{k}(L, \beta) \arrow{ld}{\mathrm{Ev}_0^{(k)}}\\
{} & L. & {}
\end{tikzcd}
\]

\begin{cor}\label{efe}
As morphisms of Kuranishi spaces, the equality
\begin{equation}\nonumber
\mathrm{Ev}^{(k)}_0 \circ \mathrm{Ft}_i = \mathrm{Ev}^{(k+1)}_0
\end{equation}
holds for each $1 \leq i \leq k.$
\end{cor}

\begin{proof}
Consider $\mathbf{p}^+ \in \mathcal{M}_{k+1}(L,\beta)$ and $\mathbf{p} \in \mathcal{M}_k(L, \beta)$ that satisfy $\mathrm{ft}_i(\mathbf{p}^+) = \mathbf{p}.$

From the definitions of evaluation and forgetful maps, one can easily show
\begin{equation}\nonumber
\begin{cases}
\mathrm{ev}^{(k)}_{0} \circ \mathrm{ft}_i  = \mathrm{ev}_0^{(k+1)},\\
\mathrm{ev}^{(k)}_{0, \mathbf{p}} \circ \mathrm{ft}_{i, \mathbf{p}^+} = \mathrm{ev}^{(k+1)}_{0, \mathbf{p}^+}.
\end{cases}
\end{equation}

Since $\widehat{\mathrm{ft}}_{i, \mathbf{p}^+, \mathrm{ev}^{(k)}_{0, \mathbf{p}}} \circ  \widehat{\mathrm{ev}}_{0, \mathbf{p}}^{{(k)}}$ and $\widehat{\mathrm{ev}}_{0, \mathbf{p}^+}^{{(k+1)}}$ involve no Koszul parts, being an $L_{\infty}[1]$-morphisms from an acyclic complex (i.e., an augmented foliation de Rham complex) to another, they must be quasi-isomorphic. Thus, we have
\begin{equation}
\widehat{\mathrm{ft}}_{i, \mathbf{p}^+, \mathrm{ev}^{(k)}_{0, \mathbf{p}}} \circ  \widehat{\mathrm{ev}}_{0, \mathbf{p}}^{{(k)}} = \widehat{\mathrm{ev}}_{0, \mathbf{p}^+}^{{(k+1)}},
\end{equation}
up to $L_{\infty}[1]$-homotopy by Corollary \ref{anhp}.

By (\ref{morcom}), we then have
\begin{equation}\nonumber
\begin{split}
\mathrm{Ev}^{(k)}_0 \circ \mathrm{Ft}_i &= \left[\left(\widehat{\mathcal{U}}, \widehat{\mathcal{U}}^{\mathrm{man}}, \mathrm{ev}^{(k)}_{0}, \mathrm{ev}^{(k)}_{0, \mathbf{p}}, \widehat{\mathrm{ev}}_{0,\mathbf{p}}^{{(k)}}\right)\right] \circ \left[\left(\widehat{\mathcal{U}}^+, \widehat{\mathcal{U}}, \mathrm{ft}_i, \mathrm{ft}_{i, \mathbf{p}^+}, \widehat{\mathrm{ft}}_{i, \mathbf{p}^+}\right)\right]\\
&\overset{*}{=} \left[\left(\widehat{\mathcal{U}}^+, \widehat{\mathcal{U}}^{\mathrm{man}}, \mathrm{ev}^{(k)}_{0} \circ \mathrm{ft}_i, \mathrm{ev}^{(k)}_{0, \mathbf{p}} \circ \mathrm{ft}_{i, \mathbf{p}^+}, \widehat{\mathrm{ft}}_{i, \mathbf{p}^+, \mathrm{ev}^{(k)}_{0, \mathbf{p}}} \circ  \widehat{\mathrm{ev}}_{0, \mathbf{p}}^{{(k)}} \right)\right]\\
&=  \left[\left(\widehat{\mathcal{U}}^+,  \widehat{\mathcal{U}}^{\mathrm{man}}, \mathrm{ev}_0^{(k+1)}, \mathrm{ev}^{(k+1)}_{0, \mathbf{p}^+}, \widehat{\mathrm{ev}}_{0, \mathbf{p}^+}^{{(k+1)}}\right)\right] = \mathrm{Ev}^{(k+1)}_0,
\end{split}
\end{equation}
where we do not need to consider extensions of pre-morphisms of (\ref{morcom}) other than themselves for the equality $*$; $\mathrm{ev}^{(k)}_{0,\mathbf{p}}$ and $\mathrm{ft}_{i,\mathbf{p}^+}$ are already surjective (cf. the discussion in the paragraph preceding Lemma \ref{levipzx}).
\end{proof}

\appendix

\section{$L_{\infty}[1]$-algebras and their homotopy theory}

In this section, we briefly introduce $L_{\infty}[1]$-algebras and their homotopy theories, following \cite{Kim2}. We first recall the notion of \textit{graded symmetric algebra} $SC$ of a vector space $C$ over a field $\mathbf{k},$
$$
S C := TC/ \langle v \otimes v' - (-1)^{|v|\cdot|v'|} v' \otimes v \rangle,
$$
with its degree $k$ component $S^k C := \{v \in S C \mid v \text{ is homogeneous of degree } k\}.$
We have a decomposition
$$
S C = \bigoplus\limits_{k=0}^{\infty} S^k C
$$
with the induced product $\odot$ on each component.
We denote by $\text{Sh}(i, k-i)$ the set of $(i, k-i)$-unshuffles, and the sign $\sgn(\tau)$ for $\tau \in \text{Sh}(i, k-i)$ is defined for homogeneous elements $a_1, \cdots, a_k \in C$, we write
$$
a_{\tau(1)} \odot \cdots \odot a_{\tau(k)} = \sgn(\tau) a_1 \odot \cdots \odot a_k.
$$

\begin{defn}
An \textit{$L_{\infty}[1]$-algebra} is a pair $\left(C, \{l_k\}\right)$ consisting of a vector space $C$ and a family of degree 1 linear maps
$$
l_k : S^k C \rightarrow C, \ k \geq 0,
$$
satisfying the relations
\begin{equation}\label{quadrel}
\sum\limits_{i = 0}^{k} \sum\limits_{\tau \in {\text{Sh}}(i, k-i)} {\sgn(\tau)} l_{k-i+1}\left(l_i(a_{\tau(1)}, \cdots, a_{\tau(i)}),a_{\tau(i+1)}, \cdots, a_{\tau(k)}\right) = 0.
\end{equation}
\end{defn}

\begin{defn}
Let $(C,\{l_k\})$ and $(C', \{l'_k\})$ be two $L_{\infty}[1]$-algebras. An $L_{\infty}[1]$\textit{-algebra morphism}, or simply $L_{\infty}[1]$\textit{-morphism}
\begin{equation}\label{lrel}
f : C \rightarrow C'
\end{equation}
is a family of degree 0 linear maps
$$
f_k : S^kC \rightarrow C', \ k \geq 0,
$$
satisfying the relations
\begin{equation}\label{frel}
\begin{split}
\sum\limits_{i = 0}^{k}& \sum\limits_{\tau \in {\text{Sh}}(i, k-i)} {\sgn(\tau)} f_{k-i+1}\left(l_i(a_{\tau(1)}, \cdots, a_{\tau(i)}),a_{\tau(i+1)}, \cdots, a_{\tau(k)}\right)\\
&= \sum\limits_{\substack{t, j_1, \cdots, j_t \geq 1,\\ j_1 + \cdots + j_t = k}} \sum\limits_{\tau \in S_k}  \frac{\sgn(\tau)}{t! j_1! \cdots j_t!} \  l'_{t}\bigl(f_{j_1}(a_{\tau(1)}, \cdots, a_{\tau(j_1)}), \cdots, \\
& \quad \quad \quad \quad \quad \quad \quad \quad \quad \quad \quad \quad f_{j_t}(a_{\tau(k -(j_1 + \cdots + j_{t-1}))}, \cdots, a_{\tau(k)})\bigr).
\end{split}
\end{equation}
Here, $S_k$ denotes the symmetric group of permutations of $k$ elements. 
\end{defn}

\begin{defn}
For two $L_{\infty}[1]$-morphisms
$$
f : C \rightarrow C', \ g: C' \rightarrow C'',
$$
we define their \textit{composition}
$$
g \circ f : C \rightarrow C''
$$
by a family of linear maps of degree 0 for $k \geq 0$
\begin{equation}\nonumber
\begin{split}
(g \circ f)_k := &\sum\limits_{i = 0}^{k} \sum\limits_{\tau \in S_k}  \frac{\sgn(\tau)}{t! j_1! \cdots j_t!} \ g_{t}\bigl(f_{j_1}(a_{\tau(1)}, \cdots, a_{\tau(j_1)}), \cdots,\\
&\quad \quad\quad \quad\quad \quad\quad \quad f_{j_t}(a_{\tau(k -(j_1 + \cdots + j_{t-1}))}, \cdots, a_{\tau(k)})\bigr).
\end{split}
\end{equation}
It is straightforward to verify that $\{(g \circ f)_k\}_{k \geq 0}$ satisfies the relation (\ref{frel}).
\end{defn}

\begin{defn}
We say an $L_{\infty}[1]$-algebra $\{l_k\}_{k \geq 0}$ is \textit{strict} if $l_0 = 0.$ Otherwise, we say it is \textit{curved.} We similarly define \textit{strict/curved} $L_{\infty}[1]$-morphisms.
\end{defn}

In the strict case, the relations (\ref{lrel}) and (\ref{frel}) coincide with the differential and the chain map relations, respectively. That is, they satisfy
\[
l_1 \left( l_1(a) \right) = 0; \ l'_1 \left(f_1 (a) \right) = f_1 \left( l_1(a) \right).
\]
\begin{defn}
We say that a strict $L_{\infty}[1]$-algebra $(C,\{l_k\})$ is \textit{acyclic} if its cohomology for each degree vanishes, that is, if 
\[
H^*(C) = \frac{\ker{l_1}}{\mathrm{Im}l_1} = 0.
\]
We say that a strict $L_{\infty}[1]$-morphism $\{f_k\}_{k \geq 1}$ between strict $L_{\infty}[1]$-algebras is a \textit{quasi-isomorphism} if $f_1$ is a quasi-isomorphic chain map.
\end{defn}

From now on and elsewhere in this paper, we always assume the strictness of $L_{\infty}[1]$-algebras and their morphisms.

To define homotopies between two $L_{\infty}[1]$-morphisms, we need to introduce the following notion:
\begin{defn}[Models of $\Delta^1 \times C$]\label{mdl1}
Let $C$ be an $L_{\infty}[1]$-algebra. We say an $L_{\infty}[1]$-algebra $\mathfrak{C}$ is a \textit{model} of $\Delta^1 \times C$ if there exist $L_{\infty}[1]$-morphisms
\[
\Eval_j : \mathfrak{C} \rightarrow C, \ \ j = 0,1\\
\]
and
\[
\Incl : {C} \rightarrow \mathfrak{C},
\]
with the following properties:
\begin{enumerate}[label = (\roman*)]
\item $\left(\Eval_j\right)_{k \geq 2} \equiv 0, \ j =0,1, \ \Incl_{k \geq 2} \equiv 0.$
\item $\Eval_j, \ j =0,1$ and  $\Incl$ are quasi-isomorphisms.
\item $(\Eval_j)_1 \circ \Incl = \mathrm{id}_C.$
\item $(\Eval_0)_1 \oplus (\Eval_{1})_1 : \mathfrak{C} \rightarrow {C} \oplus C$ is surjective.
\end{enumerate}
\end{defn}

Using the notion of models, we can define homotopies between $L_{\infty}[1]$-morphisms:

\begin{defn}[Homotopy]\label{defn:homotopy}
We say that two $L_{\infty}[1]$-morphisms $f_0, f_1 : (C, \{l_k\}) \rightarrow (C', \{l'_k\})$ are \textit{homotopic} if there exist a \textit{model of} $\Delta^1 \times C'$ denoted by $\mathfrak{C}'$ and an $L_{\infty}[1]$-morphism $h: C \rightarrow \mathfrak{C}'$ such that we have $f_j = \Eval_j \circ h, \ j = 0,1.$
\end{defn}

\begin{lem}\cite[Lemma 3.3]{Kim2}
Homotopies define an equivalence relation.
\end{lem}

We introduce  homotopy equivalences for $L_\infty[1]$-algebras.

\begin{defn}
An $L_{\infty}[1]$-morphism $f : C \rightarrow C'$ is a \textit{homotopy equivalence} if there exists another $L_{\infty}[1]$-morphism $g : C' \rightarrow C$ such that $g \circ f$ and  $f \circ g$ are homotopic to $\mathrm{id}_C$ and $\mathrm{id}_{C'},$ respectively.
\end{defn}

\begin{prop}\cite[Proposition 3.7]{Kim2}\label{prop:heer}
Homotopy equivalences define an equivalence relation.
\end{prop}

We show that a key theorem on quasi-isomorphisms, which shows the usefulness of our definition.
\begin{thm}\label{anhp}
Arbitrarily given two quasi-isomorphic $L_{\infty}[1]$-morphisms $f_0, f_1 : C \rightarrow C'$ are homotopic.
\end{thm}

\begin{proof}
This theorem is a special case of \cite[Corollary 4.6]{Kim2} for two morphisms.
\end{proof}

We close this section by stating an $L_{\infty}[1]$-algebra version of the Whitehead theorem over a field.

\begin{thm}[Whitehead theorem]\cite[Theorem 3.13]{Kim2}\label{wht}
Over a field and for strict $L_{\infty}[1]$-algebras, a quasi-isomorphism is a homotopy equivalence.
\end{thm}

\section{Presymplectic Neighborhoods and Local $L_{\infty}[1]$-Algebras}

In this section, we provide the details on the structures that we put on a Kuranishi chart. For more detail, we refer the reader to \cite{Kim1}.

Let a Kuranishi chart be given with the base $U.$ We equip $U$ with a closed two-form $\beta \in \Omega^2(U).$ 

We require that $U$ has a decomposition 
\begin{equation}\label{aabbbb}
U = \bigcup\limits_i \mathcal{S}_i,
\end{equation}
into (possibly non-connected) submanifolds,
\[
\mathcal{S}_i := \{x \in U \mid \text{rk} (\ker\beta_x) =i \}, \ 0 \leq i \leq \dim U
\]
together with their tubular neighborhoods:
\[
\begin{cases}
\iota_{i} : N_i \rightarrow  U, \text{ an open neighborhood of each }  \mathcal{S}_{i} \text{ in }U,\\
\pi_{i} : N_i \rightarrow \mathcal{S}_{i}, \text{ the associated projection}.
\end{cases}
\]

In case we have $\partial U \neq \emptyset$, we further require the decomposition (\ref{aabbbb}) restricts to $\partial U$, that is, we have a decomposition
\[
\partial U = \bigcup_i (\mathcal{S}_i \cap \partial U),
\]
where $\mathcal{S}_i \cap \partial U$ is a submanifold of $\partial U$ given by
\[
\mathcal{S}_i \cap \partial U = \{x \in \partial U \mid \text{rk} (\ker\beta_x) = i - 1 \}, \ 1 \leq i \leq \dim U,
\]
together with the corresponding tubular neighborhood in $\partial U$ for each $i$:
\[
\begin{cases}
\iota^{\partial}_{i} : N^{\partial}_i \hookrightarrow  \partial U, \text{ an open neighborhood of each }  \mathcal{S}_{i} \text{ in } \partial U,\\
\pi^{\partial}_{i} : N^{\partial}_i \rightarrow \mathcal{S}_{i} \cap \partial U, \text{ the associated projection}.
\end{cases}
\]
We denote the collar neighborhood $N^{\partial}_i$ of $\partial U$ in $U$ by
\[
\kappa^{\partial}_i : KN^{\partial}_i \rightarrow N^{\partial}_i.
\]

\subsubsection*{The presymplectic open neighborhood $W_x$} We associate an open contractible submanifold $W_x$ to each zero point $x \in s^{-1}(0).$

(i) $x \in U \setminus \partial U.$ For each zero point $x \in s^{-1}(0) \cap \mathcal{S}_i$ for some $i$, we choose $\overset{\circ}{W}_x \subset \mathcal{S}_i,$ an open ball containing $x$ in $\mathcal{S}_i.$ For the inclusion ${\iota_x} : \overset{\circ}{W}_x {\hookrightarrow} U,$ we denote $\beta|_{\overset{\circ}{W}_x} := \iota_x^*\beta.$ We have $d\left(\beta|_{\overset{\circ}{W}_x}\right) = d\iota_x^*\beta = \iota_x^* d\beta = 0$ and that $\beta|_{\overset{\circ}{W}_x}$ is of constant rank by construction. In other words, $\left(\overset{\circ}{W_x}, \beta|_{\overset{\circ}{W_x}}\right)$ is a presymplectic manifold.

(ii) $x \in \partial U.$  For each zero point $x \in s^{-1}(0) \cap \mathcal{S}_i \cap \partial U$ for some $i$, we choose $\overset{\circ}{W}_x \subset \mathcal{S}_i \cap \partial U,$ an open ball containing $x$ in $\mathcal{S}_i \cap \partial U.$ For the inclusion ${\overline{\iota}_x} : \overset{\circ}{W}_x {\hookrightarrow} \partial U \hookrightarrow U,$ we denote $\beta|_{\overset{\circ}{W}_x} := \overline{\iota}_x^*\beta.$ We have $d\left(\beta|_{\overset{\circ}{W}_x}\right) = d\overline{\iota}_x^*\beta = \overline{\iota}_x^* d\beta = 0$ and that $\beta|_{\overset{\circ}{W}_x}$ is of constant rank by construction. In other words, $\left(\overset{\circ}{W_x}, \beta|_{\overset{\circ}{W_x}}\right)$ is a presymplectic manifold.

Then we obtain another presymplectic manifold
\begin{equation}\label{adslfjkdh}
W_x = (W_x, \beta_{W_x}) := 
\begin{cases}
\left(\pi_i^{-1}(\overset{\circ}{W_x}), \pi^*_i(\beta|_{\overset{\circ}{W_x}})\right) &\text{ if } x \in U \setminus \partial U \\
\left({\left(\pi^{\partial}_i \circ \kappa^{\partial}_i \right)}^{-1}(\overset{\circ}{W_x}), (\pi_i^\partial \circ \kappa^{\partial}_i)^*(\beta|_{\overset{\circ}{W_x}})\right) &\text{ if } x \in \partial U 
\end{cases}
\end{equation}
and call it a \textit{local presymplectic neighborhood of} $x \in s^{-1}(0).$ We write 
\[
T\mathcal{F}_x := \ker \beta_{W_x}
\]
for the regular foliation (i.e., each leaf having the same dimension) determined by the kernel of $\beta_{W_x}.$

\subsubsection*{The $L_{\infty}[1]$-algebra $\mathcal{C}_x$}
At each zero point $x \in s^{-1}(0),$ we associate a \textit{local $L_{\infty}[1]$-algebra,}
\begin{equation}\nonumber 
\mathcal{C}_x := \overbrace{\bigwedge\nolimits^{-\bullet}\Gamma(E^*|_{W_x})}^{\text{Koszul}} \oplus \overbrace{\Omega^{\bullet + 1}_{\mathrm{aug}}(\mathcal{F}_x)}^{\text{de Rham}},
\end{equation}
which consists of the two parts: Koszul and de Rham.

The Koszul part, $\bigwedge\nolimits^{-\bullet}\Gamma(E^*|_{W_x})$ is the Koszul complex,
\[
0 \rightarrow \overbrace{\bigwedge\nolimits^{r}\Gamma(E^*|_{W_x})}^{\text{deg} = -r} \xrightarrow{\iota_{s|_{W_x}}} \cdots \xrightarrow{\iota_{s|_{W_x}}}  \overbrace{\Gamma(E^*|_{W_x})}^{\text{deg} = -1} \xrightarrow{\iota_{s|_{W_x}}}  \overbrace{C^{\infty}(W_x)}^{\text{deg} = 0} \rightarrow 0
\]
with the differential $l_1^{\mathrm{K}} := \iota_{s|_{W_x}},$ given by:
\[
l_1^{\mathrm{K}} : a_1 \wedge \cdots \wedge a_m \mapsto \sum\limits_{i=1}^{m} (-1)^{i+1}a_i(s|_{W_x}) \cdot a_1 \wedge \cdots \wedge \widehat{a_i} \wedge \cdots \wedge a_m,
\]
with all higher $l^{\mathrm{K}}_{k \geq 2}$ being set to zero.

The de Rham part, $\Omega^{\bullet + 1}_{\mathrm{aug}}(\mathcal{F}_x)$ is the augmented foliation de Rham complex degree shifted by 1 equipped with the $L_{\infty}[1]$-algebra structure $\{l^{\mathrm{dR}}_k\}_{k \geq 1}$ obtained from Proposition \ref{augomega}. 

In Lemma \ref{valinf} and (\ref{dasfda}), we equip $\Omega^{\bullet +1}(\mathcal{F}_x)$ with an $L_{\infty}[1]$-algebra structure $\{l^{\mathcal{F}}_k\}.$ Furthermore, we have:

\begin{prop}\label{augomega}
In the above situation, there exists an $L_{\infty}[1]$-algebra structure on the chain complex $\Omega^{\bullet +1}_{\mathrm{aug}}(\mathcal{F}_x)$ that extends $\{l^{\mathcal{F}}_k\}$ on $\Omega^{\bullet +1}(\mathcal{F}_x).$ Moreover, this $L_{\infty}[1]$-algebra has trivial cohomology.
\end{prop}

\begin{proof}
If $x \in U \setminus \partial U$, we refer the reader to \cite[Corollary 2.10, Proposition B.16]{Kim1}. If $x \in \partial U$, we only need to check that the Poincar\'{e} lemma holds for $W_x$ with $\partial W_x \neq \emptyset$. Consider the chain map
\[
\Omega^{\bullet+1}\left(W_x; \mathcal{F}_x\right) \rightarrow \Omega^{\bullet+1}\left(\mathrm{int}(W_x); \mathcal{F}_x\right)
\]
defined by restriction to the interior. Surjectivity and injectivity are easily verified, hence the Poincar\'{e} lemma for $\Omega^{\bullet+1}\left(\mathrm{int}(W_x); \mathcal{F}_x\right)$ implies that for $\Omega^{\bullet+1}\left(W_x; \mathcal{F}_x\right)$.
\end{proof}

The $L_{\infty}[1]$-structure on $\mathcal{C}_x$ is then given by
\[
l_k : \mathcal{C}_x^{\otimes k} \rightarrow \mathcal{C}_x; \ l_k := l^{\mathrm{K}}_k \oplus l^{\mathrm{dR}}_k,
\]
where the direct sum notation indicates that the operations on the two components are defined separately. It is immediate that the family $\{l_k\}_{k \geq 1}$ satisfies the $L_{\infty}[1]$-relation.

\begin{lem}\cite[Lemma 2.3]{Kim1}\label{lemdw}
For different choices of $\overset{\circ}{W}_x$, we obtain isomorphic de Rham $L_\infty[1]$-algebras.  
\end{lem}

For the definition of chart morphisms we need the notion of completed algebras:
\begin{defn}[Completed algebras]\label{ladef}
Given a smooth map $\phi: V \rightarrow U,$ We define the \textit{completion of} $\mathcal{C}_{x}$ by
\begin{equation}\nonumber 
\mathcal{C}_{x, \phi} := \left(\bigwedge\nolimits^{-\bullet}\Gamma(E^*|_{W_x})\right)_{\phi} \oplus \Omega^{\bullet + 1}_{\mathrm{aug}, \phi}(\mathcal{F}_x),
\end{equation}
where the Koszul part
\begin{equation}\label{kdzpt}
\left(\bigwedge\nolimits^{-\bullet}\Gamma(E^*|_{W_x})\right)_{\phi} :=  C^{\infty}_{\phi}(W_x) \otimes_{C^{\infty}(W_x)} \bigwedge\nolimits^{-\bullet}\Gamma(E^*|_{W_x}),
\end{equation}
where we consider the inverse limit
\begin{equation}\label{ivlmt}
C_{\phi}^{\infty}(W) := \lim_{\longleftarrow} C^{\infty}(W_x)/ I_{\phi}^{j} \cdot C^{\infty}(W_x),
\end{equation}
where we denote $I_{\phi} := \left\{ f \in C^{\infty}(W_x) \mid f|_{\text{Im}\phi} \equiv 0\right\}.$

Its $L_{\infty}[1]$-structure
\[
l^{\mathrm{K}, \phi}_k: {\left(\bigwedge\nolimits^{i}\Gamma(E^*|_{W_x})\right)_{\phi}}^{\otimes k} \rightarrow \left(\bigwedge\nolimits^{i-1}\Gamma(E^*|_{W_x})\right)_{\phi},\\
\]
for each $1 \leq i \leq r$ is defined as follows: For each $j\geq 1,$ $h \in C^{\infty}(W_x)^{(j)},$ and $a \in \bigwedge\nolimits^{-\bullet}\Gamma(E^*|_{W_x}),$ we set
\[
(j) : l^{\mathrm{K}, \phi}_1(h \otimes a) := [1]_{j-2} \otimes \iota_{s|_{W_x}}(\widetilde{h}_1a_1),
\]
where $\widetilde{h}$ is a choice of representative in $C^{\infty}(W_x)$ such that $h = \widetilde{h} + I_{\phi}^j,$ and we set $[1]_{j-2} := 0$ for $j \leq 2$ by definition. All higher $l_{\phi, k \geq 2}^{\mathrm{K}}$'s are set to zero, so the $L_{\infty}$-relation of $\{l^{\mathrm{K}}_{\phi, k}\}_{k \geq 1}$ holds trivially.

The de Rham part $\Omega^{\bullet + 1}_{\mathrm{aug}, \phi}(\mathcal{F}_x)$ is the completed foliation de Rham complex \textit{with augmentation,} given by
\[
\Omega^{\bullet + 1}_{\mathrm{aug}, \phi}(\mathcal{F}_x) := \overbrace{\Omega^{\bullet + 1}(\mathcal{F}_x)_{\phi}}^{\deg \geq -1} \oplus \overbrace{\left(C^{\infty}(W_x)_{\mathcal{F}_x}\right)_{\phi}}^{\deg = -2},
\]
where we denote
\[
\begin{cases}
\Omega^{\bullet + 1}(\mathcal{F}_x)_{\phi} := C_{\phi}^{\infty}(W_x) \otimes_{C^{\infty}(W_x)} \Omega^{\bullet + 1}(\mathcal{F}_x),\\
\big(C^{\infty}(W_x)_{\mathcal{F}_x}\big)_{\phi} : = \ker \big(l_1^{\mathrm{dR}} : \Omega^{-1}(\mathcal{F}_x)[1]_{\phi} \rightarrow \Omega^{0}(\mathcal{F}_x)[1]_{\phi}\big).
\end{cases}
\]
The de Rham part $L_{\infty}[1]$-structure $l_{\phi, k}^{\mathrm{dR}}$ is obtained by applying the following Proposition \ref{augomega}.

The formula can be given similarly to $l^{\mathrm{K}, \phi}_{\phi, k}$; however, it has to be rewritten as (\ref{dasfda}) in the setting of $V$-algebras in Appendix C, and so we omit its precise formulation here. For more details, we refer the reader to \cite[Definition B.19]{Kim1}.

In conclusion, $\mathcal{C}_{x,\phi}$ with $\left\{l_{\phi, k} := l^{\mathrm{K}}_{\phi, k} \oplus l^{\mathrm{dR}}_{\phi, k}\right\}$ is an $L_{\infty}[1]$-algebra with the $L_{\infty}[1]$-relation, which can be verified in a straightforward manner.
\end{defn}

Given a local algebra $\mathcal{C}_x,$ there exists a natural map to its completion:
For each $k \geq 1$, We define
\begin{equation}\label{varep}
\widehat{\varepsilon}_{\phi(x), \phi,k} : \mathcal{C}_x^{\otimes k} \rightarrow \mathcal{C}_{x, \phi}
\end{equation}
by
\begin{equation}\nonumber
\widehat{\varepsilon}_{\phi(x), \phi, k}\big((a_1, \xi_1), \cdots, (a_k,\xi_k)\big) := \begin{cases}
1 \otimes (a_1, \xi_1) = (1 \otimes a_1, 1 \otimes \xi_1) &\text{ if } k = 1,\\
0 &\text{ if } k \geq 2,
\end{cases}
\end{equation}
and consider the family $\widehat{\varepsilon}_{\phi(x), \phi} := \left\{\widehat{\varepsilon}_{\phi(x), \phi,k}\right\}_{k \geq 1}.$
\begin{lem}\cite[Lemma 2.8]{Kim1}\label{vecpf}
$\widehat{\varepsilon}_{\phi(x), \phi}$ is an $L_{\infty}[1]$-morphism.
\end{lem}

In the context of Definition \ref{ladef} regarding local $L_{\infty}[1]$-algebras, several special cases are worth discussing. First, we have a simple type of completions for surjective $\phi.$
\begin{lem}\cite[Lemma 2.12]{Kim1}\label{surjcp}
If $\phi$ is surjective, then we have
\[
\mathcal{C}_{\phi(x),\phi} \simeq \mathcal{C}_{\phi(x)}.
\]
\end{lem}

Second, we consider the completion for open subcharts: Let  $o : U \hookrightarrow U'.$ be an open inclusion and $\mathcal{U} := \mathcal{U}'|_U$ the open subchart on $U.$

\begin{lem}\cite[Lemma 2.13]{Kim1}\label{itavsteai}
In the above situation, there exists an $L_\infty[1]$-quasi-isomorphism:
\[
\widehat{o}_x : \mathcal{C}'_{o(x),o} \simeq \mathcal{C}_x.
\]
\end{lem}

\section{$L_{\infty}[1]$-structures arising from V-algebras}

In this section, we deal with examples of $L_{\infty}[1]$-algebras arising from presymplectic foliations. For this purpose, we introduce V-algebras and define their completions.

\subsubsection*{V-algebras}

We introduce V-algebras of \cite{Voronov1} and \cite{CS}.

\begin{defn}[V-algebras]\cite{Voronov1}\label{Voronov1}
A \textit{V-algebra} is defined by a triple $(\mathfrak{h}, \mathfrak{a}, \Pi)$ such that

\begin{enumerate}
\item[--] $\mathfrak{h}$ is a graded Lie algebra over a field $\mathbf{k}.$
\item[--] $\mathfrak{a}$ is an abelian subalgebra of $\mathfrak{h}.$
\item[--] $\Pi : \mathfrak{h} \rightarrow \mathfrak{a}$ is the associated projection.
\item[--] $\ker \Pi$ is a Lie subalgebra of $\mathfrak{h}.$
\end{enumerate}
Let $P$ be an Maurer-Cartan element in $\mathfrak{h},$ i.e., an element of degree 1 with $[P,P]=0.$ The triple $(\mathfrak{h}, \mathfrak{a}, \Pi)$ together with such a choice of $P$ determines a family of operators:
\begin{equation}\label{dlpk}
\begin{split}
&l^{P}_k : \mathfrak{a}^{\otimes k} \rightarrow \mathfrak{a},\\
&\begin{cases}
(x_1, \cdots x_k) &\mapsto \Pi[ \cdots [[P, {x_1}], {x_2}], \cdots, {x_k}], \quad \text{if }k \geq 1,\\
\quad \quad 1 &\mapsto \Pi P,  \quad \quad \quad \quad \quad \quad \quad \quad \quad \quad  \ \text{if } k = 0.
\end{cases}
\end{split}
\end{equation}
\end{defn}

Then we have:
\begin{lem}\cite[Lemma B.2]{Kim1}\label{valinf}
The family $\{l^P_k\}_{k \geq 0}$ forms a curved $L_{\infty}[1]$-algebra.
\end{lem}

Consider a smooth 1-parameter family of V-algebras 
\[
\CV(t) = \big(\mathfrak{h}(t), \frak{a}(t), \Pi(t)\big), \ t \in [0,1].
\]
with a family of Maurer-Cartan elements $P(t) \in h(t)^1.$ Then the smooth family $\{{\mathfrak{h}}(t)\}_{t \in [0,1]}$ determines a flow 
\[
\phi_t : {\mathfrak{h}}(0) \rightarrow {\mathfrak{h}}(t), \ t \in [0,1].
\]
with the generating vector field of $\phi_t$ by $m_t \in T{\mathfrak{h}}(t),$ so that they satisfy the differential equation $\frac{d\phi_t}{dt} = m_t \circ \phi_t.$ We assume that the family satisfies
\begin{equation}\label{pkkpt}
\phi_t\big(\ker\big(\Pi(0)\big)\big) \simeq \ker\big(\Pi(t)\big) \simeq \ker\big(\Pi(0)\big) \ \text{ for all } t \in [0,1].
\end{equation}

We have:
\begin{cor}\cite[Corollary B.8]{Kim1}\label{indliso}
For a smooth 1-parameter family of V-algebras with a Maurer-Cartan element, satisfying the condition (\ref{pkkpt}), there exists an induced $L_{\infty}[1]$-isomorphism
\[
U(t) : \mathfrak{a}(0) \xrightarrow{\simeq} \mathfrak{a}(t)
\]
for each $t \in [0,1].$
\end{cor}

\subsubsection*{Example from the Gotay's embedding} \cite{Gotay} proves that a presymplectic manifold can be embedded as a coisotropic submanifold in a symplectic manifold. The foliation cotangent bundle arising from a presymplectic structure, by this theorem, provides an interesting example, which was studied in \cite{OP} using physics-inspired methods. Indeed, we can reformulate their results using V-algebras.

Let $(W^n, \omega_W)$ be a presymplectic manifold. We consider the distribution 
\[
T\mathcal{F} := \ker \omega_W \subset TW.
\]
It then follows readily from the closedness of $\omega_W$ that $T\mathcal{F}$ is involutive, hence is integrable by the Frobenius theorem. 

We choose a splitting of $TW,$ that is, a vector bundle $G$ satisfying
\begin{equation}\label{wdecomp}
TW = T\mathcal{F} \oplus G.
\end{equation}
Let $(y_1, \cdots, y_k, q_1, \cdots q_{n-k})$ be a local coordinate of $x$ in $W,$ where $(q_1, \cdots q_{n-k})$ is the foliation coordinates, that is $y_i = c_i, \ i = 1, \cdots, k$ form the plaque for the foliation near $x$. In this coordinates, we have
\begin{equation}\label{txfgxq}
\begin{split}
T_x\mathcal{F} &= \text{span} \Bigg\{ \frac{\partial}{\partial q_1}, \cdots, \frac{\partial}{\partial q_{n-k}} \Bigg\}, \\
G_x &= \text{span} \Bigg\{ \frac{\partial}{\partial y_i} + \sum\limits_{\alpha =1}^{m} R^{\alpha}_i \frac{\partial}{\partial q_{\alpha}} \Bigg\}_{1 \leq i \leq k}
\end{split}
\end{equation}
for some functions ${R^{\alpha}_i}$'s in $y_i$'s and $q_{\alpha}$'s. Here, $ R^{\alpha}_i $ can be regarded as the Christoffel symbol for the 'connection' determined by the decomposition (\ref{wdecomp}).

\begin{exam}\label{gtyemb}
We present an example due to Gotay \cite{Gotay}: Any presymplectic manifold can be coisotropically embedded into a symplectic manifold. Let $T^*\mathcal{F} \to W$ be the foliation cotangent bundle, that is, the dual bundle to the foliation tangent bundle arising from the involutive distribution $T\mathcal{F} \subset TW$. His theorem is realized by the vector bundle $F := T^*\mathcal{F}$ equipped with the symplectic form 
\begin{equation}\label{otf}
\omega_{T*\mathcal{F}} := \pi^*\omega_W - d \theta,
\end{equation}
where $\theta$ is the canonical 1-form. It is easy to show that $\omega_{T^*\mathcal{F}}$ is nondegenerate, hence a symplectic form. Gotay's theorem says that on $T^*\mathcal{F}$ we have a coisotropic embedding
\begin{equation}\nonumber
\sigma : (W, \omega_{W}) \hookrightarrow (T^*\mathcal{F}, \omega_{T*\mathcal{F}}),
\end{equation}
so that $\sigma(W)$ coincides with the 0-section in $T^*\mathcal{F}.$

With respect to the symplectic structure from $\omega_{T^*\mathcal{F}}$, we obtain a Poisson structure $P \in \Gamma(T^*\mathcal{F}, \bigwedge\nolimits^2 TT^*\mathcal{F}),$ i.e., a bivector field  $P \in \Gamma(F, \bigwedge\nolimits^{2}TF)$ such that $[P,P]=0$ for the Nijenhuis-Schouten bracket $[ \ , \ ].$ Then, in the local coordinates, it is written as 
\begin{equation}\label{pico}
P = \frac{1}{2} \sum\limits_{i,j}\omega^{ij} e_i \wedge e_j + \sum\limits_{\alpha}\frac{\partial}{\partial q^{\alpha}} \wedge \frac{\partial}{\partial p_{\alpha}},
\end{equation}
where we denote 
\begin{equation}\label{eppy}
e_j := \frac{\partial}{\partial y_j} + \sum\limits_{\alpha} R_j^{\alpha} \frac{\partial}{\partial q^{\alpha}} - \sum\limits_{\beta, \nu} p_{\beta} \frac{\partial R^{\beta}_j}{\partial q^{\nu}} \frac{\partial}{\partial p_{\nu}},
\end{equation}
where $R^{\alpha}_j$ is from (\ref{txfgxq}). We refer the reader to Sections 7 through 9 in \cite{OP} for the detailed analysis.

For the zero section $\sigma \equiv 0$ of $T^*\mathcal{F}$, there exists a canonical decomposition.
\[
T_{(x,0)}T^*F = T_xW \oplus T^*_xF
\]
at $x \in W$ into the horizontal and the vertical components. Then we have 
\[
NW = \bigcup\limits_{x \in W} T^*_x\mathcal{F} = T^*\mathcal{F}.
\]
In this case, $\mathfrak{h}$ and $\mathfrak{a},$ and $\Pi$ in Lemma \ref{valinf} for the V-algebra are identified as follows:
\begin{equation}\label{dasfda}
\begin{split}
\mathfrak{h} &:= 
\lim\limits_{\overset{\longleftarrow}{m}}\frac{\Gamma(T^*\mathcal{F}, \bigwedge\nolimits^{\bullet +1}TT^*\mathcal{F})}{(I(W)|_{T^*\mathcal{F}})^m \cdot \Gamma(T^*\mathcal{F}, \bigwedge\nolimits^{\bullet +1}TT^*\mathcal{F})},\\
\mathfrak{a} &:= \Gamma\left(W; \bigwedge\nolimits^{\bullet + 1}NW \right) = \Gamma\left(W; \bigwedge\nolimits^{\bullet + 1}T^*\mathcal{F}\right) = \Omega^{\bullet+1}\left(\mathcal{F}\right),
\end{split}
\end{equation}
and the map $\Pi$ is the projection to the subspace (of $\mathfrak{h}$) generated by elements of the form $\frac{\partial}{\partial p_{i_1}} \wedge \cdots \wedge \frac{\partial}{\partial p_{i,l}}$ for $i_1, \cdots, i_l$ and $l \geq 1,$ followed by the evaluation at $p_i =0, \ \forall i.$ With a choice of the Poisson structure, we obtain an $L_{\infty}[1]$-algebra structure on $\mathfrak{a} = \Omega^{\bullet+1}(\mathcal{F}),$ that is, on the (degree shifted) foliation de Rham complex by Lemma \ref{valinf}. We write $\left\{l^{\mathcal{F}}_k\right\}_{k \geq 0}$ for the resulting $L_{\infty}[1]$-algebra.
\end{exam}

\begin{lem}\cite[Lemma B.12]{Kim1}\label{liiifdr} We have:
\begin{enumerate}[label=(\roman*)]\label{strlinf}
\item $\left\{l^{\mathcal{F}}_k\right\}$ is strict, i.e., $l^{\mathcal{F}}_0 =0.$ 
\item $l_1^{\mathcal{F}}$ coincides with the foliation de Rham differential $d_{\mathcal{F}}.$
\item For the zero presymplectic form, $i.e.$, the case when $T\mathcal{F} = TU,$ $l^{\mathcal{F}}_1$ is the ordinary de Rham differential with all the other $l^{\mathcal{F}}_k$ with $k \geq 2$ being 0.
\item For different choices of the splitting (\ref{wdecomp}), we have \textit{isomorphic} $L_{\infty}[1]$-algebras.
\item For different choices of the local coordinate system, we obtain \textit{isomorphic} $L_{\infty}[1]$-algebras.
\end{enumerate}
\end{lem}

We specialize to the particular V-algebra of Example \ref{gtyemb}, which leads to an $L_{\infty}[1]$-algebra structure on the foliation de Rham complex.

Recall the Proposition \ref{augomega} shows that there exists an $L_{\infty}[1]$-algebra that extends $\{l^{\mathcal{F}}_k\}_{k \geq 1}.$ We have its morphism analogue:

\begin{lem}\cite[Lemma B.20]{Kim1}\label{auglmo}
Given an $L_{\infty}[1]$-morphism,
\begin{equation}\label{htphi}
\widehat{\phi} : \Omega^{\bullet +1}(\mathcal{F}') \rightarrow \Omega^{\bullet +1}(\mathcal{F}),
\end{equation}
there exists an $L_{\infty}[1]$-algebra quasi-isomorphism (still denoted by $\widehat{\phi}$)
\[
\widehat{\phi} : \Omega^{\bullet +1}_{\mathrm{aug}}(\mathcal{F}') \rightarrow \Omega^{\bullet +1}_{\mathrm{aug}}(\mathcal{F})
\]
between the augmented de Rham complexes, extending (\ref{htphi}).
\end{lem}

\section{Equivalence relations}

In this section, we review the equivalence relations that are needed in the definition of $L_{\infty}$-Kuranishi atlases and spaces. More details are provided in \cite{Kim1}.

We define expansion of a chart as follows:
\begin{defn}[Expansion of a chart]\label{expvs}
Let $\mathcal{U} = (U, E, s, \Gamma, \psi)$ be a Kuranishi chart on $X$ as in Definition \ref{kurdef} and $V$ a finite dimensional vector space. From $\mathcal{U},$ we can construct another chart called an \textit{expansion of $\mathcal{U}$ by $V,$} 
\[
\mathcal{U} \times V := (U \times V, E \times V, s \times \mathrm{id}_V, \Gamma, \psi)
\] 
on $X$ consisting of:
\begin{itemize}
\item[--] $U \times V$ with the closed two-form $\pi^* \beta$, where $\pi: U \times V \to U$ denotes the projection to the $U$-component.
\item[--] $E \times V \to U \times V$ is the vector bundle obviously obtained from $E \to U.$
\item[--] $s \times \text{id}_V: U \times V \longrightarrow E \times V$, $(y,v) \mapsto \big(s(y), v\big)$ is the section.
\item[--] $\Gamma$ acts only on the $U$-component of $U \times V.$
\item[--] $\psi : {(s \times \text{id}_V)^{-1}(0) }/{\Gamma}\simeq {s^{-1}(0)}/{\Gamma} \xhookrightarrow{\psi} X$ is the homeomorphism that coincide with $\psi$ of $\mathcal{U}.$
\item[--] $W_{(x,0)} := W_x \times V$ is the open neighborhood near the zero point $(x,0).$
\item[--] $\mathcal{C}^V_{(x,0)} := \bigwedge^{-\bullet} \Gamma\big((\pi^* E \oplus V)^*\big|_{W_{(x,0)}}\big) \oplus \Gamma_{\mathrm{aug}}\big(\bigwedge^{\bullet+1}(\pi^*T\mathcal{F} \oplus V)^*\big|_{W_{(x,0)}}\big)$ is the local $L_\infty[1]$-algebra at $(x,0) \in (s \times \mathrm{id}_V)^{-1}(0).$
\item[--] $\left\{l_{(x,0),k}^V: {\mathcal{C}^V}_{(x,0)}^{\otimes k} \rightarrow \mathcal{C}_{(x,0)}^V\right\}_{k \geq 1}$ is the $L_{\infty}[1]$-operations on $\mathcal{C}^V_{(x,0)}$
given by $l_{(x,0), k}^{V, \mathrm{K}} \oplus l_{(x,0), k}^{V, \mathrm{dR}},$ where each component is given by:
\[
\begin{cases}
&\begin{split}
l_{(x,0), k}^{V, \mathrm{K}}&\big((a_1, w_1^*) , \cdots , (a_k, w_k^*)\big)\\ 
&\quad \quad := \begin{cases}
\iota_{s \times \mathrm{id}_V}(a_1,w^*_1) = \left\{\iota_s(a_1)(x), \iota_{v}w_1^*(x,v)\right\}_{(x,v) \in W_x \times V} &\text{ if } k =1,\\
0 &\text{ if } k \geq 2,
\end{cases}
\end{split}\\
&l^{V,\mathrm{dR}}_{(x,0),k} \big((\xi_1, \tau_1), \cdots , (\xi_k, \tau_k)\big) := \left(l^{\mathrm{dR}}_{x,k} \left(\xi_1, \cdots,\xi_k \right), l^V_{k}(\tau_1, \cdots, \tau_k)\right).
\end{cases}
\]
Here, we define
\[
 l^V_{k}(\tau_1, \cdots, \tau_k), := \begin{cases}
\sum_i\tau_1(v_i) \text{ for a fixed basis } \{v_i\} \text{ for } V &\text{ if } k = 1,\\
0  & \text{ if } k \geq 2
\end{cases}
\]
for $a_i \in \Gamma\left(\pi^*E^*\big|_{W_{(x,v)}}\right), \ \xi_i \in \pi^*\left(\Omega^{\bullet +1}\left(\mathcal{F}\big|_{W_{(x,v)}}\right)\right)$, and $w^*_i, \tau_i \in \Gamma\left(V^*|_{W_{(x,v)}}\right)$. It follows immediately that the family $\{l_{(x,0),k}^V\}$ forms an $L_\infty[1]$-algebra. Moreover, we have
\[
H^*\left(\Gamma\big(\bigwedge\nolimits^{\bullet+1}(\pi^*T\mathcal{F} \oplus V)^*\big|_{W_{(x,0)}}\big), l^{V, \mathrm{dR}}_{(x,0),1}\right) \simeq H^*\big(\Omega^{\bullet + 1}(\mathcal{F}_x), l^{\mathrm{dR}}_{x,1}\big),
\] 
which can be verified by observing the following: Since $l^{\mathrm{dR}}_{x,k}$ and $l^V_k$ are defined separately, we can apply the K\"{u}nneth theorem and use the fact that the cohomology of $l_1$ computes that of an affine space, which is trivial.
 
We then add the augmentation to the de Rham part to obtain
\[
\mathcal{C}^V_{(x,0)}:= \Gamma\left((\pi^* E \oplus V)^*\big|_{W_{(x,0)}}\right) \oplus \Gamma_{\mathrm{aug}}\left(\bigwedge\nolimits^{\bullet+1}(\pi^*T\mathcal{F} \oplus V)^*\big|_{W_{(x,0)}}\right),
\]
and equip it with an $L_{\infty}[1]$-structure using Lemma \ref{augomega}.

\end{itemize}
\end{defn}

\begin{defn}[Expanded atlases]
Given a Kuranishi atlas $\widehat{\mathcal{U}}$ and a nonnegative number $m,$ we define the \textit{expanded atlas} of $\widehat{\mathcal{U}}$ by
\[
\widehat{\mathcal{U}} \times \mathbb{R}^{{m}}
:= 
\Big(
\left\{\mathcal{U}_p \times \mathbb{R}^{{m}} \right\}_p,\,
\left\{ \left(U_{pq} \times \mathbb{R}^{{m}}, \phi_{pq}^{\mathbb{R}^{{m}}}, \left\{\widehat{\phi}_{pq,x}^{\mathbb{R}^{{m}}} \right\}\right)\right\}_{p,q}
\Big),
\]
where each component is given by:
\begin{itemize}
    \item[--] $\mathcal{U}_p \times \mathbb{R}^{{m}}$ is the expanded chart for each $p \in X$.

    \item[--] $U_{pq} \times \mathbb{R}^{{m}}$ is an open subset of $U_p \times \mathbb{R}^{{m}}$.
    
    \item[--] $\phi_{pq}^{\mathbb{R}^{{m}}} : U_{pq} \times \mathbb{R}^{{m}} \to U_q \times \mathbb{R}^{{m}}$ is the base coordinate change given by
    \[
    \phi_{pq}^{\mathbb{R}^{{m}}} := \phi_{pq} \times \mathrm{id}_{\mathbb{R}^{{m}}}.
    \]
    \item[--] $\widehat{\phi}_{pq,x}^{\mathbb{R}^{{m}}} : \mathcal{C}_{q,\,\widehat{\phi}_{pq}^{\mathbb{R}^{{m}}}(x,0),\phi_{pq}^{\mathbb{R}^{{m}}}}^{\mathbb{R}^{{m}}} 
    \to \mathcal{C}_{p,(x,0)},$ for each $x \in s_{p}^{-1}(0),$ is the $L_{\infty}[1]$-coordinate change given by the composition
    \[
    \mathcal{C}^{\mathbb{R}^{{m}}}_{q,\,\widehat{\phi}_{pq}^{\mathbb{R}^{{m}}}(x,0), \phi_{pq}^{\mathbb{R}^{{m}}}}
    \xrightarrow{(1)^{-1}, \simeq} 
    \mathcal{C}_{q,\,\phi_{pq}(x),\phi_{pq}}
    \xrightarrow{\widehat{\phi}_{pq,x}, \simeq}
    \mathcal{C}_{p,x} \xrightarrow{(2), \simeq} \mathcal{C}^{\mathbb{R}^{m}}_{p, (x,0)}.
    \]
Here, the $L_{\infty}[1]$-quasi-isomorphisms $(1)$ and $(2)$ are defined as in Example 2.23 and Lemma 2.23 of \cite{Kim1}.
\end{itemize}
\end{defn}

Any chart naturally restricts on an open subset of the base under the condition that the group action is closed:

\begin{defn}[Open subchart]
Let $\mathcal{U}= (U, E, s, \Gamma, \psi)$ be an $L_{\infty}$-Kuranishi chart of $X$ and $U_0 \subset U$ an open subset with $\Gamma \cdot U_0 \subset U_0.$ Then the restricted tuple 
\[
\mathcal{U}|_{U_0}= (U_0, E|_{U_{0}}, s|_{U_{0}}, \Gamma, \psi |_{(U_0 \cap s^{-1}(0))/\Gamma})
\]
canonically determines an $L_{\infty}$-Kuranishi chart called the \textit{open subchart} of $\mathcal{U}$ on $U_0.$
\end{defn}

\begin{notation}
Let $\left(X, \widehat{\mathcal{U}}\right)$ be a Kuranishi atlas. We write 
\[
\left(X, \widehat{\mathcal{U}}^0\right) < \left(X, \widehat{\mathcal{U}}\right), \text{ or simply } \widehat{\mathcal{U}}^0 <\widehat{\mathcal{U}}, 
\]
for its open subatlas $\left(X, \widehat{\mathcal{U}}^0\right)$ that consist of the open subcharts and the coordinate changes induced to them.
\end{notation}
With this notation, we define an equivalence relation between the atlases.

\begin{defn}[Equivalence of atlases]\label{eqats}
Let $\left(X, \widehat{\mathcal{U}}_1\right)$ and $\left(X, \widehat{\mathcal{U}}_2\right)$ be Kuranishi atlases. We say that they are \textit{equivalent} and write
\[
\left(X, \widehat{\mathcal{U}}_1\right) \sim \left(X, \widehat{\mathcal{U}}_2\right), \text{ or simply } \widehat{\mathcal{U}}_1 \sim \widehat{\mathcal{U}}_2
\]
if
\begin{equation}\label{urur}
\widehat{\mathcal{U}}^0_1 \times \mathbb{R}^{n_1} = \widehat{\mathcal{U}}^0_2 \times \mathbb{R}^{n_2}
\end{equation}
by which we mean that the following conditions hold:

\begin{enumerate}[label = (\roman*)]
    \item There exists a commutative diagram as follows
    \[
    \begin{tikzcd}[row sep=large, column sep=large]
    E^0_{1,p}|_{U_{1,p}^0 \times \mathbb{R}^{n_1}} \arrow[r, "\simeq"] 
    & E^0_{2,p}|_{U_{2,p}^0 \times \mathbb{R}^{n_2}}\\
    U_{1,p}^0 \times \mathbb{R}^{n_1} \arrow[r, "\simeq"] \arrow[u, "s_{1,p}^0 \times \mathrm{id}_{\mathbb{R}^{n_1}}"'] 
    & U_{2,p}^0 \times \mathbb{R}^{n_2} \arrow[u, "s_{2,p}^0 \times \mathrm{id}_{\mathbb{R}^{n_2}}"]\\
   (s_{1,p}^0 \times \mathrm{id}_{\mathbb{R}^{n_1}})^{-1}(0) \arrow{r}{(1), \simeq} \arrow[u, hook] 
& (s_{2,p}^0 \times \mathrm{id}_{\mathbb{R}^{n_2}})^{-1}(0). \arrow[u, hook]
    \end{tikzcd}
    \]

    \item There exists a group isomorphism $\Gamma^0_{1,p} \simeq \Gamma^0_{2,p}$.
    
    \item There exists a commutative diagram as follows
  \[
    \begin{tikzcd}[row sep=large, column sep=large]
     \frac{(s_{1,p}^0)^{-1}(0)}{\Gamma_{1,p}}
    \arrow[r, "\simeq"] \arrow[d, swap, "\simeq"] 
    & \frac{(s_{1,p}^0 \times \mathrm{id}_{\mathbb{R}^{n_1}})^{-1}(0)}{\Gamma_{1,p}} \arrow[r, hook, "\psi_{1,p}^0"] & X  \\
   \frac{(s_{2,p}^0)^{-1}(0)}{\Gamma_{2,p}} \arrow[r, "\simeq"] 
    &  \frac{(s_{2,p}^0 \times \mathrm{id}_{\mathbb{R}^{n_2}})^{-1}(0)}{\Gamma_{2,p}}. \arrow[ru, hook, "\psi_{2,p}^0"] &{}
    \end{tikzcd}
    \]
    \item For each pair of the zero points $x_1 \overset{(1), \simeq}{\leftrightarrow} x_2$, there exists an \textit{isomorphism}
    \[
    \mathcal{C}_{{1,p}, (x_1, 0)}^{0,\mathbb{R}^{n_1}} \xrightarrow{\simeq} \mathcal{C}_{{2,p}, (x_2, 0)}^{0,\mathbb{R}^{n_2}}.
    \]
    
    \item There exists a commutative diagram as follows
    \[
    \begin{tikzcd}[row sep=large, column sep=large]
   U_{1,p}^0 \times \mathbb{R}^{n_1} \arrow[r, "\simeq"] 
    & U_{2,p}^0 \times \mathbb{R}^{n_2}\\
    U_{1,pq}^0 \times \mathbb{R}^{n_1} 
    \arrow{r}{(2), \simeq} \arrow[u, hook] 
    & U_{2,pq}^0 \times \mathbb{R}^{n_2}. \arrow[u, hook]
    \end{tikzcd}
    \]
    
    \item We have $\phi_{1,pq}^{0, \mathbb{R}^{n_1}} = \phi_{2,pq}^{0, \mathbb{R}^{n_2}}$ modulo the diffeomorphism $(2), \simeq$ in (v).
\end{enumerate}

\end{defn}

We list some of the properties of the above-mentioned equivalence by the following lemma whose proof can be found in \cite[Lemma 3.8]{Kim1}.
\begin{lem}\cite[Lemma 3.8]{Kim1}
We have:
\begin{enumerate}[label=(\roman*)]
\item $\widehat{\mathcal{U}}^0 \sim \widehat{\mathcal{U}}$ for an open subatlas $\widehat{\mathcal{U}}^0 < \widehat{\mathcal{U}}.$
\item $\widehat{\mathcal{U}} \sim \widehat{\mathcal{U}} \times V$ for a finite dimensional vector space $V.$
\item $\sim$ is an equivalence relation.
\item $(\widehat{\mathcal{U}} \times \mathbb{R}^{n}) \times \mathbb{R}^{n'} = \widehat{\mathcal{U}} \times \mathbb{R}^{n + n'}$ for all $n, n' \geq 0.$
\end{enumerate}
\end{lem}

\begin{defn}[Equivalence of pre-morphisms]\label{defmoreq}
Without loss of generality, one can assume that $n_1 \geq n_2$ in (\ref{urur}). We say that two pre-morphisms are \textit{equivalent} and write
\begin{equation}\label{morrel}
\overline{F}_1 \sim \overline{F}_2
\end{equation}
if there exist extensions $\overline{F}^{n_1,n'_1}_1$ and $\overline{F}^{n_2,n'_2}_2$ as in (\ref{extfg1}), so that the following hold:

\begin{enumerate}[label = (\roman*)]
\item $f_1 = f_2,$
\item $\widetilde{f}_{1, p}|_{(s_{1,p}^{0})^{-1}(0) \times \{0\}} =  \widetilde{f}_{2, p}|_{(s_{2,p}^{0})^{-1}(0) \times \{0\}}$ (precise meaning provided below),
\item the following diagram commutes up to $L_{\infty}[1]$-homotopy
\begin{equation}\label{octa}
\begin{tikzcd}
\mathcal{C}^{'\mathbb{R}^{n_1'}}_{f(p),(\widetilde{f}_{1,p}(x),0)} \arrow{r}{=} \arrow{d}{\widehat{\pi}_{(\widetilde{f}_{1,p}(x),0)}} & \mathcal{C}^{'\mathbb{R}^{n_1'}}_{f(p),(\widetilde{f}_{1,p}(x),0), \widetilde{f}_{1,p}} \arrow{r}{\widetilde{\widehat{f}}_{1,p,x}} &\mathcal{C}^{\mathbb{R}^{n_1}}_{p, (x,0)} \arrow{d}{\widehat{\pi}_{(x,0)}}\\
\mathcal{C}'_{f(p),\widetilde{f}_{1,p}(x)} = \mathcal{C}'_{f(p),\widetilde{f}_{2,p}(x)} \arrow{d}{\widehat{\pi}^{-1}_{(\widetilde{f}_{2,p}(x),0)}} &{}& \mathcal{C}_{p, x} \arrow{d}{\widehat{\pi}^{-1}_{(x,0)}}\\
\mathcal{C}^{'\mathbb{R}^{n_2'}}_{f(p),(f_{2,p}(x),0)}  \arrow{r}{=}  & \mathcal{C}^{'\mathbb{R}^{n_2'}}_{f(p),(f_{2,p}(x),0), \widetilde{f}_{2,p}} \arrow{r}{\widetilde{\widehat{f}}_{2,p,x}} & \mathcal{C}^{\mathbb{R}^{n_2}}_{p, (x,0)}
\end{tikzcd}
\end{equation}
for each $x \in (s_{1,p}^{0})^{-1}(0) \times \{0\}.$ 
\end{enumerate}
Here ${(s_{i,p}^{0})^{-1}(0) \times \{0\}}$'s in the conditions (ii) and (iii) are to be understood as the same subset of both ${{U}}^0_1 \times \mathbb{R}^{n_1}$ and ${{U}}^0_2 \times \mathbb{R}^{n_2}$ modulo the identification from (\ref{urur}).

\end{defn}

\begin{lem}\cite[Lemma 3.13]{Kim1}
$\sim$ is an equivalence relation.
\end{lem}

\label{appen:forms}

\end{document}